\documentclass[11pt,a4paper,reqno]{amsart}

\usepackage[margin=1.1in]{geometry}
\usepackage{amsmath,amssymb,amsfonts,mathrsfs}
\usepackage{graphicx}
\usepackage{bm}
\usepackage{booktabs}
\usepackage{multirow}
\usepackage{color}
\usepackage{empheq}
\usepackage{enumitem}
\usepackage{subcaption}
\usepackage{algorithm}
\usepackage{algorithmic}
\usepackage{hyperref}   

\providecommand{\Div}{\operatorname{div}}          % Divergence
\providecommand*{\Dist}[2]{\operatorname{dist}({#1};{#2})}   % distance
\providecommand*{\Dist}[2]{\Dist{#1}{#2}}
\newcommand{\Bd}{{\boldsymbol{d}}}

\newcommand{\Bf}{{\boldsymbol{f}}}

\newcommand{\Bn}{{\boldsymbol{n}}}

\newcommand{\Br}{{\boldsymbol{r}}}

\newcommand{\Bu}{{\boldsymbol{u}}}
\newcommand{\Bv}{{\boldsymbol{v}}}
\newcommand{\Bw}{{\boldsymbol{w}}}
\newcommand{\Bx}{{\boldsymbol{x}}}

\newcommand{\Bz}{{\boldsymbol{z}}}

\newcommand{\BF}{{\boldsymbol{F}}}

\newcommand{\BH}{{\boldsymbol{H}}}

\newcommand{\BL}{{\boldsymbol{L}}}

\newcommand{\BR}{{\boldsymbol{R}}}

\newcommand{\BU}{{\boldsymbol{U}}}

\newcommand{\etabf}{\boldsymbol{\eta}}

\newcommand{\xibf}{\boldsymbol{\xi}}

\newcommand{\taubf}{\boldsymbol{\tau}}

\newcommand{\chibf}{\boldsymbol{\chi}}

\newcommand{\Psibf}{\boldsymbol{\Psi}}

\newcommand{\Phibf}{\boldsymbol{\Phi}}

\newcommand{\Ca}{{\mathcal A}}
\newcommand{\Cb}{{\mathcal B}}

\newcommand{\Cf}{{\mathcal F}}

\newcommand{\Ci}{{\mathcal I}}

\newcommand{\Co}{{\mathcal O}}

\newcommand{\Ct}{{\mathcal T}}

\newcommand{\bbB}{\mathbb{B}}

\newcommand{\bbD}{\mathbb{D}}

\newcommand{\bbI}{\mathbb{I}}

\newcommand{\bbL}{\mathbb{L}}

\newcommand{\bbP}{\mathbb{P}}

\newcommand{\bbR}{\mathbb{R}}

\providecommand*{\wh}[1]{\widehat{#1}}

\newcommand*{\Lp}[2][\defaultdomain]{L^{#2}({#1})}
\newcommand*{\Lpv}[2][\defaultdomain]{\BL^{#2}({#1})}

\newcommand*{\Ltwo}[1][\defaultdomain]{\Lp[#1]{2}}
\newcommand*{\Ltwov}[1][\defaultdomain]{\Lpv[#1]{2}}

\newcommand*{\Hmv}[2][\defaultdomain]{\BH^{#2}({#1})}

\newcommand*{\Honev}[1][\defaultdomain]{\Hmv[#1]{1}}

\newcommand{\D}{\mathrm{d}}
\newcommand{\Rd}{\mathrm{d}}

\newcommand{\be}{\begin{eqnarray}}
\newcommand{\ee}{\end{eqnarray}}

\newcommand{\ben}{\begin{eqnarray*}}
\newcommand{\een}{\end{eqnarray*}}

\allowdisplaybreaks
\newtheorem{theorem}{Theorem}[section]

\newtheorem{remark}[theorem]{Remark}
\theoremstyle{definition}

\newtheorem{corollary}[theorem]{Corollary}

\numberwithin{equation}{section}

\newcommand\bmx{{\bm x}}
\newcommand\bmu{{\bm u}}
\newcommand\bmv{{\bm v}}
\newcommand\bmw{{\bm w}}
\newcommand\bmn{{\bm n}}
\newcommand\bmf{{\bm f}}

\newcommand\bmchi{{\bm\chi}}
\newcommand\id{ {\mathbf {\mathrm {id} }} }

\def\Gh{\Gamma_h}
\def\Oh{\Omega_h}

\def\ng{\nabla_\Gamma}

\begin{document}

\title[]{An energy--stable ALE--FEM with compensated tangential velocity for 
Navier--Stokes free--boundary flows}

\author[]{Nuo Lei,\,\, Weiying Zheng,\,\, and\,\, Na Zhu}
\address{Nuo Lei: Department of Applied Mathematics, The Hong Kong Polytechnic University, Hong Kong; 
Hua Loo-Keng Center for Mathematical Sciences, Academy of Mathematics and Systems Science, Chinese Academy of Sciences, Beijing, China. 
{\rm Email address: {\tt nuo97.lei@polyu.edu.hk}}}
\address{Weiying Zheng: State Key Laboratory of Mathematical Sciences, Institute of Computational Mathematics and Scientific/Engineering Computing,
    Academy of Mathematics and Systems Science,
    Chinese Academy of Sciences;
    School of Mathematical Science,
    University of Chinese Academy of Sciences,
    Beijing, China.
{\rm Corresponding author. Email addresses: {\tt zwy@lsec.cc.ac.cn}}}
\address{Na Zhu: Academy of Mathematics and Systems Science, Chinese Academy of Sciences, Beijing, China.
{\rm Email address: {\tt zhuna@lsec.cc.ac.cn}}}



\maketitle

\begin{abstract}
We propose an energy-stable arbitrary Lagrangian–Eulerian (ALE) finite element formulation for incompressible Navier–Stokes flows with surface-tension-driven free boundaries. The key ingredient is a multiplier-free compensated tangential velocity formulation, which redistributes boundary nodes via the discrete tangential curvature residual. 
The relaxation parameter governing the artificial tangential velocity is scaled according to the viscous–capillary balance, yielding the natural choice \(\alpha= \gamma_0/ \mu \), 
{ where \(\gamma_0\) is the surface tension coefficient and \(\mu\) is the dynamic viscosity.}
We prove that the fully discrete scheme is well-posed, energy-dissipative, and temporally consistent with its semi-discrete counterpart. Numerical results confirm the robustness of the proposed method across a range of $\mu$ and $\gamma_0$ values. Furthermore, the method exhibits superior mesh quality preservation compared with the BGN method, particularly in the small-time-step regime.
\end{abstract}

\setlength\abovedisplayskip{4pt}
\setlength\belowdisplayskip{4pt}

\section{Introduction}

Free-boundary and moving-interface problems driven by capillarity arise in a broad range of scientific and engineering applications, including liquid-metal processing, droplet formation in inkjet printing, atomization, and fuel injection. In such problems, the evolution of the interface is strongly coupled to the bulk incompressible Navier--Stokes equations, while surface tension gives rise to a curvature-dependent interfacial force. 
This capillary force is associated with the variation of the interfacial energy, and the coupled system satisfies a fundamental energy-dissipation law in which the sum of the kinetic and interfacial energies decreases through viscous dissipation.

Among the various numerical approaches to free-boundary Navier--Stokes flows, the
arbitrary Lagrangian--Eulerian (ALE) framework is particularly attractive.
It tracks the evolving interface using a fitted mesh while allowing the bulk mesh velocity to be chosen independently of the fluid velocity. 
This flexibility provides additional freedom to control mesh deformation and maintain mesh quality throughout the evolution \cite{DONEA1982689,agnese2020,Fu2020,DuanLiALEJCP22}.

It is therefore of considerable interest to develop energy-diminishing numerical approximations for Navier--Stokes flows with surface tension.
A discrete energy-dissipation law provides an \emph{a priori} bound on the
numerical solution and prevents artificial energy production caused by the discretization. Such control is particularly important in long-time simulations, for strongly deformed interfaces, and in regimes where the solution has limited regularity. In these settings, energy estimates often form the basis for stability and convergence analyses. Energy-diminishing or structure-preserving finite element methods have been developed for two-phase Navier--Stokes flows
\cite{ZHAO2020109582,DuanLiALEJCP22,23PFEM,GARCKE2026114922},
fluid--structure interaction problems
\cite{GaoHuLisisc}, 
and free-boundary problems with moving contact lines
\cite{MR4586461,BAO2025108767,Hu2026_Droplet}.

% Moreover, maintaining high-quality mesh remains a longstanding challenge in ALE methods. The surface shape is determined only by the normal velocity, the tangential velocity is used to redistribute mesh points and improve mesh quality. 
Maintaining mesh quality is another central issue in ALE methods.
Since the free--boundary determines the evolving domain, the normal component of the mesh velocity governs the boundary motion, whereas its tangential component can be chosen to redistribute boundary nodes and thereby improve mesh quality.
A prominent approach to controlling tangential mesh motion is the
parametric finite element framework introduced by Barrett, Garcke, and N\"urnberg, hereafter referred to as the BGN method
\cite{15BGNESAIM,Barrett15,BGN2020HoNA}. A key feature of the BGN method is its implicitly defined tangential velocity, which promotes the equidistribution of mesh nodes.
The method has been widely applied to complex interfacial dynamics; see, for example,
\cite{Bao2021,Bao2021SINA,MR4586461,BAO2025108767,Bai-Li-MCOM2025,DuanLi24}. 
However, as the time-step size $\Delta t\rightarrow0$, the BGN method can be viewed as a purely spatial discretization that prescribes only the normal velocity, leaving the tangential velocity not well-controlled. Consequently, the induced mesh redistribution may deteriorate in practice. This motivates the development of continuous
tangential-velocity models that remain well posed independently of
$\Delta t$.

Motivated by the BGN scheme, Hu and Li \cite{HuLiNM2022} proposed a continuous model to determine the ALE velocity $\Bw$ of the surface
\begin{align}\label{eq-BGN MDR}
\Delta_\Gamma \bmw =  \kappa \bmn, \quad 
\bmw \cdot \bmn =  \bmu \cdot \bmn \quad \hbox{on}\;\;\Gamma(t),
\end{align}
where $\bmu$ denotes the underlying surface velocity, $\bmn$ is the unit normal on $\Gamma(t)$, and $\kappa$ is an unknown scalar function on $\Gamma(t)$. By \eqref{eq-BGN MDR}, the authors implicitly require that $\Delta_\Gamma \bmw$ is parallel to $\Bn$. In fact, the velocity characterized by \eqref{eq-BGN MDR} minimizes the surface deformation-rate functional
\begin{equation}
	\mathcal{E}( \bmw) = \frac{1}{2} \int_{\Gamma(t)} |\nabla_\Gamma\bmw|^2 \, \Rd s,
\end{equation}
subject to the constraint 
\begin{align*}
    \bmw\cdot \bmn=\bmu\cdot \bmn  \quad \hbox{on}\;\;\Gamma(t).
\end{align*}
This variational characterization provides a well-posed continuous principle for determining the tangential velocity and is known as the minimal deformation rate (MDR) formulation.

The original MDR formulation and related artificial tangential-velocity approaches were developed primarily to improve mesh redistribution while preserving the prescribed normal motion \cite{HuLiNM2022}. In a related direction, Duan and Li proposed a minimal-deformation (MD) formulation based on minimizing the deformation energy of the flow map $X(\cdot,t):\Gamma(0)\to\Gamma(t)$  \cite{DuanLi24}. Very recently, Gao, Li, and Tang proposed a ``relaxed'' minimal deformation (RMD) formulation, which generates a tangential motion using harmonic map heat flow \cite{Gao2026NM}.
These early MDR formulations provide effective mechanisms for maintaining mesh quality over a broad range of time-step sizes. Their principal objective, however, is the design of suitable tangential mesh velocities, and energy-stable time discretization is not an automatic consequence of the underlying continuous variational principles.

% More recent studies have explicitly addressed the compatibility between MDR-type mesh redistribution and discrete energy stability for geometric curvature flows. Energy-stable MDR schemes and dual formulations have been developed in \cite{Gao2026NM,GaoLiTang2026dual}. In the dual approach, an auxiliary multiplier is introduced at the continuous level. This multiplier vanishes for the exact continuous solution and therefore does not alter the underlying geometric evolution, while revealing an energy-dissipation structure that permits linearly implicit, energy-stable discretizations with MDR-type tangential mesh motion. Related convergence analyses for MDR-type is provided in \cite{MR5009273}.

More recent studies have sought to reconcile MDR-type mesh redistribution with discrete energy stability for geometric curvature flows. 
In \cite{Gao2026SISC}, an energy-stable MDR scheme was developed for mean curvature flow and surface diffusion. 
More recently, Gao, Li, and Tang \cite{GaoLiTang2026dual} proposed a dual formulation in which an auxiliary multiplier is introduced at the continuous level. 
The multiplier vanishes for the exact continuous solution and hence does not alter the underlying geometric evolution, while exposing an energy-dissipation structure that enables linearly implicit,
energy-stable discretizations with MDR-type tangential mesh motion. 

In this paper, we propose a new energy-stable ALE finite element method for a surface-tension-driven Navier--Stokes free--boundary problem. Unlike purely geometric curvature flows, the normal velocity of the surface is not prescribed directly by a geometric law; instead, it is determined by the coupled fluid velocity through the kinematic condition 
\begin{align}\label{eq-mesh-normal-condition}
	\bmw\cdot\bmn = \bmu\cdot\bmn \quad\text{on }\Gamma(t).
\end{align}
This condition determines only the normal component of the mesh velocity,
while the tangential component remains unspecified at the continuous level. Nevertheless, the choice of tangential mesh velocity is crucial at the discrete level because it strongly affects the
quality of the surface mesh and the energy stability of numerical solutions.

The preceding observation suggests that, in the small--time--step regime, an additional principle is needed to determine the tangential velocity and thereby remove the degeneracy. 
Motivated by this, we determine the tangential component using the discrete tangential curvature residual.
% with tangential curvature residual.
Let $\bmw_{\Gamma}$ denote the ALE surface velocity. Accordingly, we decompose it into its
normal and tangential components as
\begin{align}\label{eq-normal-mesh-strong}
\bmw_{\Gamma} = \bmw_{\bmn}+\bmw_{\taubf},\quad
\bmw_{\bmn} := (\bmu\cdot\bmn)\bmn,
\quad
\bmw_{\taubf}\cdot\bmn = 0
\quad\text{on }\Gamma(t).
\end{align}

To reveal a tangential redistribution mechanism, in this paper we propose to define the tangential ALE velocity $\bmw_{\taubf}\in \BU_{\taubf}(\Gamma(t))$ by solving the variational problem
\begin{equation}\label{eq-tangential-mesh-weak}
\left\langle\nabla_\Gamma\bmw_{\taubf},\nabla_\Gamma\bm\eta\right\rangle
= -\frac{\gamma_0}{\mu} \left\langle\nabla_\Gamma\mathrm{id},\nabla_\Gamma\bm\eta\right\rangle
-\varepsilon\left\langle(\mathbb{I}-\bmn\otimes\bmn)\bmu,\bm\eta \right\rangle \quad \forall\,\etabf \in \BU_{\taubf}(\Gamma(t)),
\end{equation}
where $\BU_{\taubf}(\Gamma(t)) := \big\{ \bm\eta\in [H^1(\Gamma(t))/\bbR]^d : \bm\eta\cdot\bmn=0\big\}$, $\gamma_0$ is the surface tension coefficient, and $\mu$ is the dynamic viscosity of the fluid. 
Moreover, $\varepsilon\ge 0$ is a relaxation parameter that is sufficiently small and independent of the physical parameters. The geometric identity
\[
-\Delta_\Gamma\mathrm{id}=H\bmn 
\]
implies that the first term on the right-hand side of \eqref{eq-tangential-mesh-weak} vanishes at the continuous level, namely,
\begin{equation}\label{eq:nabla_Gamma_id}
\left\langle \nabla_\Gamma \mathrm{id},
\nabla_\Gamma\bm\eta\right\rangle
=\left\langle H\bmn,\bm\eta\right\rangle=0
\quad \forall\,\etabf \in \BU_{\taubf}(\Gamma(t)).
\end{equation}
In the particular case $\varepsilon=0$, equation \eqref{eq-tangential-mesh-weak} yields a purely normal surface velocity, namely,
\[
\bmw_{\taubf} \equiv {\bf 0} \quad \hbox{on}\;\;\Gamma(t).
\]
However, at the discrete level, a purely normal surface velocity often leads to unstable numerical schemes and, eventually, to the degeneration of ALE meshes.

The discrete counterpart of \eqref{eq-tangential-mesh-weak} generally yields a nonzero tangential surface velocity $\bmw_{\taubf,h}$, even for 
$\varepsilon=0$. We prove that this tangential velocity ensures the energy stability of the numerical scheme and good mesh quality. In this sense, we call $\bmw_{\taubf,h}$ the compensated tangential velocity (CTV), equation \eqref{eq-tangential-mesh-weak} the CTV formulation, and the numerical scheme based on \eqref{eq-tangential-mesh-weak} the CTV scheme. In contrast to purely geometric strategies, the parameter $\gamma_0/\mu$ is related to the ratio of the time scales for fluid diffusion and surface capillarity.

In this work, we develop a spatially semi‑discrete scheme and a fully discrete scheme for the Navier–Stokes free‑boundary problem, based on the CTV formulation. The fully discrete scheme is unconditionally energy‑stable and reduces to the semi‑discrete scheme as $\Delta t\to 0$. Moreover, it overcomes the degradation of the BGN method in the small $\Delta t$ regime. Numerical experiments demonstrate that the CTV-ALE method preserves the expected energy decay while maintaining high-quality interface and bulk meshes. In particular, the tangential redistribution remains effective in the small $\Delta t$ regime and robust with respect to variations of $\mu$ and $\gamma_0$.

The rest of this paper is organized as follows. 
Section~\ref{sec:model} introduces the free-boundary problem and presents the continuous energy law. Section~\ref{sec:ALE-FEM} develops the semi-discrete scheme and the associated energy law. Section~\ref{sec:ALE-fully discrete} develops the fully discrete scheme and establishes the well-posedness and the energy-dissipation of the scheme. Section~\ref{sec:numerical} reports numerical experiments for free-boundary flows in two and three dimensions. Section~\ref{sec:conclusion} concludes the paper and discusses future work. Throughout this paper, vector-valued quantities will be denoted by boldface notations and matrix-valued quantities will be denoted by blackboard notations, such as $\Ltwov:=(\Ltwo)^d$ and $\bbD\in\bbR^{d\times d}$ for $d=2,3$.

\section{A free-boundary problem for the Navier-Stokes equations}
\label{sec:model}

In this section, we introduce the free-boundary problem of the Navier–Stokes equations and propose a CTV formulation for determining surface velocity. The energy-dissipation property of the model is also presented.

\subsection{The model problem}

We consider the Navier-Stokes flow with surface tension in a time-varying domain $\Omega(t) \subset \mathbb{R}^d$, 
\begin{subequations}\label{eq-model}
\begin{align}
\partial_t\bmu + (\bmu\cdot\nabla)\bmu-\nabla\cdot \bm\sigma  =\bmf,\quad 
\nabla \cdot \bmu  =0\quad \hbox{in}\;\; \Omega(t), \label{eq-model-1}\\
\bm\sigma \bmn = -\gamma_0 H \bmn \quad \hbox{on}\;\;\Gamma(t),\label{eq-model-3}\\
\bmu(0)  = \bmu^0 \quad\hbox{in}\;\;\Omega(0),\label{eq-model-4}
\end{align}
\end{subequations}
where $\Gamma(t)=\partial\Omega(t)$ is the free boundary of $\Omega(t)$, 
$\bm\sigma := 2\mu \mathbb {D}(\bmu)-p \mathbb I$ is the stress tensor, $p$ stands for the hydrodynamic pressure, and $\mathbb {D}(\bmu) = \frac{1}{2}\left(\nabla \bmu + (\nabla \bmu)^{T}\right)$ denotes the deformation matrix. 
The external force $\bmf$ is assumed in $\Ltwov[\Omega(t)]$. In \eqref{eq-model-3},  $H$ is the mean curvature of the free surface, $\bmn$ is the unit outward normal, and $\gamma_0 > 0$ is the surface tension coefficient. 	
In particular, let $\id$ denote the identity function on $\Gamma(t)\subset \mathbb{R}^d$. It holds that \cite{Deckelnick_Dziuk_Elliott_2005, Barrett15,15BGNESAIM, BARRETT2013511},
\begin{equation} \label{eq-mesh-v-2}
-\Delta_\Gamma \id  = H\bmn  \quad \hbox{on}\;\;\Gamma(t), 
\end{equation}
where $\Delta_{\Gamma}$ and $\ng$ are the Laplace–Beltrami and surface gradient  operators on $\Gamma(t)$, respectively.

\subsection{The ALE formulation}

The ALE formulation relies on an ALE velocity $\bmw$ which drives the motion of $\Omega(t)\cup\Gamma(t)$. Usually, the surface velocity $\bmw_{\Gamma}:=\bmw |_{\Gamma(t)}$ is determined first and then extended into the bulk domain. 
Here we employ the minimal-deformation-rate extension of $\bmw_{\Gamma}$ \cite{Hu2026_Droplet} and let $\bmw$ be the solution to the problem
\begin{align}\label{eq:bulk-MDR}
-\nabla\cdot\mathbb D(\bmw)=0\quad \hbox{in}\;\;\Omega(t), \quad 
\bmw =\bmw_{\Gamma} \quad \hbox{on}\;\;\Gamma(t).
\end{align}

Write $\Omega^0=\Omega(0)$, $\Gamma^0=\Gamma(0)$, and let
$\Phibf:\Omega^0\times[0,T]\to\mathbb R^d$ be the ALE map generated by the bulk velocity $\bmw$, satisfying
\[
    \partial_t\Phibf(\bm\xi,t)
    =
    \bmw(\Phibf(\bm\xi,t),t),
    \quad
    \Phibf(\bm\xi,0)=\bm\xi.
\]
Then $\Omega(t)=\Phibf(\Omega^0,t)$ and $\Gamma(t)=\Phibf(\Gamma^0,t)$.
Associated with $\Phibf$ and $\bmw$, the ALE time derivative is defined by
\[
\partial_t^\bullet\bmu:=\partial_t\bmu+\bmw\cdot\nabla\bmu.
\]
Then the ALE form of \eqref{eq-model} is given by
\begin{subequations}\label{ALE-model}
\begin{align}
\partial_t^\bullet\bmu + (\bmu-\bmw)\cdot\nabla\bmu-\nabla\cdot \bm\sigma  
    =\bmf,  \quad \nabla \cdot \bmu  =0
\quad \hbox{in}\;\; \Omega(t), \label{ALE-model-1}\\
-\nabla\cdot\mathbb D(\bmw)=0\quad \hbox{in}\;\;\Omega(t),  \label{ALE-model-2}\\
\bm\sigma \bmn = -\gamma_0 H \bmn,\quad 
-\Delta_\Gamma\mathrm{id}=H\bmn,\quad 
    \bmw =\bmw_{\Gamma} \quad \hbox{on}\;\;\Gamma(t),\label{ALE-model-3}\\
\bmu(0)  = \bmu^0 \quad\hbox{in}\;\;\Omega(0).\label{ALE-model-4}
\end{align}
\end{subequations}
To close the model, we also need equations for defining the surface velocity $\bmw_{\Gamma}$. This will be accomplished by the CTV formulation in the next subsection.

\subsection{The CTV formulation}

We now derive a weak formulation of \eqref{ALE-model} combined with the CTV formulation for defining $\bmw_{\Gamma}$. For a domain $D\subset \bbR^d$ and a surface $\Sigma$, the inner products on $\Ltwo[D]$ and $\Ltwo[\Sigma]$ are, respectively, denoted by
\[
(u, v)_D := \int_D uv,\quad 
\langle u,v \rangle_{\Sigma} := \int_{\Sigma} uv.
\]
When $D=\Omega(t)$ and $\Sigma=\Gamma(t)$, the subscripts of the inner products are omitted for convenience.
Moreover, the tangential space of $\BH^1(\Gamma(t))$ modulo $\bbR^d$ is defined by
\begin{align} \label{eq:V-Gamma}
\BU_{\taubf}(\Gamma(t)) := \big\{\bm\eta\in [H^1(\Gamma(t))/\bbR]^d: \bm\eta\cdot\bmn=0\big\}.    
\end{align}

As introduced in \eqref{eq-normal-mesh-strong}--\eqref{eq-tangential-mesh-weak}, the surface velocity $\bmw_{\Gamma}$ is first decomposed into $\bmw_{\Gamma}= \bmw_{\bmn} +\bmw_{\bm\tau}$, where the normal velocity of $\Gamma(t)$ is given by 
% $\bmw_{\bmn}:=(\bmu\cdot\bmn)\bmu$
$\bmw_{\bmn}:=(\bmu\cdot\bmn)\bmn$ 
and the tangential velocity $\bmw_{\bm\tau}\in \BU_{\taubf}(\Gamma(t))$ solves the elliptic problem on the surface
\begin{align} \label{eq:CTV}
\mathscr{A}(\Gamma;\bmw_{\taubf},\bm\eta)
= -\alpha\mathscr{A}(\Gamma;\mathrm{id},\bm\eta)
-\varepsilon \left\langle \mathbb{P}\bmu, \bm\eta  \right\rangle
\quad \forall\, \bm\eta \in \BU_{\taubf}(\Gamma(t)),
\end{align}
where $\alpha:=\gamma_0/\mu$, $\mathbb{P} := \bbI-\bmn\otimes\bmn$, $\bbI$ is the identity matrix, and 
\[
\mathscr{A}(\Gamma;\bm\xi,\bm\eta):=\langle \nabla_\Gamma\bm\xi, \nabla_\Gamma \bm\eta \rangle_{\Gamma}.
\]
Moreover, $\varepsilon\ge 0$ is a relaxation parameter which is small enough and independent of the physical parameters. Actually, $\varepsilon=0$ is allowed in our algorithm and theoretical analysis.

Multiplying both sides of \eqref{eq-mesh-v-2} by $\bm\eta\in \BU_{\taubf}(\Gamma(t))$ and integrating over $\Gamma(t)$, integration by parts yields $\mathscr{A}(\Gamma;\mathrm{id},\bm\eta)=0$. Hence, this term does not contribute to the tangential velocity of $\Gamma(t)$ at the continuous level. However, the corresponding term is nonzero at the discrete level and plays a crucial role in the stability of numerical solutions. Furthermore, setting $\varepsilon=0$ leads to $\bmw_{{\taubf}}=0$, corresponding to purely normal motion of the boundary.

Multiplying both sides of \eqref{eq-mesh-v-2} by $\xi\bmn$ and integrating the results by parts over $\Gamma(t)$, we obtain the weak form of \eqref{eq-mesh-v-2}: 
\begin{equation*}
\langle H,\xi\rangle = \mathscr{A}(\Gamma;\mathrm{id},\xi\bmn)
\quad \forall\, \xi\in H^1(\Gamma(t)).    
\end{equation*}
Since $\nabla\cdot\bmu=0$ in $\Omega(t)$ and $(\bmu-\bmw)\cdot\bmn=0$ on $\Gamma(t)$, integration by parts shows
\begin{align*}
((\Bu-\Bw)\cdot\nabla\bmu,\bmv)
=\frac12[((\Bu-\Bw)\cdot\nabla\bmu,\bmv)
-((\Bu-\Bw)\cdot\nabla\bmv,\bmu)
+(\bmu,\bmv\nabla\cdot\bmw)].
\end{align*}
Multiplying \eqref{ALE-model-1} by $\bmv\in \Honev[\Omega(t)]$ and integrating the stress term by parts, we obtain
\[
(\partial_t^\bullet\bmu,\bmv) 
+\frac{1}{2}(\bmu,\bmv\nabla\cdot\Bw)
+\Co(\Omega;\bmu-\bmw,\bmu,\bmv)
+ \mathcal{A}(\Omega;\bmu,\bmv) - \mathcal{B}(\Omega;p,\bmv)
+\gamma_0\langle H\bmn,\bmv\rangle = (\bmf,\bmv),
\]
where the bilinear and trilinear forms are defined by
\begin{align*}
\mathcal{A}(\Omega;\bmu,\bmv) :=\,& 
2\mu(\mathbb D(\bmu),\mathbb D(\bmv))_\Omega,\\
\mathcal{B}(\Omega;q,\bmv) :=\,& (q,\nabla\cdot\bmv)_\Omega, \\
\Co(\Omega;\Bz,\bmu,\bmv) :=\,& 
\frac12[(\bm z\cdot\nabla\bmu,\bmv)_\Omega
-(\Bz\cdot\nabla\bmv,\bmu)_\Omega].
\end{align*}
The weak form of \eqref{ALE-model-2} is to find $\bmw\in\Honev[\Omega(t)]$ such that $\bmw=\bmw_{\Gamma}$ on $\Gamma(t)$ and 
\[
\left(\mathbb D(\bmw),\mathbb D(\bm\chi)\right)_\Omega =0
\quad \forall \bm\chi\in \BH_0^1(\Omega(t)).
\]

Finally, collecting all variational equations above, we obtain the weak formulation of the free-boundary flow problem: find
\[
(\bmu,p,H,\bmw_{\taubf},\bmw)\in \Honev[\Omega(t)]\times L^2(\Omega(t))
   \times L^2(\Gamma(t)) \times \BU_{\taubf}(\Gamma(t))\times \BH^1(\Omega(t))
\]
such that $\bmw=\bmw_{\bm\tau} + (\bmu\cdot\Bn)\Bn$ on $\Gamma(t)$, $\bmu(0) = \bmu^0$ in $\Omega^0$, and
\begin{subequations}\label{eq:weak}
\begin{align} 
&(\partial_t^{\bullet} \bmu, \bmv)
+\frac{1}{2}(\bmu,\bmv\nabla\cdot\Bw)
+ \mathcal{O}(\Omega;\bmu-\bmw,\bmu,\bmv) \notag \\
&\hspace{12.5mm} +\mathcal{A}(\Omega;\bmu,\bmv)    
    -\mathcal{B}(\Omega;p,\bmv) 
    +\gamma_0\left\langle H, \bmv\cdot\bmn\right\rangle 
    =(\bmf,\bmv) \quad \forall\,\Bv\in\Honev[\Omega(t)],
    \label{eq:weak-u}\\
&\mathcal{B}(\Omega;q,\bmu) =0 \quad 
    \forall\,q\in\Ltwo[\Omega(t)], \label{eq:weak-p}\\
&\mathscr{A}(\Gamma;\bmw_{\taubf},\bm\eta) = 
    -\alpha\mathscr{A}(\Gamma;\mathrm{id},\bm\eta)
    -\varepsilon\langle \mathbb{P}\bmu,\bm\eta\rangle
    \quad \forall\,\bm\eta\in \BU_{\taubf}(\Gamma(t)), 
        \label{eq:weak-wtau} \\ 
&\langle H,\xi\rangle = \mathscr{A}(\Gamma;\mathrm{id},\xi\bmn)
    \quad \forall\,\xi\in H^1(\Gamma(t)),
    \label{eq:weak-H}\\
&(\mathbb D(\bmw),\mathbb D(\bm\chi)) = 0\quad
    \forall\,\bm\chi\in \BH_0^1(\Omega(t)).
    \label{eq:weak-w} 
\end{align}
\end{subequations}
The system is coupled with the equation for the flow map 
\begin{align} \label{eq-weak-form-flow-map}
\partial_t\Phibf=\bmw\circ\Phibf, \quad \Phibf(0) = \mathrm{id}
\quad\hbox{on}\;\;\Omega^0\cup\Gamma^0.
\end{align}

\subsection{Total energy-dissipation}
\label{sec:continuous-energy}

To end this section, we present the energy stability of problem \eqref{eq:weak}. 

\begin{theorem}\label{thm:continuous-energy}
Suppose $\Gamma(t)$ is a non-degenerate, closed, and smooth surface and $(\bmu,p,H,\bmw_{\taubf},\bmw)$ is a sufficiently smooth solution to problem
\eqref{eq:weak}. Define $Q(t):=\{(\Bx,s): \Bx\in\Omega(s),\; s\in (0,t)\}$. There exists a constant $\varepsilon_0>0$ depending only on the geometry of $Q(t)$ such that, for any $\varepsilon\in [0,\varepsilon_0)$,
\begin{align*}
E(t) + \mu \int_0^t \Big[\|\bbD(\Bu)\|^2_{\bbL^2(\Omega(s))} 
+\frac12 |\bmw_{\taubf}|_{\BH^1(\Gamma(s))}^2\Big]\D s
\le  e^{ 2(1+\mu)t} \big(E(0) + \|\Bf\|_{\BL^2(Q(t))}^2\big),
\end{align*}
where $E(t) := \frac12\|\bmu\|^2_{\BL^2(\Omega(t))}
+\gamma_0|\Gamma(t)|$ and $|\Gamma(t)|$ denotes the Hausdorff measure of $\Gamma(t)$. In particular, when $\bmf=\bm0$ and $\varepsilon=0$, the total energy is dissipative: $E’ \le 0$. 
\end{theorem}
\begin{proof}
Take $\Bv=\Bu$ in \eqref{eq:weak-u}, $\etabf=\mu\bmw_{\taubf}$ in \eqref{eq:weak-wtau}, and $\xi=\gamma_0\bmw\cdot\Bn$ in \eqref{eq:weak-H} and sum these equations. Using the facts that $\Div\Bu=0$ in $\Omega(t)$ and $\Bu\cdot\Bn=\bmw\cdot\Bn$ on $\Gamma(t)$, we have
\begin{equation}\label{eq:energy0}
\begin{aligned} 
\frac{1}{2}\int_{\Omega(t)}\big(\partial_t^\bullet|\Bu|^2
+|\Bu|^2\nabla\cdot\Bw\big)
&+ 2\mu\|\bbD(\Bu)\|_{\bbL^2(\Omega(t))}^2 
+\mu|\bmw_{\taubf}|_{\BH^1(\Gamma(t))}^2 \\
&+\gamma_0\mathscr{A}(\Gamma;\mathrm{id},\bmw)
+\varepsilon\mu \langle\bbP\Bu,\bmw_{\taubf}\rangle
= (\bmf,\bmu).
\end{aligned}
\end{equation}
By the Reynolds transport theorem \cite{BGN2020HoNA}, we also have
\begin{align*}
\int_{\Omega(t)}\big(\partial_t^\bullet|\Bu|^2+|\Bu|^2\nabla\cdot\Bw\big) 
= \frac{\mathrm{d}}{\mathrm{d}t}\int_{\Omega(t)} |\Bu|^2,
\quad \mathscr{A}(\Gamma;\mathrm{id},\bmw) 
=\frac{\Rd}{\Rd t}|\Gamma(t)|.
\end{align*}
Substituting the two equalities into \eqref{eq:energy0} yields the energy law
\begin{align*}
E'(t) + 2\mu\|\bbD(\Bu)\|_{\bbL^2(\Omega(t))}^2 
+\mu|\bmw_{\taubf}|_{\BH^1(\Gamma(t))}^2
% +\gamma_0\Ca_{\Gamma}(\mathrm{id},\bmw)
= (\bmf,\bmu) -\varepsilon\mu \langle\bbP\Bu,\bmw_{\taubf}\rangle.
\end{align*}
This shows $E'(t)\le 0$ for $\Bf={\bf 0}$ and $\varepsilon=0$.

Next we integrate the above energy law over $(0,t)$ and apply H\"{o}lder's inequality to the right-hand side. It follows that 
\begin{align}\label{eq:Et}
& E(t) + \mu\int_0^t \big[2\|\bbD(\Bu)\|^2_{\bbL^2(\Omega(s))} 
+|\bmw_{\taubf}|_{\BH^1(\Gamma(s))}^2\big]\D s \notag\\
=\, & E(0) + \int_0^t \big[(\Bf,\Bu)_{\Omega(s)}
- \varepsilon \mu \langle\bbP\Bu,\bmw_{\taubf}\rangle\big]\D s\notag\\
\le\,& E(0) + \int_0^t\big(\|\Bf\|_{\BL^2(\Omega(s))}^2 
   + \|\Bu\|^2_{\BL^2(\Omega(s))}
+\varepsilon\mu\|\Bu\|_{\Ltwov[\Gamma(s)]} 
\|\bmw_{\taubf}\|_{\Ltwov[\Gamma(s)]}\big)\D s.
\end{align}
The symmetric-gradient operator $\bbD$ has a kernel consisting of rigid-body motions
\begin{equation}\label{eq:def_rigid_space}
\BR(\Omega(t)) :=\left\{\bm a + \bbB(\bmx-\bmx_c(t)):
\bm a\in\mathbb R^d,\;\bbB\in\mathbb R^{d\times d},\;
\bbB^T = -\bbB\right\},
\end{equation}
where $\Bx_c(t):=\big(\int_{\Omega(t)}\bmx\,\D\bmx\big) \big/ \big(\int_{\Omega(t)} \,\Rd\bmx\big)$. Due to this kernel, Korn's inequality does not hold on $\Honev[\Omega(t)]$. To circumvent this issue, we introduce the rigid-motion-free space
\[
\BU(\Omega(t)) :=
   \left\{
      \bmv\in \BH^1(\Omega(t))
      \,\middle|\,
      (\bmv,\bm r)_{\Omega(t)}=0
      \quad
      \forall\,\Br\in\BR(\Omega(t))
   \right\}.
\]
Split $\Bu=\Bu_\perp + \Br$ with $\Bu_\perp\in \BU(\Omega(t))$ and $\Br\in\BR(\Omega(t))$. Clearly $\bbD(\Bu) =\bbD(\Bu_\perp)$.
It is known that Korn's inequality holds for $\Bu_\perp$ (\cite[Theorem 62.F]{Zeidler1988IV}), that is, there exists a constant $C_{\rm K}>0$ depending only on $\Omega(t)$ such that
\begin{equation}\label{eq:Korn_continuous}
C_{\mathrm K} \|\Bu_\perp\|_{\BH^1(\Omega(t))}\le \|\bbD(\Bu_\perp)\|_{\bbL^2(\Omega(t))}
=\|\bbD(\Bu)\|_{\bbL^2(\Omega(t))}.
\end{equation}

By the trace inequality and the norm equivalence on the finite dimensional space $\BR(\Omega(s))$, there exists a generic constant $C$ that depends only on $\Omega(s)$ such that
\begin{align*}
\|\Bu\|_{\Ltwov[\Gamma(s)]}^2  
\le\,& C\|\Bu\|_{\Ltwov[\Omega(s)]}
\|\Bu\|_{\Honev[\Omega(s)]} \\
\le\,& C\|\Bu\|^2_{\Ltwov[\Omega(s)]} + C_{\rm K}^2 |\Bu_\perp|^2_{\Honev[\Omega(s)]} 
    +C_{\rm K}^2 |\Br|^2_{\Honev[\Omega(s)]}\\
\le\,& C\|\Bu\|^2_{\Ltwov[\Omega(s)]} + \|\bbD(\Bu)\|^2_{\bbL^2(\Omega(s))} 
    +C \|\Br\|^2_{\Ltwov[\Omega(s)]}.
\end{align*}
Since $\Bu_\perp$ and $\Br$ are $L^2(\Omega(s))$-orthogonal, we have 
$\|\Br\|_{\Ltwov[\Omega(s)]} \le \|\Bu\|_{\Ltwov[\Omega(s)]} $. 
This yields
\begin{equation}\label{eq:uL2}
\|\Bu\|_{\Ltwov[\Gamma(s)]}^2  
\le C\|\Bu\|^2_{\Ltwov[\Omega(s)]} + \|\bbD(\Bu)\|^2_{\bbL^2(\Omega(s))}.
\end{equation}
Similarly, applying Poincar\'{e}'s inequality to $\bmw_{\taubf}$ yields
$\|\bmw_{\taubf}\|_{\Ltwov[\Gamma(s)]}
\le C|\bmw_{\taubf}|_{\Honev[\Gamma(s)]}$. By Young's inequality, there exists a constant $C_0>0$ depending only on $\Omega(s)$ such that 
\[
\varepsilon\mu\|\Bu\|_{\Ltwov[\Gamma(s)]}
\|\bmw_{\taubf}\|_{\Ltwov[\Gamma(s)]}
\le C_0\varepsilon^2\mu\big(\|\Bu\|^2_{\Ltwov[\Omega(s)]} + \|\bbD(\Bu)\|^2_{\bbL^2(\Omega(s))}\big)
+\frac{\mu}{2}|\bmw_{\taubf}|^2_{\Honev[\Gamma(s)]}.
\]
Set $\varepsilon_0:=C_0^{-1/2}$ and substitute the inequality into \eqref{eq:Et}. For any $\varepsilon\in[0,\varepsilon_0)$, we have
\begin{align*}
& E(t) + \mu\int_0^t \bigg[\|\bbD(\Bu)\|^2_{\bbL^2(\Omega(s))} 
+\frac{1}{2}|\bmw_{\taubf}|_{\BH^1(\Gamma(s))}^2\bigg]\D s \notag\\
\le\,& E(0) + \int_0^t\big[\|\Bf\|_{\BL^2(\Omega(s))}^2 
   + (1+\mu)\|\Bu\|^2_{\BL^2(\Omega(s))}\big]\D s.
\end{align*}
The proof is finished by using the Gronwall inequality.
\end{proof}

Based on Theorem~\ref{thm:continuous-energy}, we can control
the artificial tangential velocity $\bmw_{\taubf}$ with $\varepsilon$. This is the result of the following corollary.

\begin{corollary}\label{lem:tangential-velocity-stability}
There exists a constant $C>0$ depending only on the geometry of $Q(t)$, but not on $\varepsilon$ and physical parameters, such that
\begin{align}
\mu\int_0^t|\bmw_{\taubf}(s)|_{\BH^1(\Gamma(s))}^2 \D s
\le C\varepsilon^2 e^{ 2(1+\mu) t} \big[E(0) + \|\Bf\|_{\BL^2(Q(t))}^2\big].
\end{align}
\end{corollary}
\begin{proof}
Taking $\bm\eta= \bmw_{{\taubf}}$ in \eqref{eq:weak-wtau} and using the identity 
$\mathscr{A}(\Gamma;\mathrm{id},\bmw_{\taubf})=0$,
% $\Ca_\Gamma({\rm id}, \bmw_{\taubf}) =0$, 
we obtain
\begin{align*}
\|\nabla_\Gamma \bmw_{\taubf}\|_{\bbL^2(\Gamma(s))}^2
\le \varepsilon \|\bmu\|_{\Ltwov[\Gamma(s)]}
    \|\bmw_{{\taubf}}\|_{\Ltwov[\Gamma(s)]}
\le C\varepsilon \|\bmu\|_{\Ltwov[\Gamma(s)]}
\|\nabla_\Gamma\bmw_{{\taubf}}\|_{\bbL^2(\Gamma(s))}.
\end{align*}
Together with \eqref{eq:uL2}, this shows 
\[
\|\nabla_\Gamma \bmw_{\taubf} \|_{\bbL^2(\Gamma(s))} \le C\varepsilon \big[\|\Bu\|_{\Ltwov[\Omega(s)]} + \|\bbD(\Bu)\|_{\bbL^2(\Omega(s))}\big].
\]
The result is a direct consequence of Theorem~\ref{thm:continuous-energy}.
\end{proof}

\section{Semi-discrete ALE finite element method}
\label{sec:ALE-FEM}

In this section, we propose the finite element approximation of \eqref{eq:weak} on a time-dependent mesh $\Ct_h(t)$ which evolves with a discrete ALE mapping $\Phibf_h(\cdot,t)$ from an initial triangulation of $\Omega^0$.

\subsection{Finite element meshes}
\label{sec:ALE-mesh}

Assume that the given smooth initial domain $\Omega^0\subset\mathbb R^d$ is partitioned by a shape-regular and quasi-uniform iso-parametric triangulation $\Ct_h^0$ with mesh size $h$ and polynomial degree $k$, 
which fits the initial interface $\Gamma^0=\Gamma(0)$. 
Each simplex $K\in\Ct_h^0$ is closed and possibly curved, and is the image of the reference simplex $\widehat K$ under a polynomial mapping of degree $k$, 
denoted by \(\BF_K:\widehat K\to K\). Here $\BF_K$ is actually the parametrization of $K$. Each boundary simplex $K\in\Ct_h^0$, with one edge/face lying on the discrete boundary, contains a possibly curved boundary edge/face that approximates $\Gamma^0$ with geometric accuracy $O(h^{k+1})$. 
Such an approximation can be constructed using Lenoir's iso-parametric finite elements.
The initial domain $\Omega^0$ is thus approximated by the discrete domain $\Omega_h^0$.

{ 
% The geometry nodes determining $\Ct^0_h$ are denoted by $\Bx_j^0 \in \mathbb R^d$, $j=1,\cdots,N$. 
During the evolution, the mesh connectivity is kept, while only the geometry nodes move. 
The set of boundary faces/edges of $\Ct_h(t)$ is denoted by $\Cf_h(t)$. 
It provides a triangulation of $\Gamma_h(t)$ for  $t\in [0, T]$.
The corresponding discrete domain and boundary are defined by $\Omega_h(t)$ and $\Gamma_h(t)$, respectively.
}

Without causing confusion, we shall abbreviate the notations by omitting the $t$ variable, and simply write $\Omega_h$, $\Gamma_h$, $\Ct_h$,  $\Cf_h$ and so on in this section.

\subsection{Finite element spaces}
\label{sec:fem-spaces}

For each $K\in\Ct_h$, we still denote the iso-parametric transformation from $\wh{K}$ to $K$ by $\BF_K$. For any face/edge $F\subset\partial K$, we write $\BF_F:=\BF_K|_{F}$. Suppose $k\ge 2$ is an integer. Using these iso-parametric transformations, we define the finite element spaces on the evolving
domain and boundary by
\begin{align*}
U(\Ct_h):=\,& \big\{
v_h\in H^1(\Omega_h):
v_h\circ \BF_K\in   P_k(\widehat K)
\text{ for } K\in\Ct_h \big\}, \\
M(\Ct_h):=\,& \big\{v_h\in H^1(\Omega_h)/\bbR:
v_h\circ \BF_K\in   P_{k-1}(\widehat K)
\text{ for } K\in\Ct_h \big\}, \\
U(\Cf_h) :=\,&\big\{\xi_h\in H^1(\Gamma_h):
\xi_h\circ \BF_F\in P_k(\widehat{F})
\text{ for } F\in\Cf_h \text{ and }
\widehat{F}=\BF_F^{-1}(F)\big\},
\end{align*}
where $P_k(\widehat K)$ and $P_k(\widehat F)$ denote the spaces of $k^{\rm th}$-order polynomials on $\widehat K$ and $\wh{F}$, respectively. The vector-valued finite element spaces on $\Ct_h$ and $\Cf_h$ are denoted by $\BU(\Ct_h)$ and $\BU(\Cf_h)$, respectively.

% Define $\BU_0(\Ct_h)= \BU(\Ct_h)\cap\Honev[\Omega_h]$.
{ Define $\BU_0(\Ct_h)= \BU(\Ct_h)\cap \bm{H}_0^1(\Omega_h)$.}
It is known that the Taylor-Hood finite element pair admits the discrete inf-sup condition (cf. \cite{BGN2020HoNA}) 
\begin{equation} \label{eq:rigid-free-inf-sup}
\sup_{{\bf 0}\ne\Bv_h\in\BU_0(\Ct_h)}
\frac{\mathcal{B}(\Omega_h; q_h,\Bv_h)}{\|\Bv_h\|_{\bm{H}^1(\Omega_h)}}
\ge C_{\inf}\|q_h\|_{L^2(\Omega_h)}
\quad \forall\, q_h\in M(\Ct_h),
\end{equation}
where the constant $C_{\inf}>0$ depends only on $\Omega_h$ and $k$.

Note that $\Gamma_h$ is the union of element faces in $\Cf_h$. The unit outer normal $\Bn_h$ is discontinuous across the common boundary of any two faces $F_1,F_2\in\Cf_h$. To deal with the issue, we introduce the $L^2$-projection of $\Bn_h$ onto $\BU(\Cf_h)$, that is, $\bar\Bn_h\in \BU(\Cf_h)$ satisfies
\begin{align}\label{eq:nh}
\langle\bar\bmn_h, \bmv_h\rangle_{\Gh}
= \langle \bmn_h, \bmv_h\rangle_{\Gh} \quad \forall \bmv_h\in\BU(\Cf_h).
\end{align}
Then we define the following tangential finite element space
\begin{align} \label{eq:Uh-Gamma}
\BU_{\taubf}(\Cf_h) := \big\{\bm\eta_h\in [U(\Cf_h)/\bbR]^d: \bm\eta_h\cdot
\bar\Bn_h=0\big\}.
\end{align}
The projection matrix onto $\BU_{\taubf}(\Cf_h)$ is given by $\mathbb{P}_h=\bbI-\bar{\bmn}_h\otimes\bar{\bmn}_h$.

\subsection{Semi-discrete schemes}
\label{sec:semi-disc}

Let $\Ci_h: C(\Gamma_h)\to U(\Cf_h)$ be the nodal interpolation operator of iso-parametric elements. Suppose $\Bf$ is defined on $\Oh(t)$ for all $t\in [0,T]$. 
We propose an ALE evolving finite element approximation to \eqref{eq:weak} as follows: find
\[
(\Bu_h,p_h,H_h,\bmw_{\taubf,h},\Bw_h)\in 
\BU(\Ct_h) \times M(\Ct_h) \times U(\Cf_h)\times \BU_{\taubf}(\Cf_h)
\times \BU(\Ct_h)
\]
such that $\Bw_h=\Ci_h ((\bmu_{h} \cdot  \bar\bmn_h) \bar\bmn_h)
+\bmw_{{\taubf},h}$ on $\Gamma_h$ and 
\begin{subequations}\label{eq:weakh}
\begin{align} 
&\left(\partial_{t,h}^{\bullet} \bmu_h, \bmv_h\right)_{\Omega_h} +  \frac{1}{2} (\bmu_h,\bmv_h\nabla\cdot\bmw_h)_{\Omega_h}
+ \mathcal{O}(\Oh;\bmu_h-\bmw_h,\bmu_h,\bmv_h)  \notag \\ 
&\hspace{10mm} +\mathcal{A}(\Oh;\bmu_h,\bmv_h)
- \mathcal{B}(\Oh; p_h,\bmv_h)
+ \gamma_0 \left\langle H_h, \Ci_h(\bmv_h\cdot\bar\bmn_h)\right\rangle_{\Gamma_h} 
= (\bmf,\bmv_h)_{\Omega_h},   \\
&\mathcal{B}(\Oh;q_h,\bmu_h) = 0, \\ 
&\langle H_h , \xi_h\rangle_{\Gamma_h} =  
    \mathscr{A}(\Gamma_h;\mathrm{id}, \Ci_h(\xi_h \bar\bmn_h)) ,\\ 
&\mathscr{A}(\Gamma_h;\bmw_{{\taubf},h},  \bm\eta_h) 
= -\alpha\mathscr{A}(\Gamma_h;\mathrm{id},\bm\eta_h) 
-\varepsilon\langle \mathbb{P}_h\bmu_h ,\bm\eta_h \rangle_{\Gamma_h}, \\
&(\mathbb{D}(\bmw_h),\mathbb{D}( \bmchi_h))_{\Omega_h} =0,
\end{align} 
\end{subequations}
for all $(\bmv_h, q_h, \xi_h, \bm\eta_h, \bmchi_h)$ in 
$\BU(\Ct_h) \times M(\Ct_h) \times U(\Cf_h)\times \BU_{\taubf}(\Cf_h)
\times \BU_0(\Ct_h)$ { with the given initial fluid velocity $\Bu_h(0) \in \BU(\Ct_h)$}. 
Here ${\rm id}$ stands for the identity mapping on $\Gamma_h$ and the ALE time derivative is defined by
\begin{align} \label{eq:Dth}
\partial_{t,h}^\bullet\Bv =\partial_t\Bv +\bmw_h\cdot\nabla\Bv.
\end{align}
The system \eqref{eq:weakh} is coupled with the semi-discrete flow map $\Phibf_h$, satisfying
\begin{align}\label{eq:Phih}
\partial_t\Phibf_h(\xibf,t)
=\bmw_h(\Phibf_h(\xibf,t),t),\quad 
\Phibf_h(\xibf,0) = \xibf 
\quad\forall\,\xibf\in\Oh^0\cup\Gh^0.    
\end{align}

Next we show that the semi-discrete problem inherits the energy-dissipation property of the continuous model. Define the semi-discrete total energy by
\[
E_h(t) := \frac{1}{2}\|\bmu_h(t)\|_{\BL^2(\Omega_h(t))}^2
+ \gamma_0|\Gamma_h(t)|.
\]
 
\begin{theorem} \label{thm:semidiscrete-energy}
Suppose $\Gamma_h(t)$ is a non-degenerate and closed surface. Define $Q_h(t):=\{(\Bx,s): \Bx\in\Omega_h(s),\; s\in (0,t)\}$. 
There exists a constant $\varepsilon_1$ depending only on the geometry of $Q_h(t)$ such that, for any $\varepsilon\in [0,\varepsilon_1)$,
\begin{align*}
E_h(t) + \mu \int_0^t \Big[\|\bbD(\Bu_h)\|^2_{\bbL^2(\Omega_h(s))} 
+\frac12 |\bmw_{\taubf,h}|_{\BH^1(\Gamma_h(s))}^2\Big]\D s
\le  e^{ 2(1+\mu) t} 
\big[E_h(0) + \|\Bf\|_{\BL^2(Q_h(t))}^2\big].
\end{align*}
In particular, when $\bmf=\bm0$ and $\varepsilon=0$, the energy is dissipative: $E_h' \le 0$. 
\end{theorem}
\begin{proof}
The proof is similar to that of Theorem~\ref{thm:continuous-energy}. We sketch the ideas here. Taking $\bmv_h=\bmu_h$, $q_h=p_h$, $\xi_h=\Ci_h (\bmu_h\cdot \bar{\bmn}_h)$, and $\bm\eta_h = \bmw_{\taubf,h}$ in \eqref{eq:weakh}, we obtain
\begin{align*}
\frac{1}{2}\frac{\D}{\D t}\|\Bu_h\|_{\Ltwov[\Omega_h]}^2
+2\mu\|\mathbb D(\bmu_h)\|_{\bbL^2(\Omega_h)}^2
+\gamma_0\langle H_h,\xi_h\rangle_{\Gamma_h}
= (\Bf, \bmu_h)_{\Oh}.
\end{align*}
The nodal interpolations admit $\Ci_h(\xi_h\bar\Bn_h) = \Ci_h((\bmu_h\cdot\bar\bmn_h)\bar\Bn_h)$. Combining the third and fourth equations yields
\begin{align*}
\gamma_0\langle H_h,\xi_h\rangle_{\Gamma_h}
=\gamma_0 \mathscr{A}(\Gamma_h;\mathrm{id}, \bmw_h)
+\mu|\bmw_{\taubf,h}|_{\BH^1(\Gamma_h)}^2
+\varepsilon\mu\left\langle\mathbb{P}_h\bmu_h ,\bmw_{{\taubf},h} 
\right\rangle_{\Gamma_h}.    
\end{align*}
By \eqref{eq:Phih} and the surface transport identity, we have
\(\frac{\mathrm d}{\mathrm dt}|\Gamma_h|
=\mathscr{A}(\Gamma_h;\mathrm{id}, \bmw_h) 
% \mathcal{A}_{\Gamma}(\mathrm{id}, \bmw_h)
\). 
Then
\begin{align*}
E_h' +2\mu\|\mathbb D(\bmu_h)\|_{\bbL^2(\Omega_h)}^2
+\mu|\bmw_{\taubf,h}|_{\BH^1(\Gamma_h)}^2
= (\Bf, \bmu_h)_{\Oh}
-\varepsilon\mu\langle\bbP_h\bmu_h,
\bmw_{\taubf,h}\rangle_{\Gamma_h}.
\end{align*}
The rest of the proof is parallel to that of Theorem~\ref{thm:continuous-energy}. We omit the details here. 
\end{proof}

\section{Fully discrete scheme}
\label{sec:ALE-fully discrete}

In this section, we propose a fully discrete finite element
method for solving \eqref{eq:weak}. 
Let $0=t_0<t_1<\cdots<t_M=T$ be a uniform partition of $[0,T]$ with step size $\Delta t=t_{m+1}-t_m$. We define the one-step ALE map
\[
\Phibf_h^{m}:=\mathrm{id}+\Delta t\,\bmw_h^{m}: \Oh^{m-1}\to\Oh^{m}, \quad
m\ge 0,
\]
where $\Omega_h^{m},\Bw_h^m$ are the approximate fluid domain and mesh velocity at $t_m$, respectively. The boundary of $\Oh^{m}$ is denoted by $\Gh^{m}$. It admits 
\[
\Gh^{m} = \Phibf_h^{m}(\Gh^{m-1}).
\]
Similarly, the finite element meshes of $\Oh^{m-1}$ and $\Gh^{m-1}$ are driven by $\Phibf_h^{m}$ and admit
\[
\Ct_{h}^{m} = \big\{\Phibf^{m}_h(K): K\in\Ct_h^{m-1}\big\},\quad 
\Cf_{h}^{m} = \big\{\Phibf^{m}_h(F): F\in\Cf_h^{m-1}\big\}.
\]
The Jacobian of $\Phibf^m_h$ is denoted by 
\[
J_h^m:= \det(\nabla\Phibf_h^m).
\]
For $\Delta t$ small enough, $\Phibf_h^m$ is a bijection and $J_h^m>0$. Throughout this section, we assume that all elements in $\Ct_h^m$ are nondegenerate. 

\subsection{The CTV scheme}

Based on \eqref{eq:weakh}, we adopt the idea of \cite{GaoHuLisisc,Hu2026_Droplet} to derive a fully discrete scheme. 
Let $\bmn_h^m$ be the unit outer normal on $\Gh^m$. 
Similar to \eqref{eq:nh}, we denote the $L^2$-projection of $\Bn^m_h$ onto $\BU(\Cf^m_h)$ by $\bar\Bn^m_h\in \BU(\Cf^m_h)$, satisfying
\begin{align*}
\langle\bar\bmn^m_h, \bmv_h\rangle_{\Gh^m}
= \langle \bmn^m_h, \bmv_h\rangle_{\Gh^m} \quad \forall\, \bmv_h\in\BU(\Cf_h^m).
\end{align*}
The projection matrix onto $\BU_{\taubf}(\Cf_h^m)$ is given by $\mathbb{P}_h^m=\bbI-\bar{\bmn}_h^m\otimes\bar{\bmn}_h^m$.
Moreover, let $\Ci_h^m: C(\Gamma_h^m)\to U(\Cf_h^m)$ be the nodal interpolation operator of iso-parametric elements at $t_m$. Without an abuse of notation, we also use $\Ci_h^m$ to denote the interpolation operator on vector function spaces. Suppose $\Bf$ is defined on $\Oh^m$ for all $m\ge 0$. Write $\Bf_m = \Bf(t_m)$ and $\Oh^{-1}=\Oh^0$.

Suppose the fluid velocity $\bmu_h^m$, the mesh velocity $\bmw_h^m$, the flow map $\Phibf_h^m$, and the approximate domain $\Oh^m$ with boundary $\Gh^m$ are given at time $t_m$. Define $\hat\Bu_h^m$ and $\hat\Bw_h^m$ such that  
\[
\hat\Bu_h^m\circ\Phibf_h^m :=  (J_h^m)^{-1/2}\Bu_h^m,
\quad
\hat\Bw_h^m\circ\Phibf_h^m :=  (J_h^m)^{-1/2}\Bw_h^m.
\]
The fully discrete ALE scheme is to find
\[
(\bmu_h^{m+1},p_h^{m+1},H_h^{m+1},\bmw_{{\taubf},h}^{m+1}, \bmw_h^{m+1})
\in \BU(\Ct_h^m) \times M(\Ct_h^m) \times U(\Cf_h^m)\times \BU_{\taubf}(\Cf_h^m)
\times \BU(\Ct_h^m),
\]
such that $\Bw_h^{m+1}=\Ci_h^m ((\bmu_{h}^{m+1} \cdot  \bar\bmn_h^m) \bar\bmn_h^m)
+\bmw_{{\taubf},h}^{m+1}$ on $\Gamma_h^m$ and 
\begin{subequations}\label{eq-full discretization}
\begin{align}
&\frac{1}{\Delta t}(\Bu_h^{m+1}-\hat\Bu_h^m,\Bv_h)_{\Omega_h^m} +\Co(\Oh^m;\hat\bmu_h^m-\hat\bmw_h^m,
\bmu_h^{m+1},\Bv_h) 
+\Ca(\Oh^m;\bmu_h^{m+1},\bmv_h)\notag \\
&\hspace{20mm}
 -\mathcal{B}(\Oh^m;p_h^{m+1},\bmv_h) 
 + \gamma_0\langle H_h^{m+1}, \Ci_h^m(\bmv_h\cdot \bar\bmn_h^m)\rangle_{\Gh^m} 
= (\Bf_{m+1}, \bmv_h)_{\Oh^m},   
    \label{eq-full discretization 1}\\
&\mathcal{B}(\Oh^m;q_h,\bmu_h^{m+1}) = 0,   \label{eq-divuhm}\\ 
&\langle H_h^{m+1},  {\xi}_h \rangle _{\Gh^m} = \mathscr{A}(\Gamma_h^m;\Phibf_h^{m+1}, \Ci_h^m(\xi_h \bar{\bmn}_h^m)),  \label{eq-full discretization 2}  \\[0.3em]  
&\mathscr{A}(\Gamma_h^m;
\bmw_{{\taubf},h}^{m+1}, \bm\eta_h ) 
= -\alpha\mathscr{A}(\Gamma_h^m;
\Phibf_h^{m+1}, \bm\eta_h) -\varepsilon \langle \mathbb{P}_h^{m} \bmu_h^{m+1} ,  \bm\eta_h \rangle_{\Gamma_h^m}, 
    \label{eq-full discretization 3} \\
&(\mathbb{D}( \bmw_h^{m+1}), \mathbb{D}(\bmchi_h))_{\Oh^m} = 0,
    \label{eq:wh-m1}
\end{align}
\end{subequations}
for all $(\bmv_h,q_h,\xi_h,\bm\eta_h, \bm\chi_h) \in
\BU(\Ct_h^m) \times M(\Ct_h^m) \times U(\Cf_h^m)\times \BU_{\taubf}(\Cf_h^m)
\times \BU_0(\Ct_h^m)$. 
The system is coupled with the one-step ALE mapping
\begin{align}\label{eq:Phi-hm}
\Phibf_h^{m+1}:=\mathrm{id}+\Delta t\,\bmw_h^{m+1} \quad \hbox{on}\; 
\Oh^{m}\cup\Gh^{m}, \;
m\ge 0.
\end{align}

Problem \eqref{eq-full discretization} is posed on the domains $\Oh^m$, $\Oh^{m-1}$, the boundary $\Gh^m$, and the one-step mapping $\Phibf_h^m$, which are calculated at $t_m$. Therefore, \eqref{eq-full discretization} is actually a linear scheme for computing $(\bmu_h^{m+1},p_h^{m+1},H_h^{m+1},\bmw_{{\taubf},h}^{m+1}, \bmw_h^{m+1})$. 

\subsection{Energy dissipation}
Now we establish the well-posedness and energy-dissipation property of the fully discrete scheme \eqref{eq-full discretization}. Define the discrete total energy as
\begin{equation}
\label{eq:discrete-total-energy}
E_h^{m}:= \frac{1}{2}
\|\bmu_h^{m}\|_{\BL^2(\Omega_h^{m-1})}^2
+\gamma_0|\Gamma_h^{m}|, \quad m\ge 1.
\end{equation}

\begin{theorem}\label{thm:discrete-energy-dissipation}
There exists a constant $\varepsilon_2$ which is small enough, but independent of the meshes and physical parameters such that, when $\varepsilon \le\varepsilon_2$ { and $2(1+\mu)\Delta t <1$} for any $\ell\ge 1$,
\begin{align}\label{ieq:Eh}
&E_h^{\ell} +\frac{\mu}{2}
\sum_{m=1}^{\ell}\Delta t\big(
2\|\mathbb{D}(\bmu_h^m) \|_{\bbL^2(\Oh^{m-1})}^2 
+\big|\bmw_{{\taubf},h}^m\big|_{\BH^1(\Gh^{m-1})}^2\big) \notag \\
\le\,& e^{ 2(1+\mu)t_{\ell}}\bigg[E_h^0 + 
\sum_{m=1}^{\ell}\Delta t\|\Bf_m\|^2_{\Ltwov[\Oh^{m-1}]}\bigg].
\end{align}
\end{theorem}
\begin{proof}
Choose $\bmv_h=\bmu_h^{m+1}$, $q_h=p_h^{m+1}$, and 
% $\xi_h=\gamma_0\Ci_h^m(\bmu_h^{m+1}\cdot\bar{\bmn}_h^m)$ 
{ $\xi_h=\Ci_h^m(\bmu_h^{m+1}\cdot\bar{\bmn}_h^m)$ }
in \eqref{eq-full discretization}.
Notice the identities that $\Co(\Oh^m;\hat\bmu_h^m-\hat\bmw_h^m,
\bmu_h^{m+1},\Bu_h^{m+1})=0$ and  
\begin{align*}
2(\Bu_h^{m+1}-\hat\Bu_h^m,\Bu_h^{m+1})_{\Omega_h^m} 
=\,& \|\Bu_h^{m+1}\|_{\Ltwov[\Oh^m]}^2
-\|\hat\Bu_h^{m}\|_{\Ltwov[\Oh^m]}^2
+\|\Bu_h^{m+1} -\hat\Bu_h^m\|_{\Ltwov[\Oh^m]}^2 \\
=\,& \|\Bu_h^{m+1}\|_{\Ltwov[\Oh^m]}^2
-\|\Bu_h^{m}\|_{\Ltwov[\Oh^{m-1}]}^2
+\|\Bu_h^{m+1} -\hat\Bu_h^m\|_{\Ltwov[\Oh^m]}^2.
\end{align*}
By combining all equations in \eqref{eq-full discretization}, we obtain
\begin{align}\label{eq-dis-energy-1}
&\frac{1}{2\Delta t}\big(\|\bmu_h^{m+1}\|_{\BL^2(\Omega_h^m)}^2
-\|\bmu_h^m\|_{\BL^2(\Omega_h^{m-1})}^2
+\|\bmu_h^{m+1}-\hat\bmu_h^m\|_{\BL^2(\Omega_h^m)}^2\big)
+2\mu\|\bbD(\bmu_h^{m+1})\|_{\bbL^2(\Oh^m)}^2    \notag \\
&\hspace{31mm} +\gamma_0
\mathscr{A}(\Gamma_h^m;\Phibf_h^{m+1},
\Ci_h^m(\xi_h \bar{\bmn}_h^m))=(\Bf_{m+1},\bmu_h^{m+1})_{\Oh^m}.
\end{align}

Taking $\bm\eta_h=-\mu\bmw_{{\taubf},h}^{m+1}$ in \eqref{eq-full discretization 3} gives
\begin{equation*}
-\gamma_0\mathscr{A}(\Gamma_h^m;\Phibf_h^{m+1}, \bmw_{\taubf,h}^{m+1})
=\mu\big|\bmw_{\taubf,h}^{m+1}\big|_{\BH^1(\Gh^m)}^2
+ \varepsilon\mu\big\langle \mathbb{P}_h^m\bmu_h^{m+1},\bmw_{{\taubf},h}^{m+1}\big\rangle_{\Gh^m}.
\end{equation*}
The last term on the left-hand side of \eqref{eq-dis-energy-1} admits
\begin{align}\label{ieq:A-Gammah}
&\gamma_0\mathscr{A}(\Gamma_h^m;\Phibf_h^{m+1},\Ci_h^m(\xi_h \bar{\bmn}_h^m)) 
    \notag \\ 
=\,&  \gamma_0\mathscr{A}(\Gamma_h^m;\Phibf_h^{m+1},
    \bmw_{h}^{m+1} - \bmw_{{\taubf},h}^{m+1} ) \notag \\
=\,&\frac{\gamma_0}{\Delta t}
\mathscr{A}(\Gamma_h^m;\Phibf_h^{m+1}, \Phibf_h^{m+1}-\mathrm{id}) 
+\mu \big|\bmw_{\taubf,h}^{m+1}\big|_{\BH^1(\Gh^m)}^{2}  
+\varepsilon\mu  \big\langle \mathbb{P}_h^m\bmu_h^{m+1},\bmw_{{\taubf},h}^{m+1}
\big\rangle_{\Gh^m} \notag \\
\ge\,&\frac{\gamma_0}{\Delta t}(|\Gamma_h^{m+1}|-|\Gamma_h^m|)
+\mu \big|\bmw_{\taubf,h}^{m+1}\big|_{\BH^1(\Gh^m)}^{2}  
+\varepsilon\mu  \langle \mathbb{P}_h^m\bmu_h^{m+1},\bmw_{{\taubf},h}^{m+1}
\rangle_{\Gh^m},
\end{align}
where we have used $\mathscr{A}(\Gamma_h^m;\Phibf_h^{m+1},\Phibf_h^{m+1}-\mathrm{id})\geq |\Gamma_h^{m+1}|-|\Gamma_h^m|$ in the last inequality (see \cite{BARRETT20084281}). 

Next we estimate $\langle \mathbb{P}_h^m\bmu_h^{m+1},\bmw_{{\taubf},h}^{m+1}
\rangle_{\Gh^m}$. 
By arguments similar to \eqref{eq:uL2}, there exists a constant $C>0$ depending only on $\Oh^m$ such that
\begin{equation*}
\|\Bu_h^{m+1}\|_{\Ltwov[\Gh^m]}  
\le C\|\Bu_h^{m+1}\|_{\Ltwov[\Oh^m]} 
+ \|\bbD(\Bu_h^{m+1})\|_{\bbL^2(\Oh^m)}.
\end{equation*}
Moreover, using Poincar\'{e}'s inequality, we have
\[
\|\bmw^{m+1}_{\taubf,h}\|_{\Ltwov[\Gamma_h^m]}
\le C|\bmw^{m+1}_{\taubf,h}|_{\Honev[\Gamma_h^m]}.
\]
The H\"{o}lder inequality shows that
\begin{align*}
\varepsilon\big|\langle \mathbb{P}_h^m\bmu_h^{m+1},\bmw_{{\taubf},h}^{m+1}
\rangle_{\Gh^m}\big|
\le \,& \varepsilon \|\bmu_h^{m+1}\|_{\Ltwov[\Gh^m]}
\|\bmw_{{\taubf},h}^{m+1}\|_{\Ltwov[\Gh^m]} \\
\le\,& C_1\varepsilon^2\big[\|\Bu_h^{m+1}\|_{\Ltwov[\Oh^m]}^2 
+ \|\bbD(\Bu_h^{m+1})\|_{\bbL^2(\Oh^m)}^2\big]
+ \frac{1}{2}|\bmw^{m+1}_{\taubf,h}|_{\Honev[\Gamma_h^m]}^2.   
\end{align*}
Choosing $\varepsilon\le \varepsilon_2:={ (2C_1)^{-1/2}}$ and substituting the inequality into \eqref{ieq:A-Gammah}, we obtain
\begin{align*}
&\gamma_0\mathscr{A}(\Gamma_h^m;\Phibf_h^{m+1},\Ci_h^m(\xi_h \bar{\bmn}_h^m)) 
    \notag \\ 
\ge\,&\frac{\gamma_0}{\Delta t}(|\Gamma_h^{m+1}|-|\Gamma_h^m|)
+\frac{\mu}{2} \big|\bmw_{\taubf,h}^{m+1}\big|_{\BH^1(\Gh^m)}^{2}  
-{ \frac{\mu}{2}}
\big(\|\Bu_h^{m+1}\|_{\Ltwov[\Oh^m]}^2 
+ \|\bbD(\Bu_h^{m+1})\|_{\bbL^2(\Oh^m)}^2\big).
\end{align*}
Furthermore, substituting this result into \eqref{eq-dis-energy-1}, { and applying Young's inequality}, we end up with 
\begin{align} \label{ieq:E-m1}
&E_h^{m+1}-E_h^m
+\mu\Delta t\|\bbD(\bmu_h^{m+1})\|_{\bbL^2(\Oh^m)}^2
+\frac{\mu}{2}\Delta t\big|\bmw_{\taubf,h}^{m+1}\big|_{\BH^1(\Gh^m)}^{2} 
    \notag \\  
\le\,& { \frac{1+\mu}{2}} \Delta t\|\Bu_h^{m+1}\|_{\Ltwov[\Oh^m]}^2
+\Delta t \|\Bf_{m+1}\|^2_{\Ltwov[\Oh^m]}.
\end{align}
The proof is finished by summing the inequality for $m=0,\cdots,\ell-1$ and using the discrete Gronwall inequality { together with
\(
\frac{1+\mu}{1-(1+\mu)\Delta t}
\le 2(1+\mu),
\)
follows from the constraint
$2(1+\mu)\Delta t<1$.}
\end{proof}

\begin{corollary}\label{cor:well-posedness-full-discrete}
Under the condition of Theorem~\ref{thm:discrete-energy-dissipation}, problem \eqref{eq-full discretization} has a unique solution at each time step. 
\end{corollary}
\begin{proof}
Since \eqref{eq-full discretization} is a linear problem, it suffices to prove the uniqueness of the solution. Suppose $\Bu_h^m=\Bw_h^m=\Bf_{m+1}\equiv {\bf 0}$ and $2(1+\mu)\Delta t <1$. 
Then \eqref{ieq:E-m1} implies $\Bu_h^{m+1}\equiv {\bf 0}$ and $\bmw^{m+1}_{\taubf,h}\equiv {\bf 0}$. Moreover, \eqref{eq:wh-m1} implies $\Bw_h^{m+1}\equiv {\bf 0}$ and \eqref{eq-full discretization 3} implies
\[
\mathscr{A}(\Gamma_h^m; \Phibf_h^{m+1}, \bm\eta_h) = 0
\quad \forall\,\etabf_h\in \BU_{\taubf}(\Cf_h^m).
\]
Combining this equality and \eqref{eq-full discretization 2}, we find that $H_h^{m+1}\equiv 0$. Finally, using \eqref{eq-full discretization 1} and the discrete inf-sup condition \eqref{eq:rigid-free-inf-sup}, we obtain $p_h^{m+1}\equiv 0$. This shows the uniqueness of the solution. The proof is finished.
\end{proof}

\begin{remark}\upshape
The assumptions on $\varepsilon$ in Theorems~\ref{thm:continuous-energy}, \ref{thm:semidiscrete-energy}, and \ref{thm:discrete-energy-dissipation} are  sufficient conditions. They are very mild. The particular choice $\varepsilon=0$ is allowed. The parameter leads to a class of stable numerical schemes.
In Section~\ref{num:epsilon}, we further investigate the influence of $\varepsilon$ numerically and demonstrate that the proposed method is robust with respect to $\varepsilon\to 0$.
\end{remark}

\subsection{Temporal consistency}

Compared with the BGN scheme, the CTV scheme \eqref{eq-full discretization} has a key advantage in that it is temporally consistent; that is, the scheme remains stable as $\Delta t\to 0$. To state this more precisely, we prove that the solution of \eqref{eq-full discretization} converges to the weak solution of \eqref{eq:weakh} as $\Delta t\to 0$.

For convenience, we extend the numerical solutions continuously to the exteriors of their domains. Define the temporal interpolations of discrete functions
\begin{align}\label{eq-int}
\xi_h^{\Delta t}(t) = \xi_h^m 
+\frac{t-t_m}{\Delta t}(\xi_h^{m+1}-\xi_h^m) 
\quad \forall\, t\in [t_m,t_{m+1}],\;
m=0,1,\cdots,M-1,
\end{align}
where $\xi= \Bu, \hat\Bu, p, \Bw, \Bw_{\taubf}, H$ are used for numerical solutions, $\xi=E$ for the total energy, and $\xi=J$ for the discrete Jacobian. From \eqref{eq:Phi-hm} and \eqref{eq-int}, it is easy to see that
\begin{align*}
\frac{\Bu_h^{m+1}-\hat\Bu_h^m}{\Delta t} 
=\frac{\partial\Bu_h^{\Delta t}}{\partial t}
+\frac{\Bu_h^{\Delta t}(t_m) -\hat\Bu_h^{\Delta t}(t_m)}{\Delta t}.
\end{align*}
Then \eqref{eq-full discretization} can be rewritten as follows
\begin{align} 
&\Big(\frac{\partial\Bu_h^{\Delta t}}{\partial t} +\frac{\Bu_h^{\Delta t}(t_m) -\hat\Bu_h^{\Delta t}(t_m)}{\Delta t}, \bmv_h\Big)_{\Omega_h^m}
 + \Co(\Oh^m; \hat\Bu_h^{\Delta t}(t_m) -\hat\Bw_h^{\Delta t}(t_m), \bmu_h^{\Delta t}(t_{m+1}),\Bv_h)  \notag \\
&\hspace{40mm} 
+\Ca(\Oh^m; \bmu_h^{\Delta t}(t_{m+1}),\bmv_h)
-\Cb(\Oh^m; p_h^{\Delta t}(t_{m+1}),\bmv_h) \notag\\
&\hspace{40mm}
 + \gamma_0\langle H_h^{\Delta t}(t_{m+1}), \Ci_h^m(\bmv_h\cdot \bar\bmn_h^m) \rangle_{\Gh^m} 
= (\Bf(t_{m+1}), \bmv_h)_{\Oh^m},   
   \label{eq-int-uh}\\
&\mathcal{B}(\Oh^m;q_h,\bmu_h^{\Delta t}(t_{m+1})) = 0,   \label{eq-int-divuh}\\ 
&\langle H_h^{\Delta t}(t_{m+1}),  {\xi}_h \rangle _{\Gh^m} 
    = \mathscr{A}(\Gamma_h^m;\mathrm{id}+\Delta t\,\bmw_h^{\Delta t}(t_{m+1}), 
    \Ci_h^m(\xi_h \bar{\bmn}_h^m)), 
    \label{eq-int-Hh} \\ 
&\mathscr{A}(\Gamma_h^m;
\bmw_{\taubf,h}^{\Delta t}(t_{m+1}), \bm\eta_h ) 
= -\alpha\mathscr{A}(\Gh^m;\mathrm{id}+\Delta t\,\bmw_h^{\Delta t}(t_{m+1}), \bm\eta_h) \notag \\
&\hspace{40mm}
    -\varepsilon \langle \mathbb{P}_h^m \bmu_h^{\Delta t}(t_{m+1}),  
\bm\eta_h \rangle_{\Gh^m}, 
    \label{eq-int-wtau} \\
&(\mathbb{D}(\bmw_h^{\Delta t}(t_{m+1})), \mathbb{D}(\bmchi_h))_{\Oh^m} = 0,  \label{eq-int-wh}
\end{align} 
for all $(\bmv_h,q_h,\xi_h,\bm\eta_h, \bm\chi_h) \in
\BU(\Ct_h^m) \times M(\Ct_h^m) \times U(\Cf_h^m)\times \BU_{\taubf}(\Cf_h^m)
\times \BU_0(\Ct_h^m)$. 
\vspace{1mm}

\begin{theorem}
Suppose $(\Bu_h,p_h,\Bw_h, \Bw_{\taubf,h}, H_h)$ is the solution of \eqref{eq:weakh}, and that $\varepsilon_1,\varepsilon_2$ are the two constants in Theorems~\ref{thm:semidiscrete-energy} and \ref{thm:discrete-energy-dissipation}, respectively. 
For $0\le \varepsilon\le \min\{\varepsilon_1,\varepsilon_2\}$, 
\begin{align*}
\lim_{\Delta t\to 0} \big\|\xi_h^{\Delta t}(t) -\xi_h(t)\|_{\Ltwo[\Oh(t)]} =0 
\quad \forall\, t\in [0,T], \quad 
\xi= \Bu, p, \Bw, \Bw_{\taubf}, H.
\end{align*}
\end{theorem}
\begin{proof}
First we fix an instant $t=t_m>0$ and let $\Delta t\to 0$. With an abuse of notations, we write 
\[
\Oh(t)=\Oh^m,\;\; \Gh(t)=\Gh^m,\;\; 
\Ct_h(t)=\Ct_h^m,\;\; \Cf_h(t)=\Cf_h^m,
\;\; 
\bar\Bn_h(t)=\bar\Bn_h^m, \;\;
\Ci_h(t)=\Ci_h^m.
\]
Without causing confusion, we also abbreviate the notations in this proof to 
\[
\Oh=\Oh(t), \;\; \Gh=\Gh(t),\;\;
\Ct_h=\Ct_h(t),\;\;
\Cf_h=\Cf_h(t),\;\;
\bar\Bn_h=\bar\Bn_h(t),\;\;\Ci_h=\Ci_h(t).
\]

The key ingredient of the proof is that the dimensions of the finite element spaces are fixed for all time and equal to the dimensions at the initial time. By
Theorem~\ref{thm:discrete-energy-dissipation}, the sequence $\{E_h^{\Delta t}(t+\Delta t)\}$ is uniformly bounded as a function of $\Delta t$. By the equivalence of norms on a finite-dimensional space (see \cite[inequality~(4.5.5)]{brenner2008}), the sequence $\{\Bu_h^{\Delta t}(t+\Delta t)\}$ is also uniformly bounded in the $\BH^1$-norm as a function of $\Delta t$, and has a sub-sequence, denoted by the same notation, which converges to some $\Bu_h(t)\in \BU(\Ct_h)$, namely,
\begin{align*}
\lim_{\Delta t\to 0}
\|\Bu_h^{\Delta t}(t+\Delta t)-\Bu_h(t)\|_{\BH^1(\Oh)} =0.
\end{align*}
The norm equivalence also shows
\begin{align*}
\lim_{\Delta t\to 0}
\|\Bu_h^{\Delta t}(t+\Delta t)-\Bu_h(t)\|_{\BH^1(\Gh)} =0.
\end{align*}
Passing to the limit in \eqref{eq-int-divuh} as $\Delta t\to 0$ yields
\begin{align}\label{lim-divuh}
\mathcal{B}(\Oh;q_h,\bmu_h(t)) = 0
\quad \forall\, q_h\in M(\Ct_h).
\end{align}

Similarly, the approximate normal vector $\bar\Bn_h(t)\in\BU(\Cf_h)$ is uniformly bounded in the $\BH^1(\Gh)$ norm with respect to $\Delta t$. 
We rearrange \eqref{eq-int-wtau} as follows 
\begin{align*}
&(1+\alpha\Delta t)
\mathscr{A}(\Gamma_h;
\bmw_{\taubf,h}^{\Delta t}(t+\Delta t), \bm\eta_h) \\
=\,& -\alpha\mathscr{A}(\Gh;\mathrm{id}+\Delta t\,\Ci_h ((\bmu_{h}^{\Delta t}(t+\Delta t) \cdot  \bar\bmn_h) \bar\bmn_h), \bm\eta_h)  
-\varepsilon \langle { \mathbb{P}_h} \bmu_h^{\Delta t}(t+\Delta t),  
\bm\eta_h \rangle_{\Gh}.
\end{align*}
Pass to the limit in the equality as $\Delta t\to 0$. There exists a $\Bw_{\taubf,h}(t)\in\BU_{\taubf}(\Cf_h)$ such that
\begin{align} \label{lim-wh-tau} 
\mathscr{A}(\Gamma_h;
\bmw_{\taubf,h}(t), \bm\eta_h ) 
=-\alpha\mathscr{A}(\Gh;\mathrm{id}, \bm\eta_h)  
-\varepsilon \langle { \mathbb{P}_h} \bmu_h(t),\bm\eta_h \rangle_{\Gh}
\quad \forall\, \etabf_h\in\BU_{\taubf}(\Cf_h).
\end{align}
The weak convergence on $\BU_{\taubf}(\Cf_h)$ implies the strong convergence
\begin{align*}
\lim_{\Delta t\to 0} \|\bmw_{\taubf,h}^{\Delta t}(t+\Delta t) -\bmw_{\taubf,h}(t)\|_{\BH^1(\Gh)}
=0.
\end{align*}
Furthermore, by \eqref{eq-int-wh}, there exists a $\Bw_h(t)\in \BU(\Ct_h(t))$ such that 
\begin{align}\label{lim-wh}
&\lim\limits_{\Delta t\to 0}
\|\Bw_h^{\Delta t}(t+\Delta t)-\Bw_h(t)\|_{\BH^1(\Oh)}=0, 
    \notag\\
&(\mathbb{D}( \bmw_h(t)), \mathbb{D}(\bmchi_h))_{\Oh} = 0 \quad
\forall \chibf_h \in \BU_0(\Ct_h).  
\end{align}
By \eqref{eq-int-Hh}, $\|H_h^{\Delta t}(t+\Delta t)\|_{\Ltwo[\Gh]}$ is uniformly bounded as $\Delta t\to 0$. Then $\{H_h^{\Delta t}(t+\Delta t)\}$ has a sub-sequence which converges to some $H_h(t)\in U(\Cf_h)$, which admits
\begin{align}
\label{lim-Hh}
\langle H_h(t),\xi_h\rangle_{\Gamma_h} 
=\mathscr{A}(\Gh;\mathrm{id}, 
\Ci_h(\xi_h\bar{\Bn}_h))\quad
\forall\,\xi_h\in U(\Cf_h).    
\end{align}

It is left to study the limit of \eqref{eq-int-uh} as $\Delta t\to 0$. 
Since $\Bu_h^{\Delta t}$ and $\Bw_h^{\Delta t}$ are continuous in time, we infer that
\begin{align*}
&\Bu_h(t) = \lim_{\Delta t\to 0}\Bu_h^{\Delta t}(t+\Delta t) 
= \lim_{\Delta t\to 0}\Bu_h^{\Delta t}(t),\\
&\Bw_h(t)=\lim_{\Delta t\to 0}\Bw_h^{\Delta t}(t+\Delta t) 
= \lim_{\Delta t\to 0}\Bw_h^{\Delta t}(t).
\end{align*}
Furthermore, the definition of $\hat\Bu_h^m$ implies that
\begin{align*}
\frac{\Bu_h^m -\hat\Bu_h^m}{\Delta t} 
=\,& \frac{\Bu_h^m-\Bu_h^m \circ(\Phibf_h^m)^{-1}}{\Delta t}
+ \frac{1-(J_h^m)^{-1/2}}{\Delta t} 
\Bu_h^m\circ(\Phibf_h^m)^{-1}.
\end{align*}
By the identity $J_h^m=\det(\bbI+\Delta t \nabla\Bw_h^m)$, we have 
\begin{align*}
&\lim_{\Delta t\to 0}
\frac{1-[J^{\Delta t}_h(t)]^{-1/2}}{\Delta t} 
=\frac{1}{2} \nabla\cdot\Bw_h(t),\\ 
&\lim_{\Delta t\to 0}
\frac{\big[\Bu_h^{\Delta t}-\Bu_h^{\Delta t}
\circ(\Phibf_h^{\Delta t})^{-1}\big](t)}{\Delta t}
=(\Bw_h(t)\cdot\nabla)\Bu_h(t).
\end{align*}
We conclude that, for any $\Bv_h\in\BU(\Ct_h)$, 
\begin{align*}
\lim_{\Delta t\to 0}\Big(
\frac{\Bu_h^{\Delta t}(t) -\hat\Bu_h^{\Delta t}(t)}{\Delta t},\Bv_h\Big)_{\Oh}
=\big((\Bw_h(t)\cdot\nabla)\Bu_h(t),\Bv_h\big)_{\Oh} 
+\frac{1}{2}\big(\Bu_h(t), 
\Bv_h\nabla\cdot\Bw_h(t)\big)_{\Oh}.
\end{align*}

Now we define the subspace 
\[
\BU_\perp(\Ct_h) :=\big\{\Bv_h\in \BU(\Ct_h): \Cb(\Oh; q_h,\Bv_h)=0\;\;\forall\,
q_h\in M(\Ct_h)\big\}.
\]
Choosing $\Bv_h\in \BU_\perp(\Ct_h)$ in \eqref{eq-int-uh} and passing to the limit, we obtain
\begin{align*}
\lim_{\Delta t\to 0}\Big(
\frac{\partial\Bu_h^{\Delta t}}{\partial t},\Bv_h\Big)_{\Oh}
&+\big((\Bw_h(t)\cdot\nabla)\Bu_h(t),\Bv_h\big)_{\Oh} 
+\frac{1}{2}\big(\Bu_h(t), 
\Bv_h\nabla\cdot\Bw_h(t)\big)_{\Oh}  \notag \\
&+\Co(\Oh;\bmu_h(t)-\bmw_h(t), \bmu_h(t),\Bv_h) 
+\Ca(\Oh;\bmu_h(t),\bmv_h) \notag \\
&+ \gamma_0\langle H_h(t), 
\Ci_h(\bmv_h\cdot \bar\bmn_h)\rangle_{\Gh} 
= (\bmf(t), \bmv_h)_{\Oh}.
\end{align*}
Therefore, the limit of the first term on the left-hand side exists. There exists a $\Bd_h(t)\in \BU_\perp(\Ct_h)$ such that
\[
\lim_{\Delta t\to 0}
\frac{\partial\Bu_h^{\Delta t}}{\partial t} 
= \Bd_h(t).
\] 
Furthermore, by the inf-sup condition \eqref{eq:rigid-free-inf-sup}, there exists a $p_h(t)\in M(\Ct_h)$ such that
\begin{align}\label{lim-uh}
\big(\Bd_h(t),\Bv_h\big)_{\Oh}
&+\big((\Bw_h(t)\cdot\nabla)\Bu_h(t)
+\tfrac{1}{2}(\nabla\cdot\Bw_h(t))\Bu_h(t), \Bv_h\big)_{\Oh} 
+\Ca(\Oh;\bmu_h(t),\bmv_h)
    \notag \\
&+\Co(\Oh;\bmu_h(t)-\bmw_h(t),
\bmu_h(t),\Bv_h)  
-\mathcal{B}(\Oh;p_h(t),\bmv_h) 
    \notag \\
&+ \gamma_0\langle H_h(t), \Ci_h(\bmv_h\cdot \bar\bmn_h)\rangle_{\Gh} 
= (\bmf(t), \bmv_h)_{\Oh} \qquad
\forall\,\Bv_h\in\BU(\Ct_h).
\end{align}

By \eqref{eq-int} and the arbitrariness of $t$, the limit $\Bu_h(t)$ is continuous in $t$. It is easy to verify that $\Bd_h=\partial_t\Bu_h$ holds in the weak sense. Using \eqref{eq:Dth}, we have
\[
(\Bd_h(t),\Bv_h)_{\Oh}
+((\Bw_h(t)\cdot\nabla)\Bu_h(t), \Bv_h)_{\Oh} 
= (\partial^\bullet_{t,h}\Bu_h(t), \Bv_h)_{\Oh}.
\]
Substituting this equality into \eqref{lim-uh} and collecting \eqref{lim-divuh}, \eqref{lim-wh-tau}, \eqref{lim-wh}, and \eqref{lim-Hh}, we find that the limit functions satisfy \eqref{eq:weakh}.
The proof is finished.
\end{proof}

\section{Numerical tests} 
\label{sec:numerical}

In this section, we present numerical experiments to demonstrate the performance of the CTV scheme.
Unless otherwise stated, we take
$\bmf=\bm 0$, $\Bu_h=\bm 0$, and $\varepsilon=10^{-4}$. The dynamics of the system is driven by the surface tension of fluid. 
We employ the Taylor--Hood $P_2$--$P_1$ pair for the fluid velocity and pressure. 
To quantify the domain shape and bulk mesh quality in the 2D case, we use the circularity \cite{DuanLiALEJCP22} and the mesh-quality indicators
\[
C_h^m := 2\pi^{1/2} |\Oh^m|^{1/2}|\Gh^m|^{-1},
    \qquad
\sigma_{\max}^m :=
    \max_{K\in\Ct_h^m}\frac{h(K)}{r(K)}.
\]
Here, $0<\mathcal C_h^m\le 1$, with $\mathcal C_h^m=1$ corresponding to a circle,
while $h(K)$ and $r(K)$ denote the diameters of the circumcircle and the maximal inscribed disk of $K$, respectively. Smaller values of $\sigma_{\max}^m$ indicate better mesh quality. All of the numerical experiments are implemented using NGSolve; see \url{https://ngsolve.org}.

\subsection{Accuracy test}
\label{num:accuracy}
We take the initial computational domain $\Omega_h^0$ to be an ellipse with a semi-major axis $a=2$ and a semi-minor axis $b=1$. 
% In the early stage of the evolution, due to the non-uniform curvature distribution along the elliptical boundary, surface tension induces an unbalanced Laplace pressure difference within the fluid. This pressure difference drives the fluid to flow from the major axis towards the minor axis. 
% As time progresses, the droplet gradually retracts, accompanied by the dissipation of kinetic energy, and eventually evolves into a steady state of a perfect circle with constant curvature, as illustrated in Fig. \ref{fig:ex1-mesh}. And we show the time evolution of the area circularity and the circularity in Fig. Additionally, Fig.  shows the temporal evolution of the total energy, which indicates the proposed method strictly adheres to the energy-dissipation principle.
The time-step size is taken as $\Delta t = 10^{-3}$, 
and the simulations are performed on a sequence of successively refined meshes up to $T=1$, 
with viscosity coefficient $\mu=0.01$ and surface tension coefficient $\gamma_0=1$.

Figure~\ref{fig:converge_mesh} shows the mesh in the initial state 
with size $h = 0.15$ and the mesh in the final state.
The numerical solutions on different mesh levels are defined on different evolving domains. 
To compute the approximation errors, we first pull the numerical solutions back to the initial domain $\Oh^0$ by ALE maps 
\[
\tilde{\bmu}_h^M :=\bmu_h^M\circ\Psibf_h^{M-1},\qquad
\tilde p_h^M:= p_h^M\circ\Psibf_h^{M-1},
\]
where $\Psibf_h^m := \Phibf_h^m \circ \Phibf_h^{m-1}\circ\cdots\circ\Phibf_h^1$. 
The pulled-back solutions on $\Ct_h^0$ are prolonged to $\Ct_{h/2}^0$ by using finite element interpolations
\[
\begin{aligned}
\| \bm{e}_\bmu\|_{\bm{H}^1}
:=
\big\|\Ci_{h/2}\tilde{\bmu}_h^M
-\tilde{\bmu}_{h/2}^M
\big\|_{\Honev[\Omega_{h/2}^0]},\quad
\|e_p\|_{L^2}
:=\big\|
\pi_{h/2}\tilde{p}_h^M - \tilde{p}_{h/2}^M
\big\|_{L^2(\Omega_{h/2}^0)},
\end{aligned}
\]
where $\Ci_{h/2}$ and $\pi_{h/2}$ are the nodal interpolation operators onto $\BU(\Ct_{h/2}^0)$ and $M(\Ct_{h/2}^0)$, respectively.
The resulting self-convergence errors are shown in Table~\ref{tab:converge}. 
Both errors exhibit asymptotically second-order convergence, consistent with the expected optimal spatial accuracy.

\begin{figure}[htbp]
    \centering
    \includegraphics[width=0.4\linewidth]{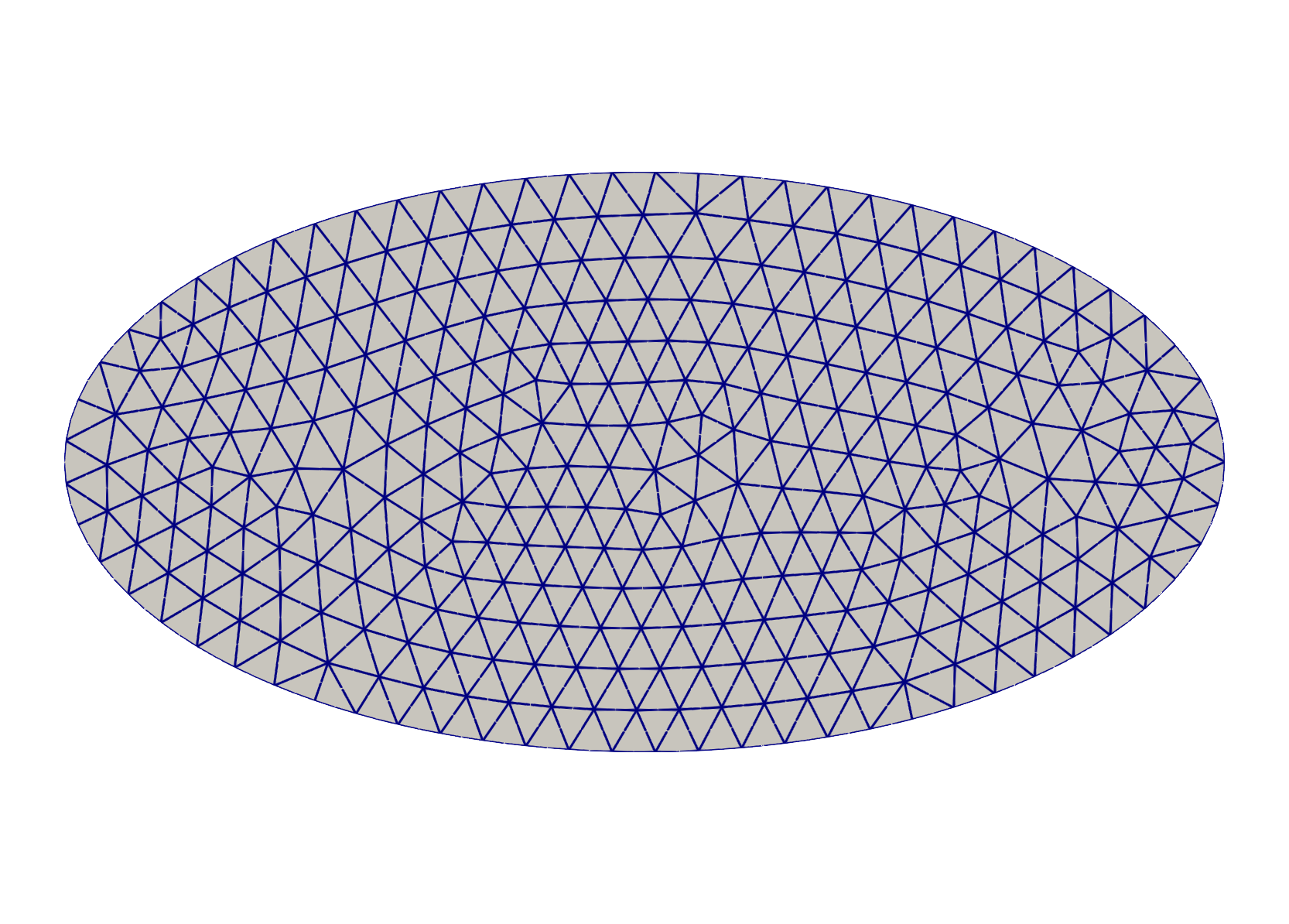}
    \hfill
    \includegraphics[width=0.4\linewidth]{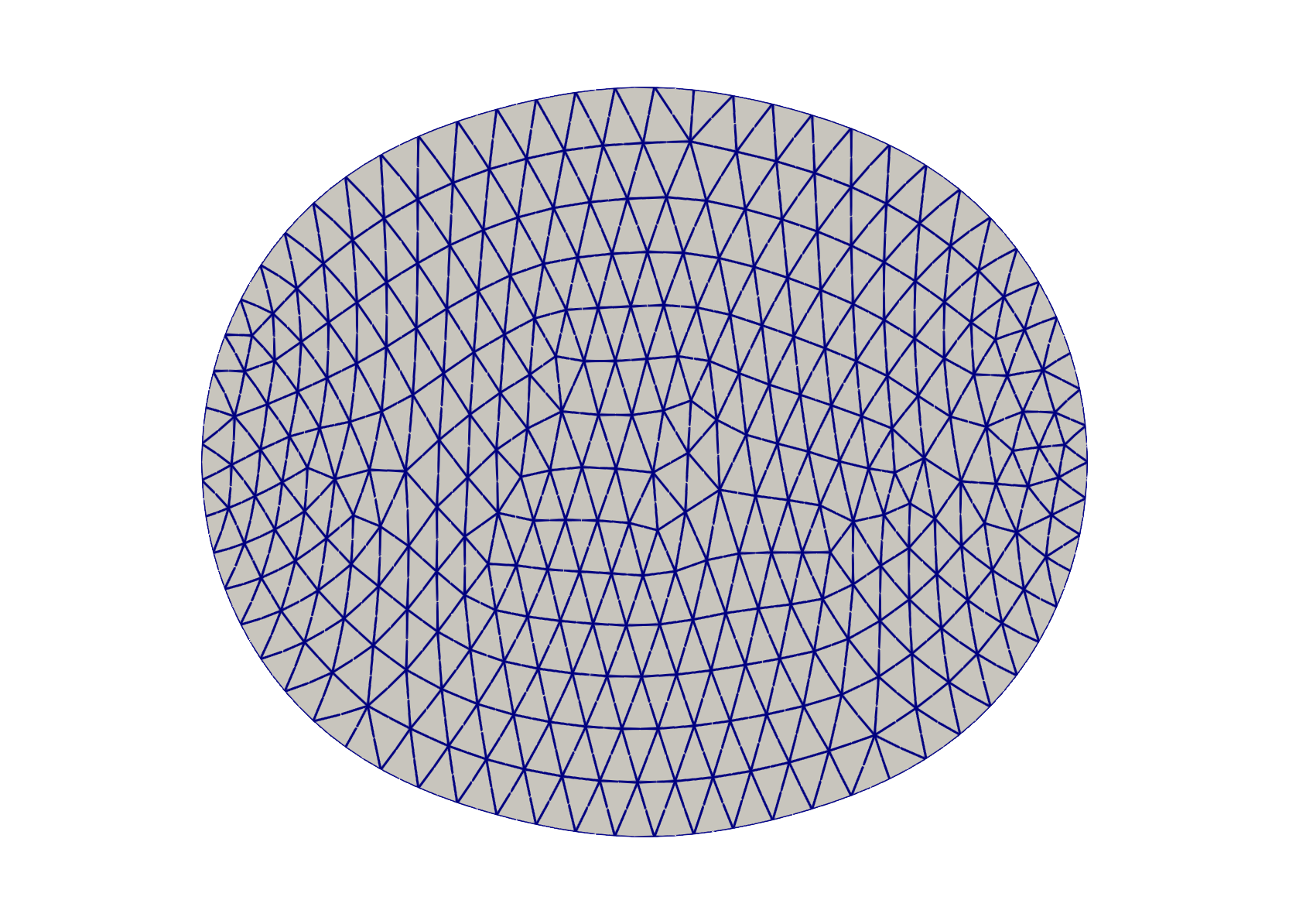}
    \vspace{-10pt}
    \caption{Computational mesh at $T=0$ (left) and $T=1$ (right) with mesh size $h=0.15$.}
    \label{fig:converge_mesh}
\end{figure}

\begin{table}[htbp]
\centering
\vspace{-10pt}
\caption{Spatial convergence rates for different variables.}
\label{tab:converge}
\begin{tabular}{ccccc}
\toprule
$h$ & $\| \bm{e}_\bmu\|_{\bm{H}^1}$ & Order & $\|e_p\|_{L^2}$ & Order \\
\midrule
0.15000 & 4.0248e-01 & -& 1.5230e-02 & - \\
0.07500 & 1.4654e-01 & 1.46 & 3.3822e-03 & 2.17  \\
0.03750 & 4.0760e-02 & 1.85 & 7.6361e-04 & 2.15  \\
0.01875 & 1.0569e-02 & 1.95 & 1.8837e-04 & 2.02  \\
\bottomrule
\end{tabular}
\end{table}

\subsection{Robustness with respect to $\varepsilon$}
\label{num:epsilon}

In this example, we investigate the robustness of the CTV method with respect to the artificial parameter $\varepsilon\ge 0$. 
The computational domain and coefficients $\mu, \gamma_0$ are identical to those in Subsection~\ref{num:accuracy}. 
We fix the time step at $\Delta t=10^{-2}$, and compare the temporal evolution of the total energy and mesh quality indicator $\sigma_{\max}^m$ for $\varepsilon=1$, $10^{-4}$, $10^{-8}$ and $0$, respectively.
As shown in Figure~\ref{fig:ex2-eps-consistency}, the total energies decrease monotonically in time for these values of $\varepsilon$, and the mesh quality indicator $\sigma_{\max}^m$ also remains stable when $\varepsilon$ approaches zero.

\begin{figure}[htbp]
    \centering
    \includegraphics[width=0.45\linewidth]{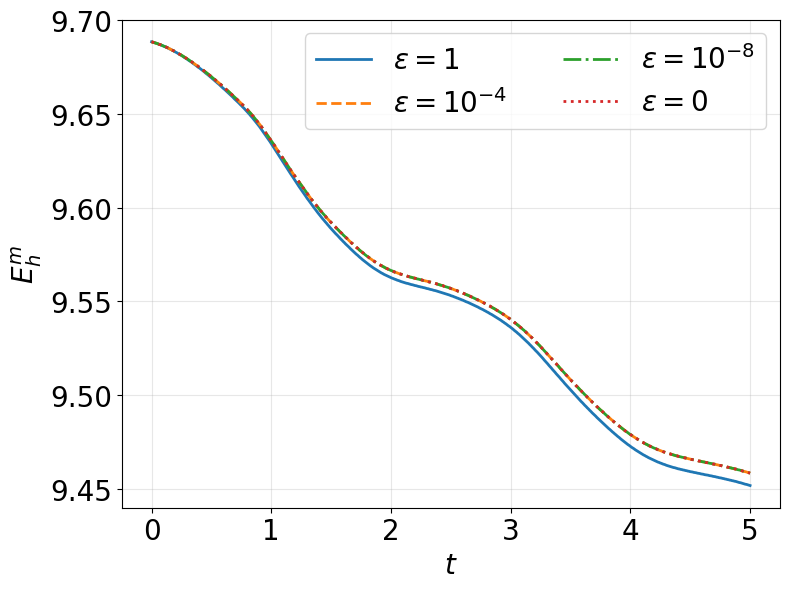}
    \includegraphics[width=0.45\linewidth]{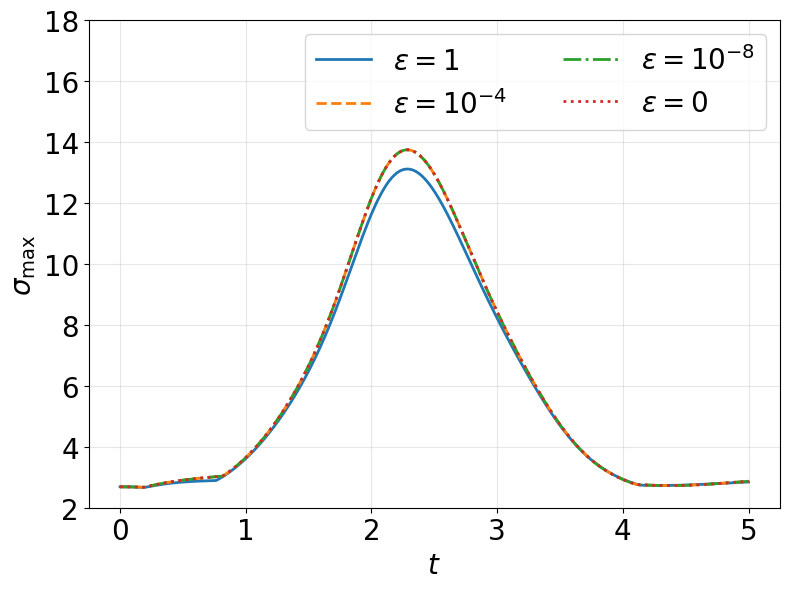}
    \vspace{-10pt}
    \caption{The evolution of the total energy and $\sigma_{\max}$ of the bulk, with $h=0.075$.}
    \label{fig:ex2-eps-consistency}
\end{figure}

\subsection{Six-leaf clover}
\label{num:robust_dt}

In this example, we compare the CTV method against the BGN method in terms of mesh quality. 
The initial domain is a six-leaf clover described in polar coordinates by
\[
    r(\theta)
    =  R\bigl(1+\epsilon\cos(k\theta)\bigr),
    \quad
    \theta\in[0,2\pi),
\]
with $R=1.2$, $\epsilon=0.2$, $k=6$.
The nonuniform curvature of the initial free surface generates a capillary force that drives the domain to oscillate and gradually relax toward a circular equilibrium.

Specifically, we compute the mesh velocity $\bmw_h^{m+1}$ on $\Gh^m$ via the BGN method and then extend it into the bulk via the MDR method. 
The tangential mesh redistribution becomes ineffective as the time-step size decreases, leading to a severe mesh distortion and eventually mesh tangling. In this example, we take a small time step $\Delta t = 10^{-4}$ and set the surface tension coefficient to $\gamma_0 = 1$. 

Figure~\ref{fig:example3-mu-1} shows the ALE meshes obtained with viscosity coefficient $\mu = 10^{-3}$ at $t=0.131,0.165,0.3$. The top panel shows the mesh generated by the BGN method, while the bottom panel shows that generated by the CTV method. The BGN mesh becomes tangled, causing the computation to terminate after $t=0.3$. In contrast, the CTV method yields a shape-regular mesh and supports long-time evolution.

    \begin{figure}[htbp]
	\centering
	\begin{subfigure}[b]{0.25\textwidth}
		\centering
		\includegraphics[width=\textwidth]{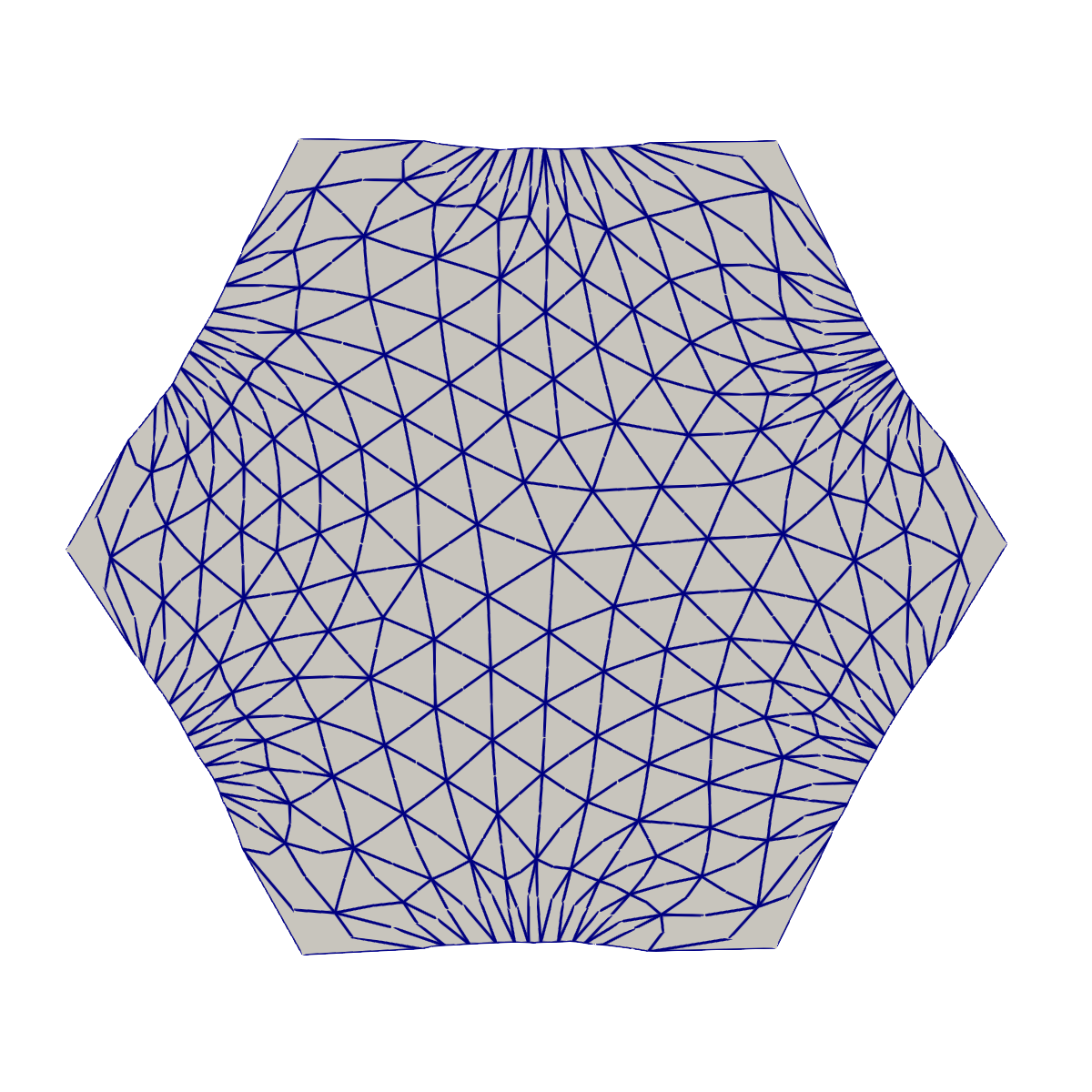}
		\caption{BGN: $t=0.131$}
	\end{subfigure}
	\begin{subfigure}[b]{0.25\textwidth}
		\centering
		\includegraphics[width=\textwidth]{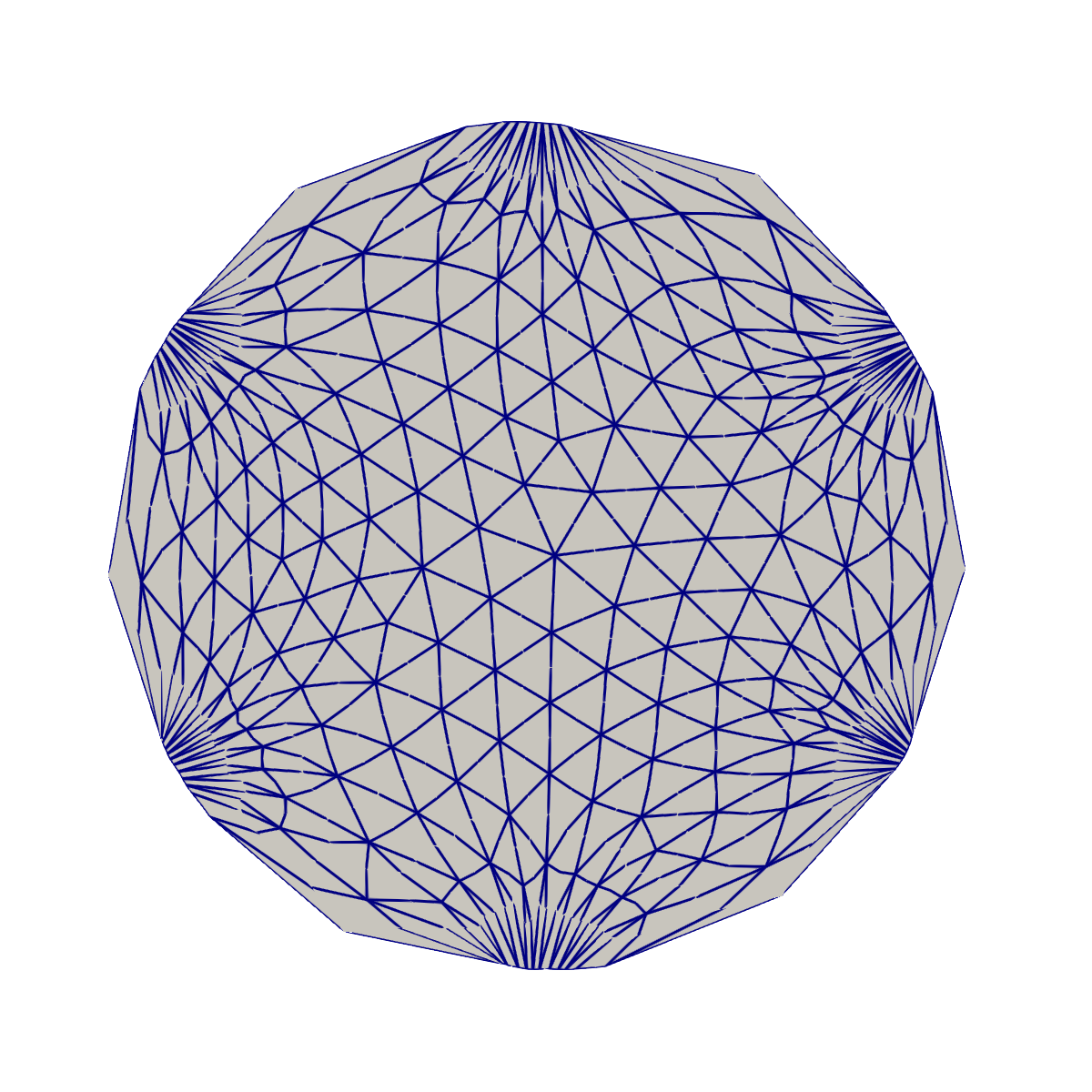}
		\caption{BGN: \(t = 0.165\)}
	\end{subfigure}
	\begin{subfigure}[b]{0.25\textwidth}
		\centering
		\includegraphics[width=\textwidth]{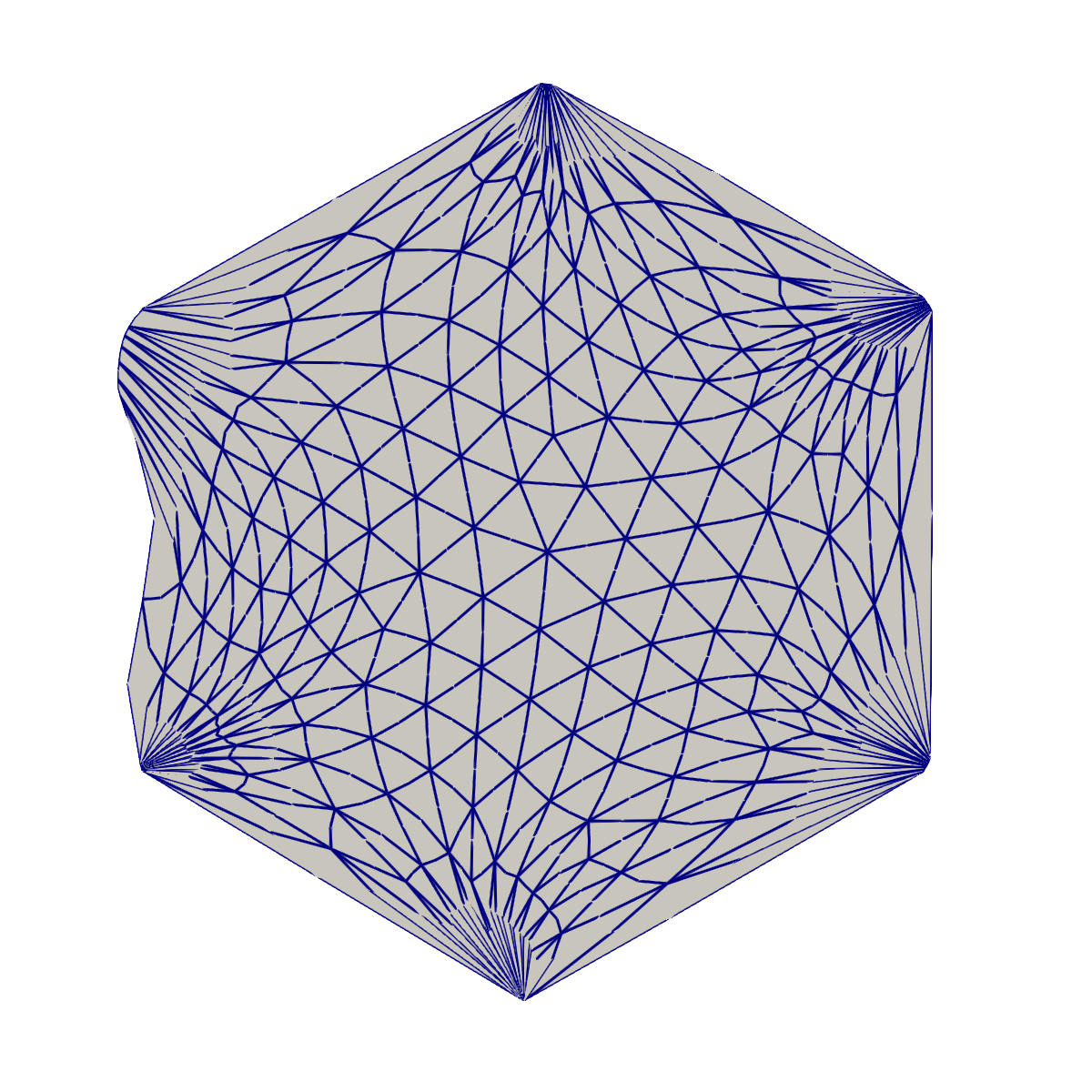}
		\caption{BGN: \(t = 0.3\)}
	\end{subfigure}
	\begin{subfigure}[b]{0.25\textwidth}
		\centering
		\includegraphics[width=\textwidth]{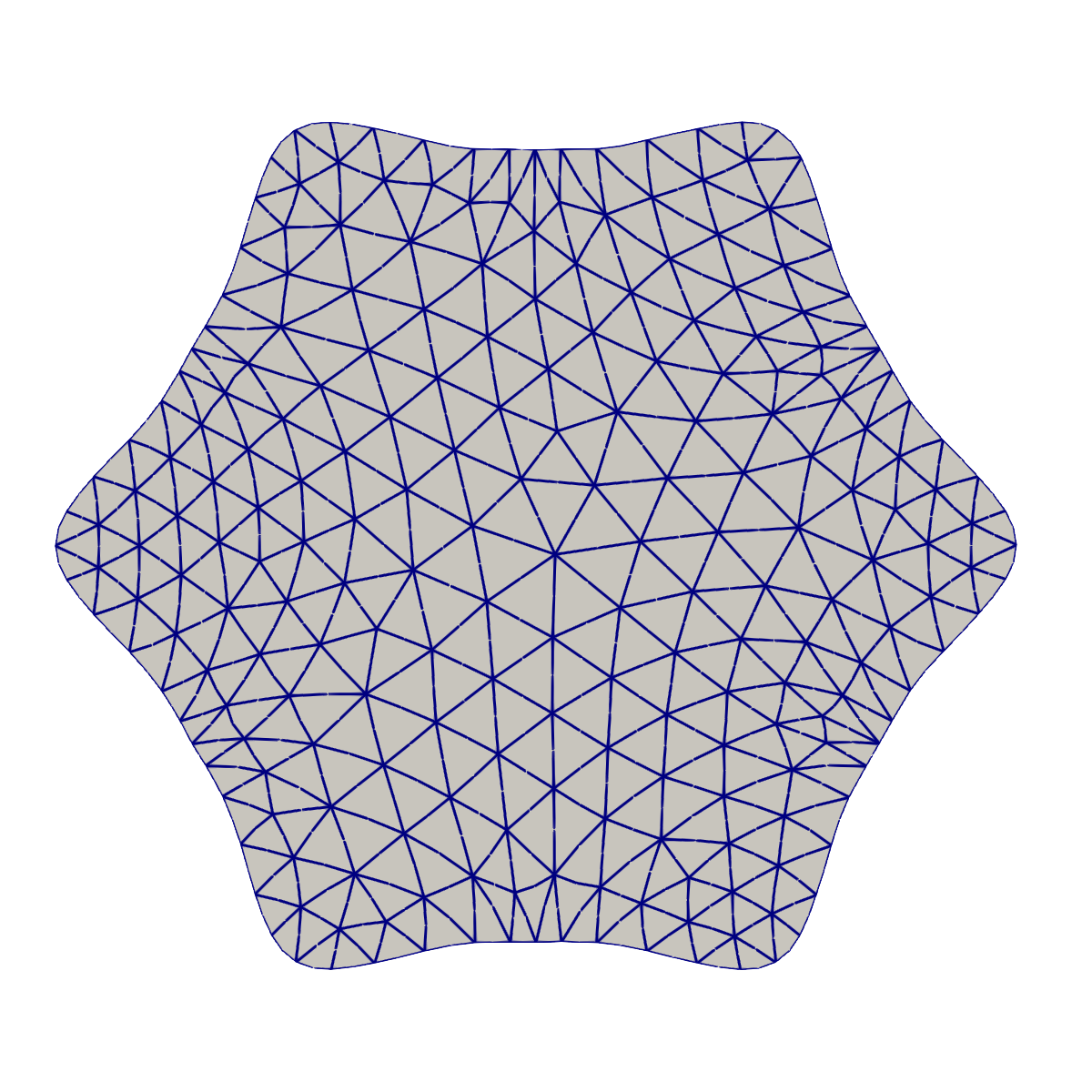}
		\caption{CTV: \(t = 0.131\)}
	\end{subfigure}
	\begin{subfigure}[b]{0.25\textwidth}
		\centering
		\includegraphics[width=\textwidth]{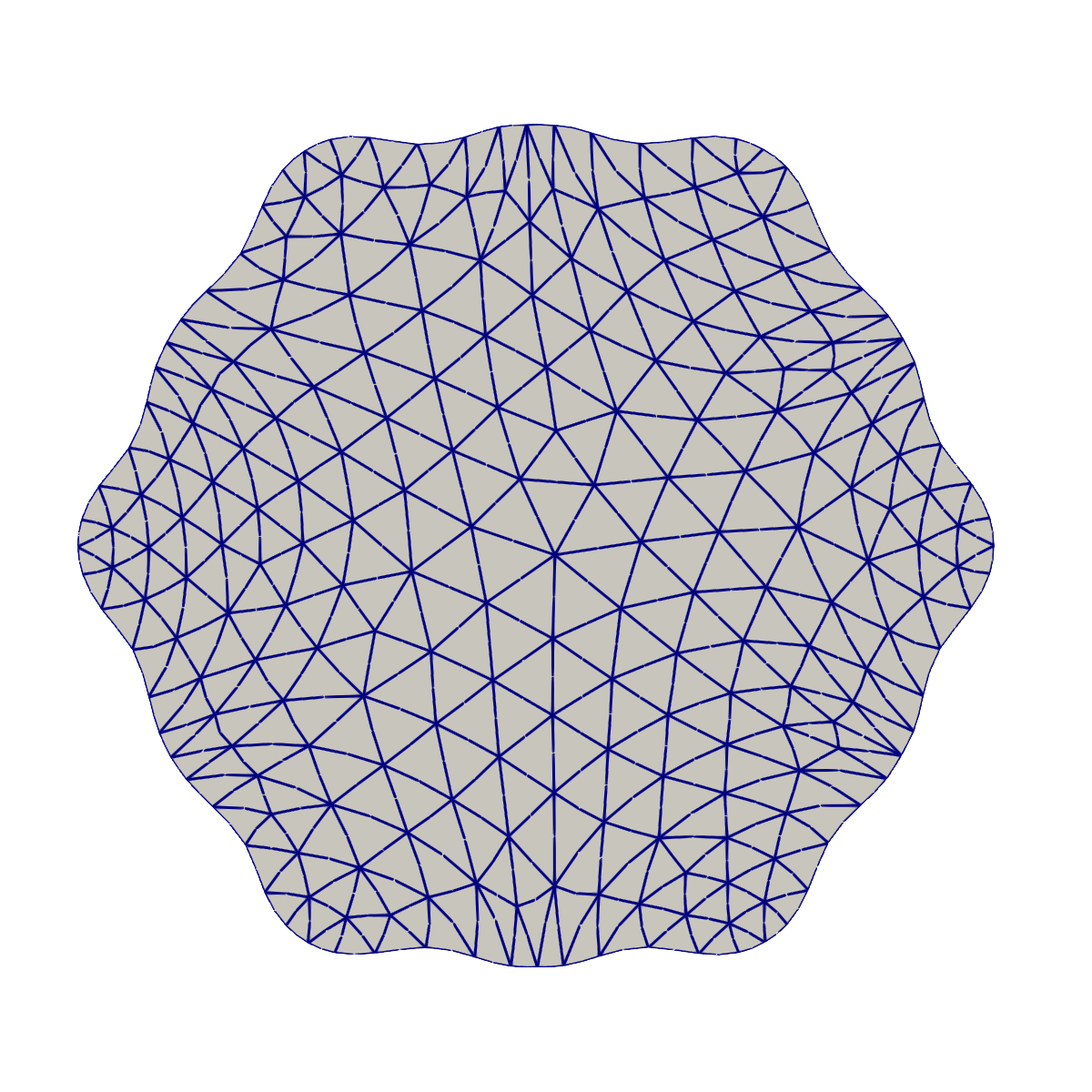}
		\caption{CTV: \(t = 0.165\)}
	\end{subfigure}
	\begin{subfigure}[b]{0.25\textwidth}
		\centering
		\includegraphics[width=\textwidth]{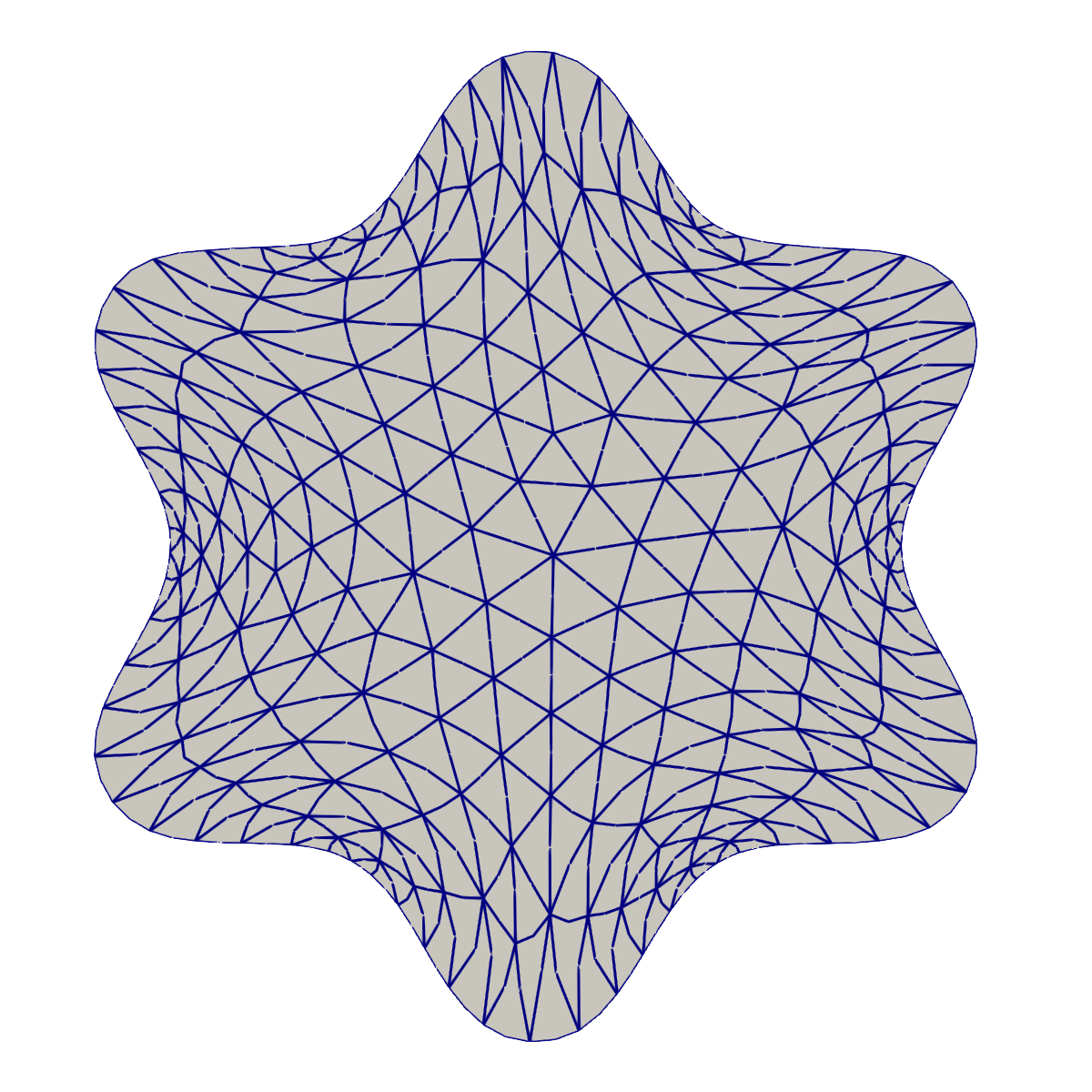}
		\caption{CTV: \(t = 0.3\)}
	\end{subfigure}
	\caption{ALE meshes at different time instants with $\Delta t = 10^{-4}$ and $\mu=10^{-3}$.}
	\label{fig:example3-mu-1}
\end{figure}

Next we increase the viscosity coefficient to $\mu=10^{-2}$.
Figure~\ref{fig:example3-mu-2} shows the ALE meshes obtained at times $t=0.158, 0.2, 0.3$. Similarly, the top panel shows the mesh generated by the BGN method, and the bottom panel shows that generated by the CTV method. In the BGN mesh, we find two concave quadrilateral elements: one at the upper-right concave part of the boundary and the other at the lower-right concave part. These two elements are illegal and cause the computation to terminate again after $t=0.3$. However, the CTV meshes are of good quality.
Moreover, from Figure~\ref{fig:example3-mu-3}, we observe that the CTV method remains stable even for $t=1$, with the mesh-quality indicator $\sigma_{\max}^m$ remaining around $10$ throughout the simulation, as shown in Figure~\ref{fig:example3_sigma}.

\begin{figure}[htbp]
	\centering
	\begin{subfigure}[b]{0.25\textwidth}
		\centering
		\includegraphics[width=\textwidth]{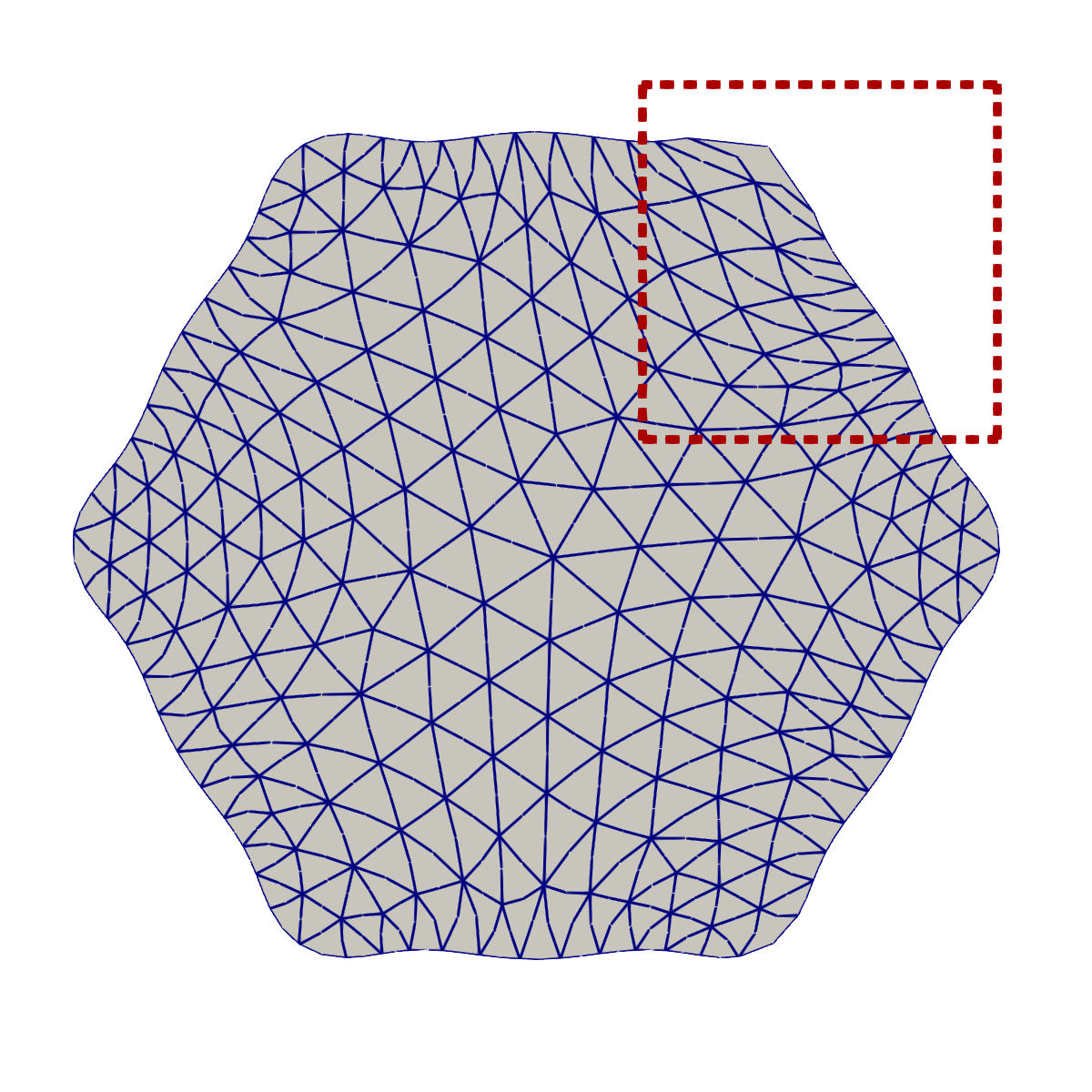}
		\caption{BGN: \(t=0.158\)}
	\end{subfigure}
	\begin{subfigure}[b]{0.25\textwidth}
		\centering
		\includegraphics[width=\textwidth]{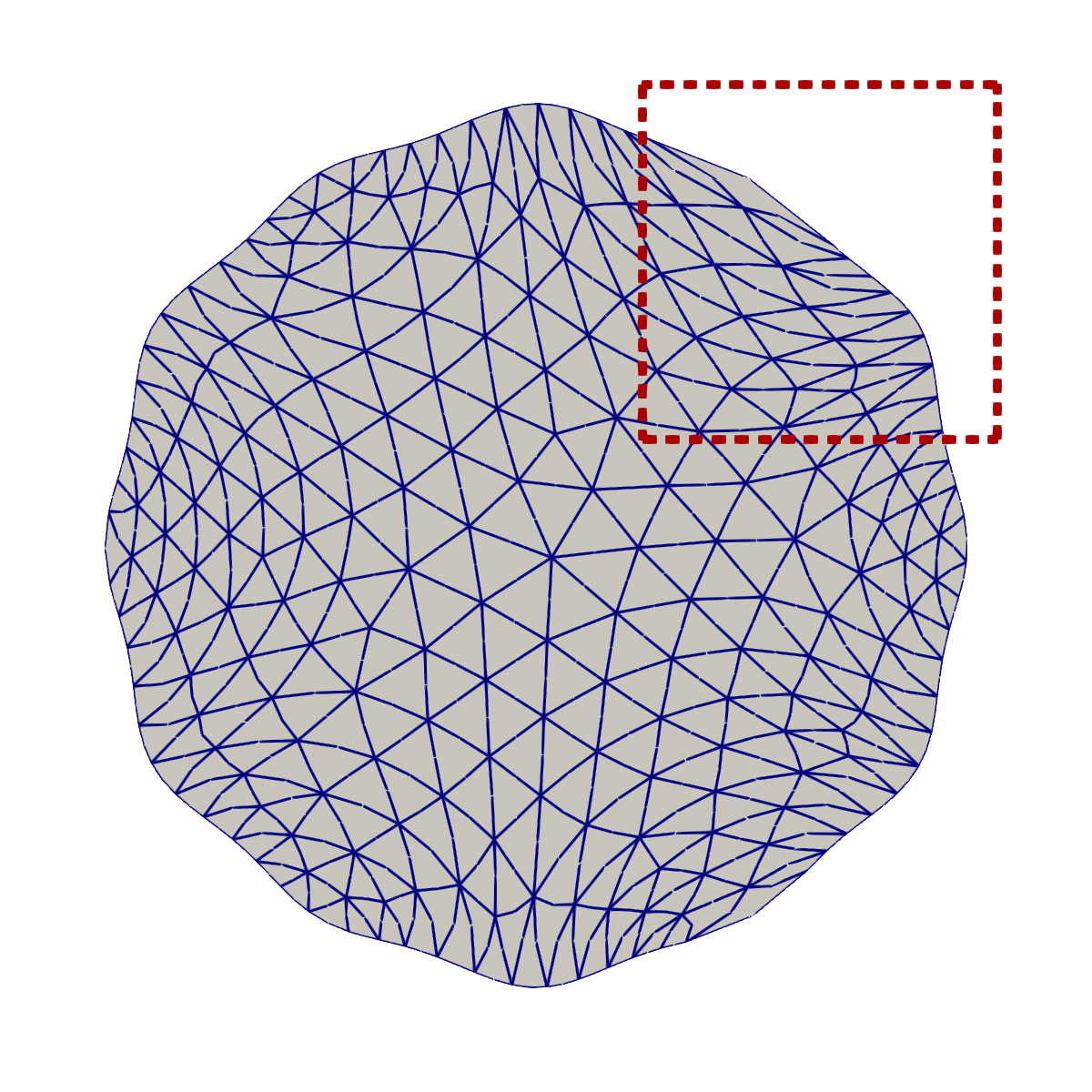}
		\caption{BGN: \(t = 0.2\)}
	\end{subfigure}
	\begin{subfigure}[b]{0.25\textwidth}
		\centering
		\includegraphics[width=\textwidth]{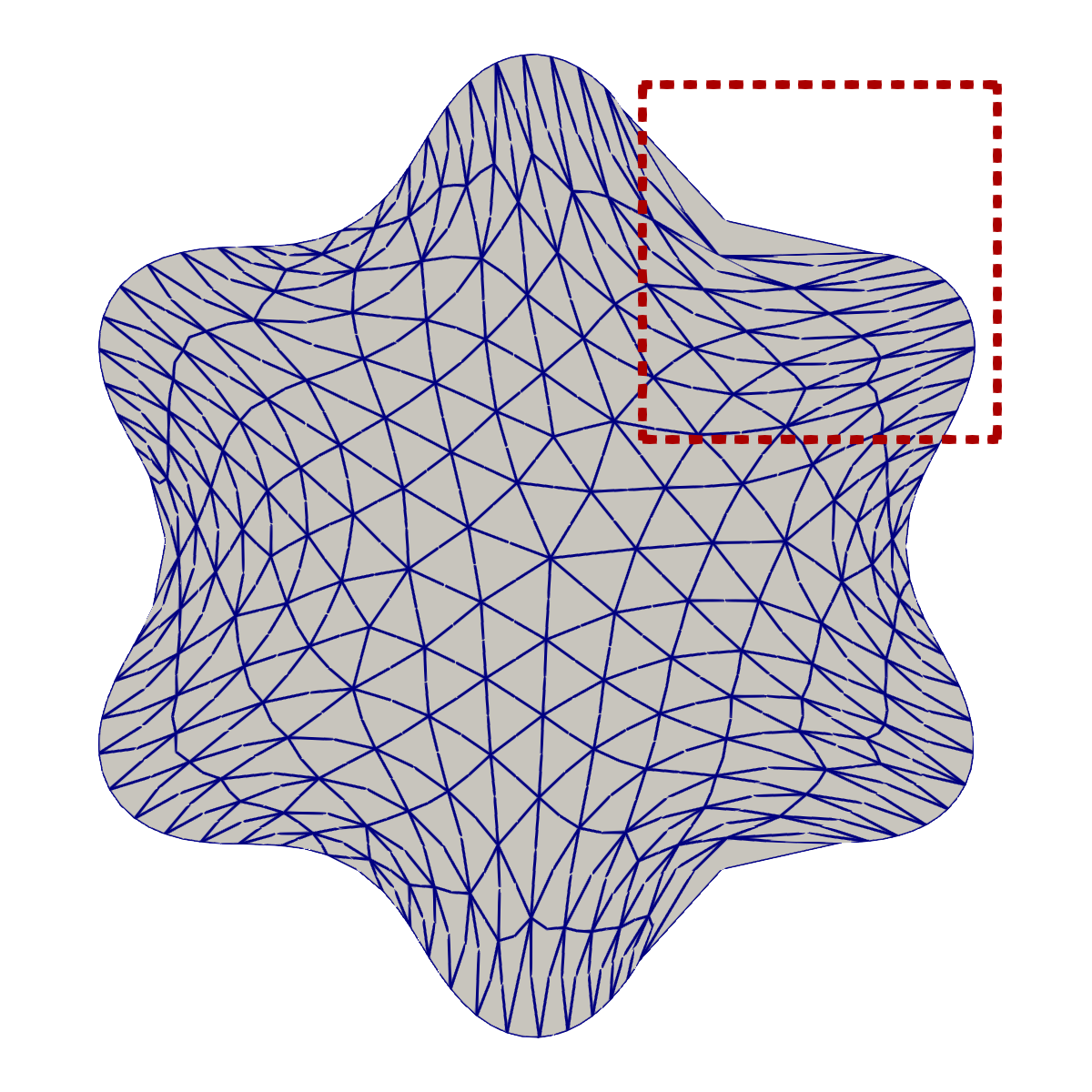}
		\caption{BGN: \(t = 0.3\)}
	\end{subfigure}
	\begin{subfigure}[b]{0.25\textwidth}
		\centering
		\includegraphics[width=\textwidth]{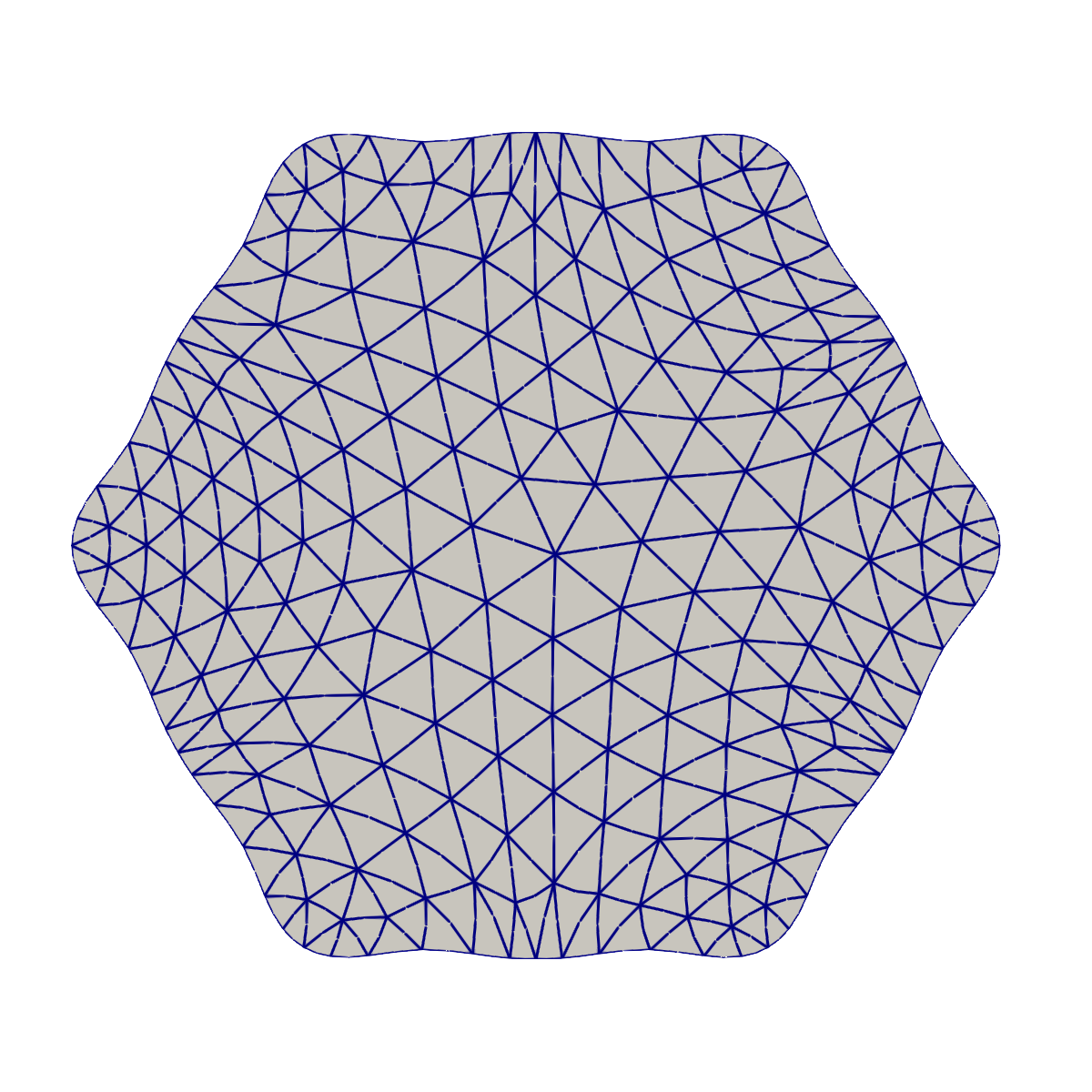}
		\caption{CTV: \(t = 0.158\)}
	\end{subfigure}
	\begin{subfigure}[b]{0.25\textwidth}
		\centering
		\includegraphics[width=\textwidth]{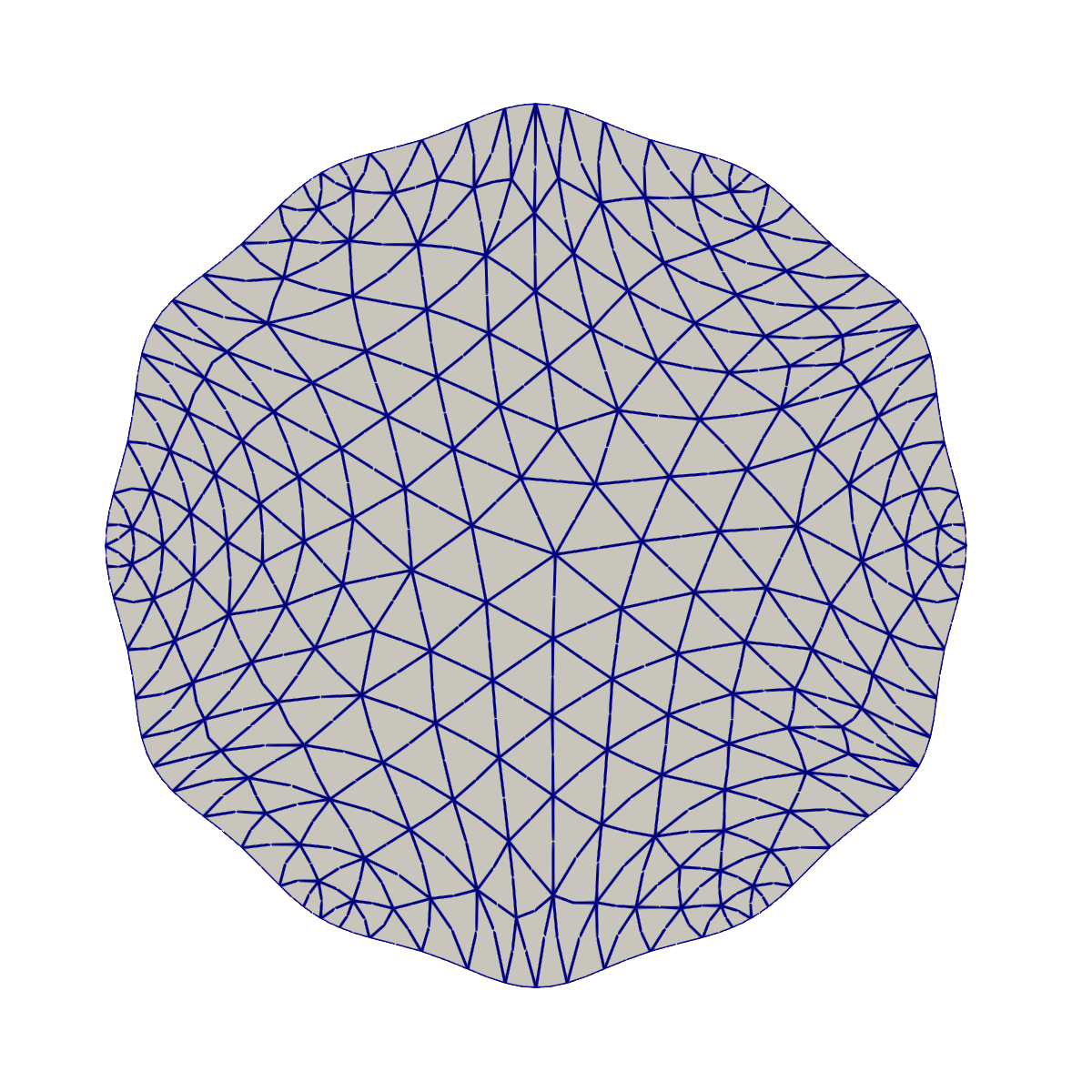}
		\caption{CTV: \(t = 0.2\)}
	\end{subfigure}
	\begin{subfigure}[b]{0.25\textwidth}
		\centering
		\includegraphics[width=\textwidth]{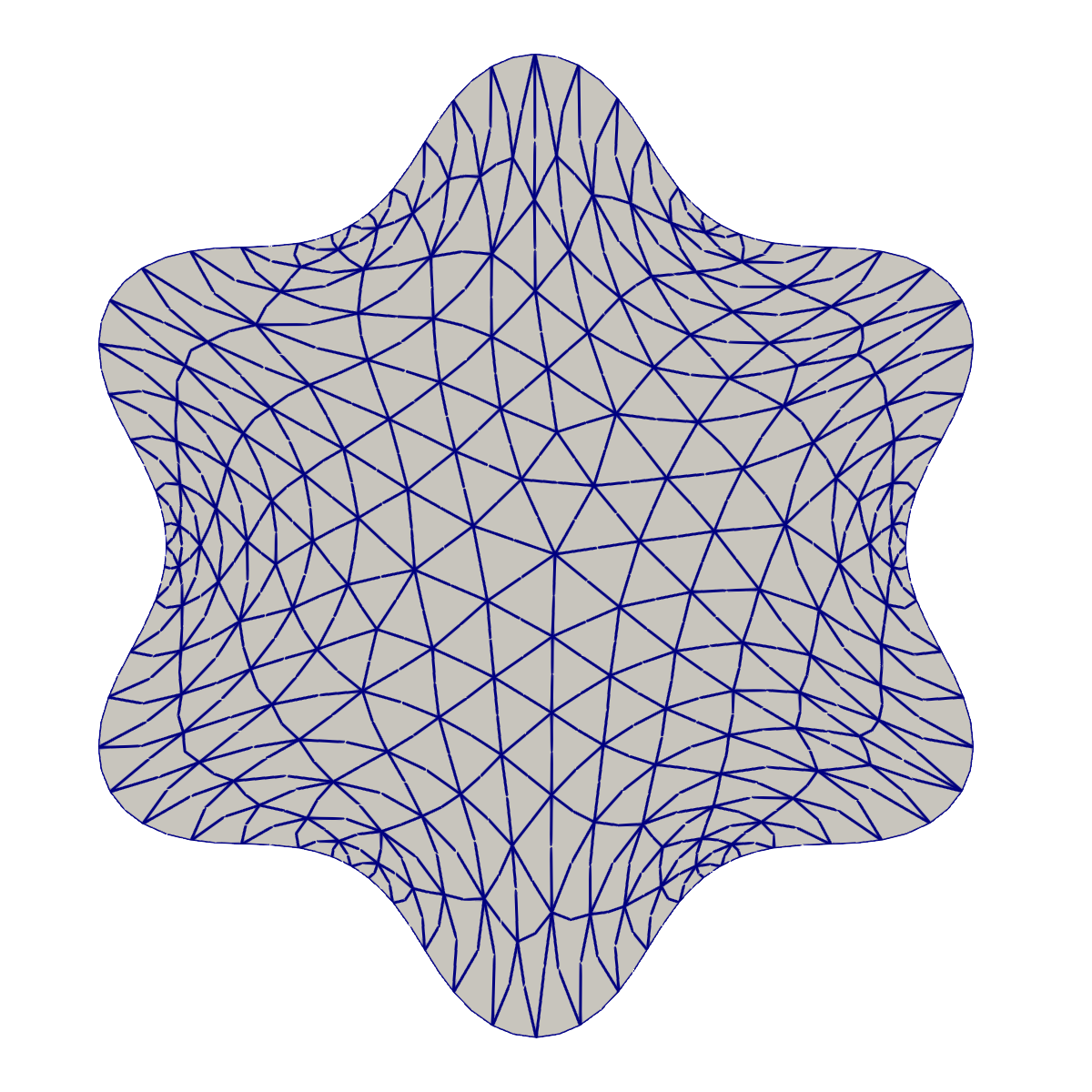}
		\caption{CTV: \(t = 0.3\)}
	\end{subfigure}
	\caption{ALE meshes at different time instants with $\Delta t = 10^{-4}$ and $\mu=10^{-2}$.}
	\label{fig:example3-mu-2}
\end{figure}

\begin{figure}[htbp]
	\centering
	\begin{subfigure}[b]{0.25\textwidth}
		\centering
		\includegraphics[width=\textwidth]{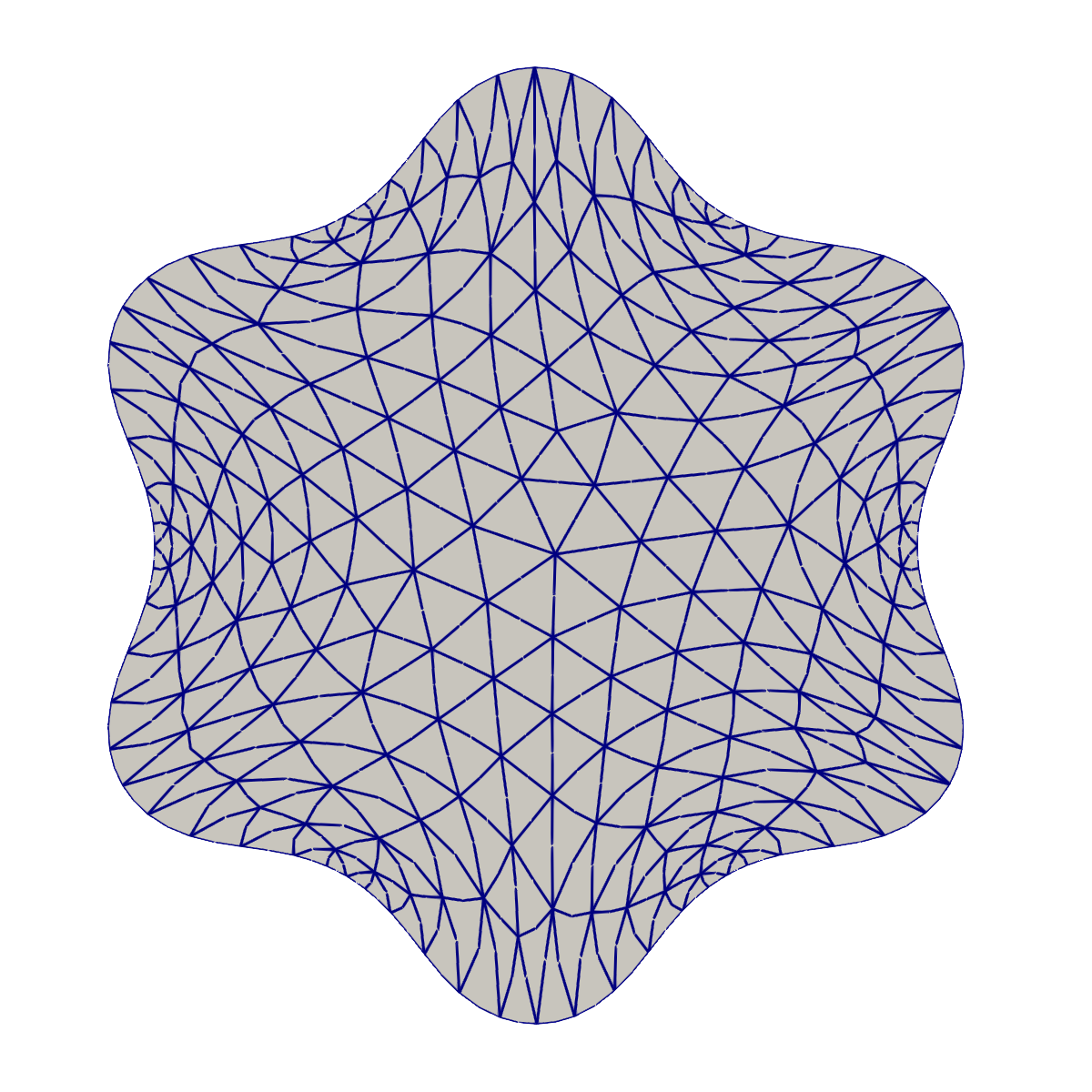}
		\caption{CTV: \(\mu = 10^{-2}\)}
	\end{subfigure}
	\begin{subfigure}[b]{0.25\textwidth}
		\centering
		\includegraphics[width=\textwidth]{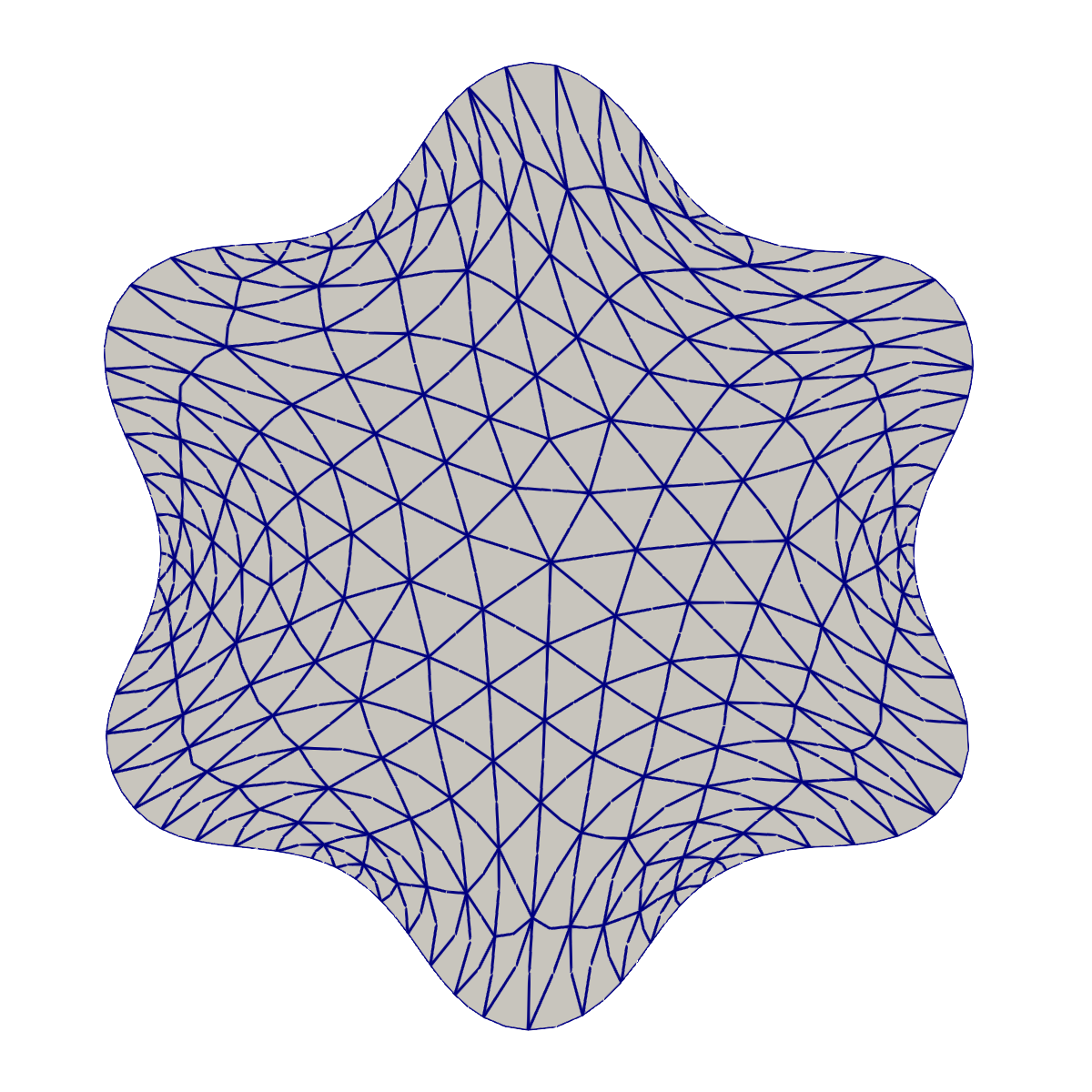}
		\caption{CTV: \(\mu = 10^{-3}\)}
	\end{subfigure}
	\caption{ALE meshes at $t=1$ with $\Delta t = 10^{-4}$.}
	\label{fig:example3-mu-3}
\end{figure}

\begin{figure}[htbp]
	\centering
	\includegraphics[width=0.5\textwidth]{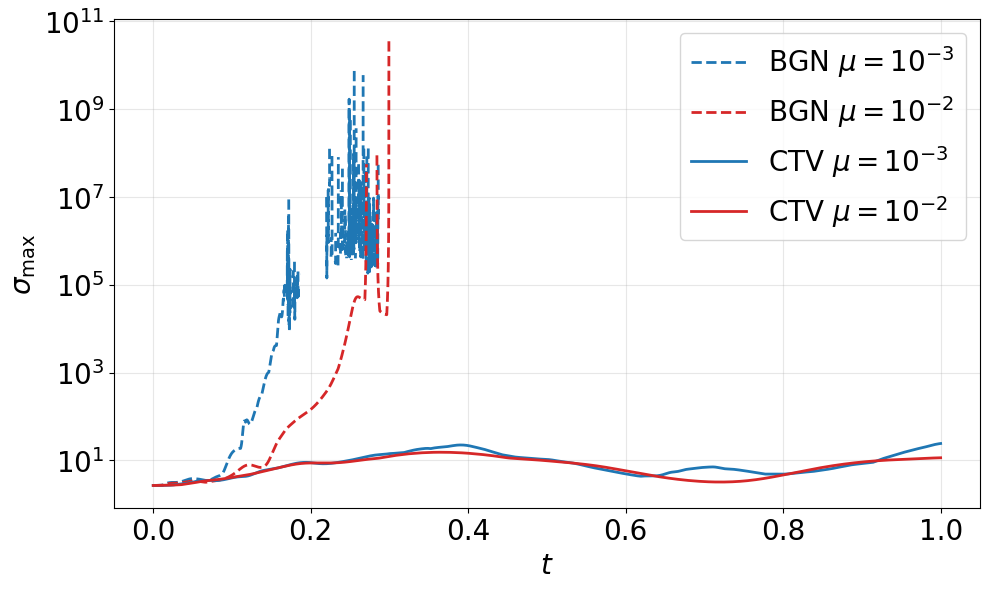}
\caption{Mesh quality indicator $\sigma_{\max}^m$ for the BGN and CTV methods with $\Delta t = 10^{-4}$. 
}
\vspace{-15pt}
\label{fig:example3_sigma}
\end{figure}

\subsection{Three-dimensional damping of circularity}
\label{num:3D}

In this experiment, we compute the three-dimensional (3D) damping of circularity. 
Following \cite{Bansch2001}, the initial domain is an ellipsoid with semi-axis lengths $(1.2,0.8,0.8)$, and its evolution is driven solely by surface tension.
We define the circularity for the 3D case as
\[
C_h^m:=(36\pi)^{1/3}|\Oh^m|^{2/3} |\Gh^m|^{-1}.
\]

We set $\mu=0.01$, $\Delta t=10^{-2}$, and the final time $T=5$. 
The experiment is performed for several values of the surface tension coefficient: $\gamma_0=1,2,4,6,8,10$.
Figure~\ref{fig:Ellipsoid_energy_circularity} summarizes the damping behavior for different values of $\gamma_0$. 
We observe that the circularity oscillates and that a larger surface tension coefficient leads to a higher oscillation frequency.
To quantify this dependence, let $t_{p_0}<t_{p_1}<\cdots<t_{p_n}$ denote successive local minima of the circularity. 
We define the mean oscillation frequency as
$f=n/(t_{p_n}-t_{p_0})$.
The right panel of Figure~\ref{fig:Ellipsoid_energy_circularity} shows the computed frequency versus $\gamma_0$. 
The numerical results closely follow the scaling $f\sim\gamma_0^{1/2}$,
which is consistent with the capillary oscillation behavior reported in the literature \cite{Bansch2001}.

Figure~\ref{fig:Ellipsoid3D} shows the mesh velocity for $\gamma_0=1$ at several representative times. 
The mesh velocity vectors clearly illustrate the alternating deformation of the domain driven by surface tension and its gradual relaxation toward a spherical shape under viscous dissipation.

\begin{figure}[htbp]
    \centering    
    \begin{subfigure}[b]{0.45\textwidth}
        \centering
        \includegraphics[width=\textwidth]{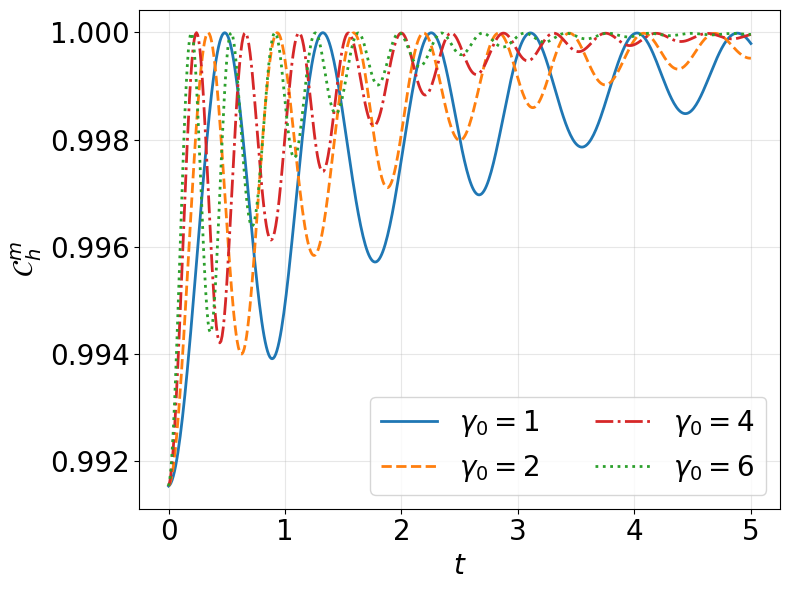} 
    \end{subfigure}
    \hfill 
    \begin{subfigure}[b]{0.45\textwidth}
        \centering
        \includegraphics[width=\textwidth]{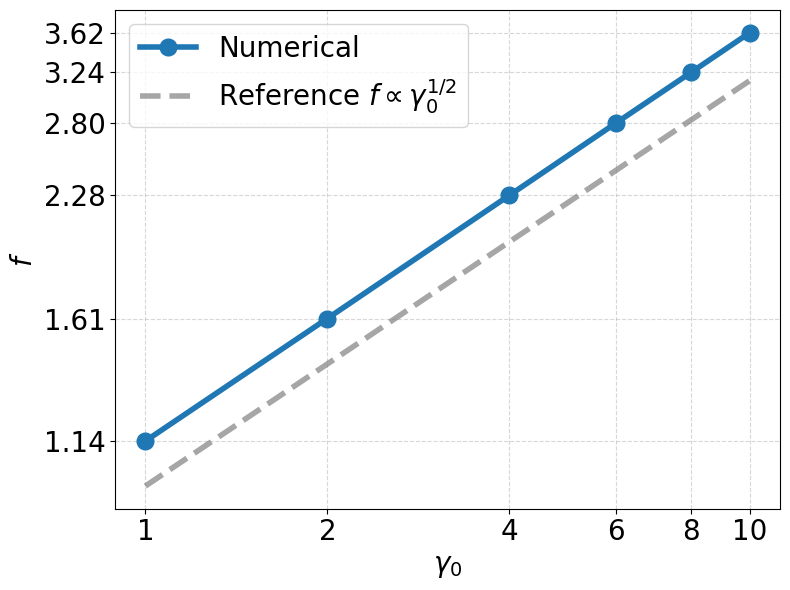} 
    \end{subfigure}
    \vspace{-10pt}
\caption{
Damping test with different $\gamma_0$.
Left: the circularity $\mathcal{C}_h^m$.
Right: mean oscillation frequency $f$ versus $\gamma_0$.} 
\label{fig:Ellipsoid_energy_circularity}
\end{figure}

\begin{figure}[htbp]
    \centering    
    \begin{subfigure}[b]{0.3\textwidth}
        \centering
        \includegraphics[width=\textwidth]{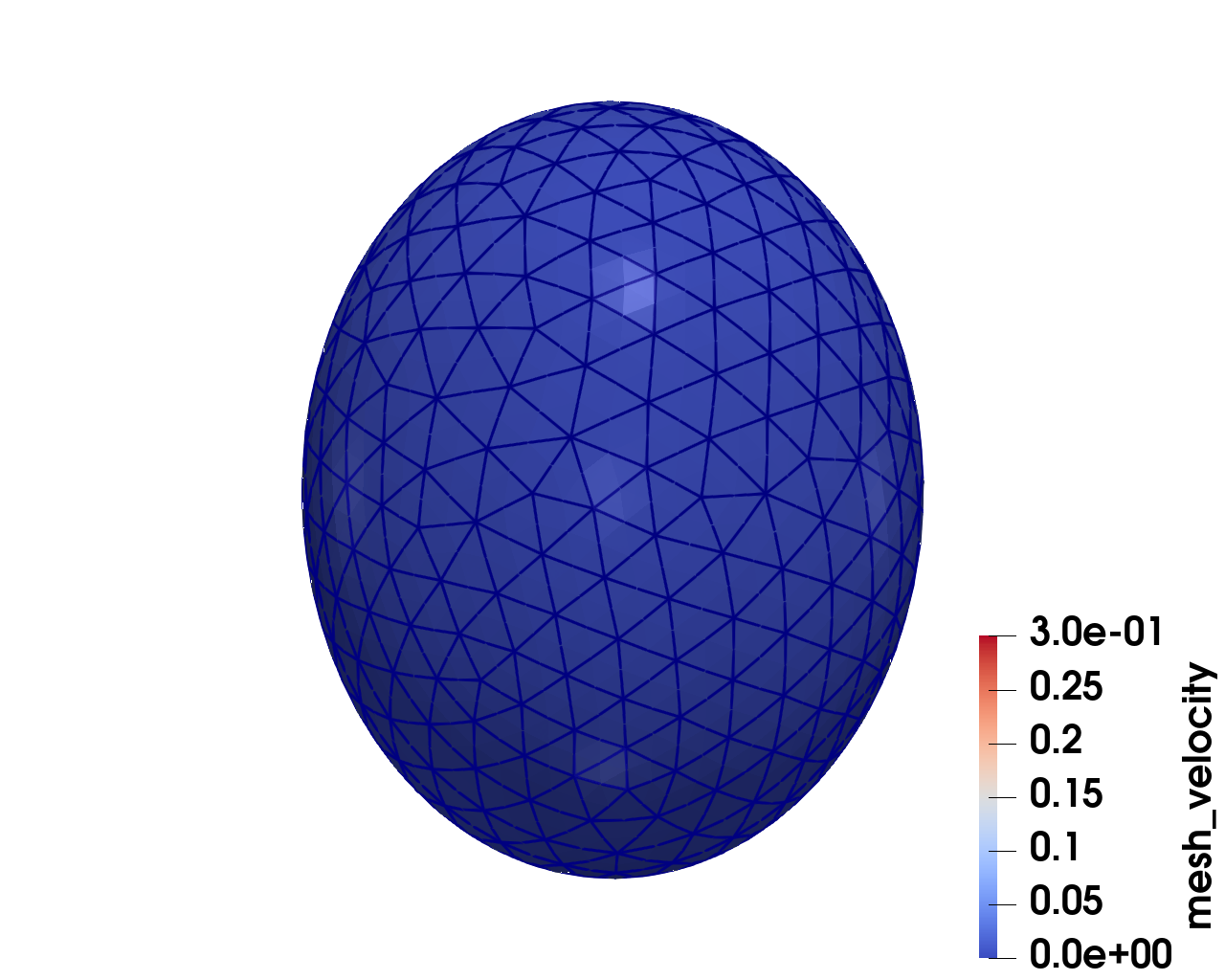} 
        \caption{$t=0.0$}
    \end{subfigure}
    \hfill 
    \begin{subfigure}[b]{0.3\textwidth}
        \centering
        \includegraphics[width=\textwidth]{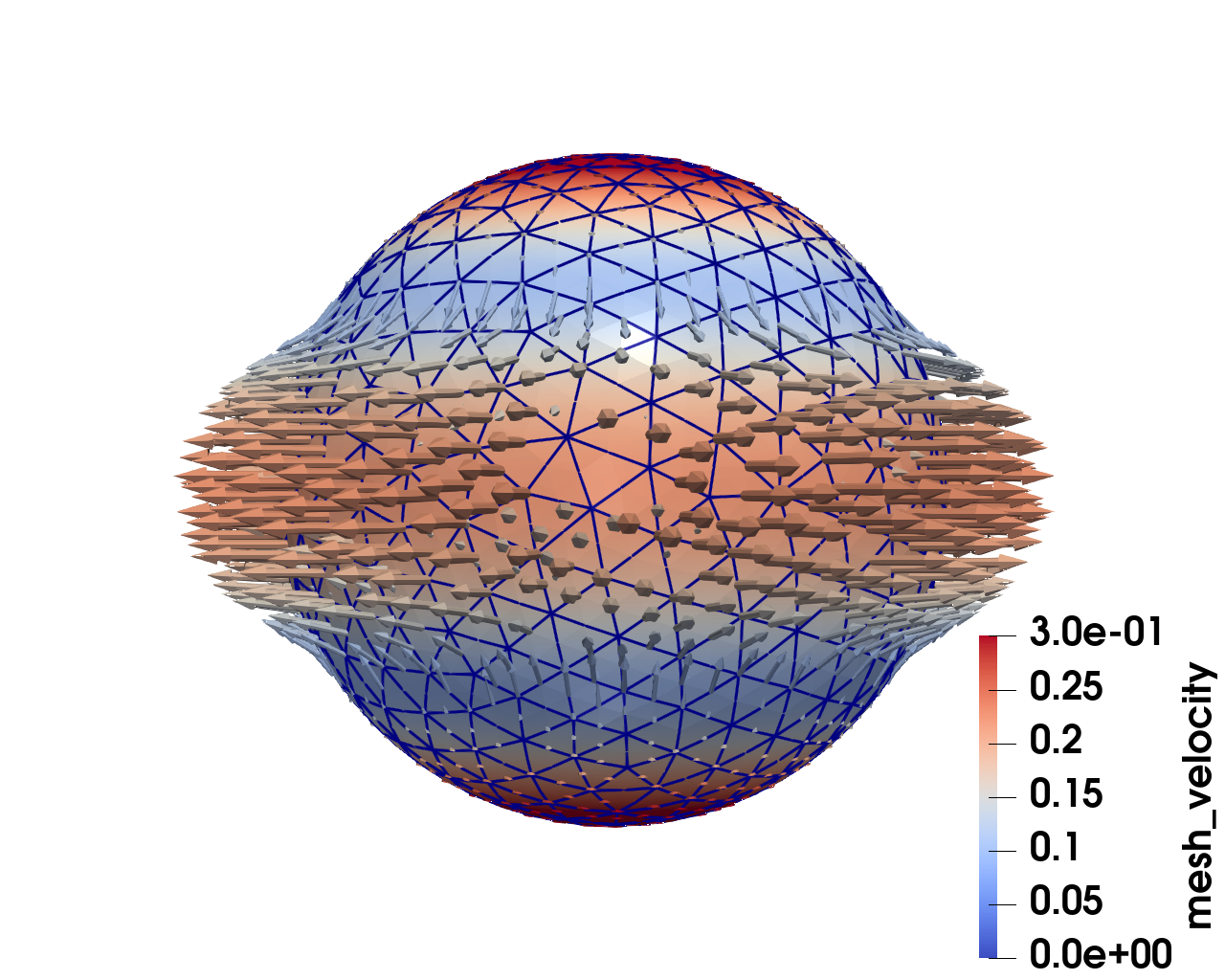} 
        \caption{$t=0.48$}
    \end{subfigure}
    \hfill
    \begin{subfigure}[b]{0.3\textwidth}
        \centering
        \includegraphics[width=\textwidth]{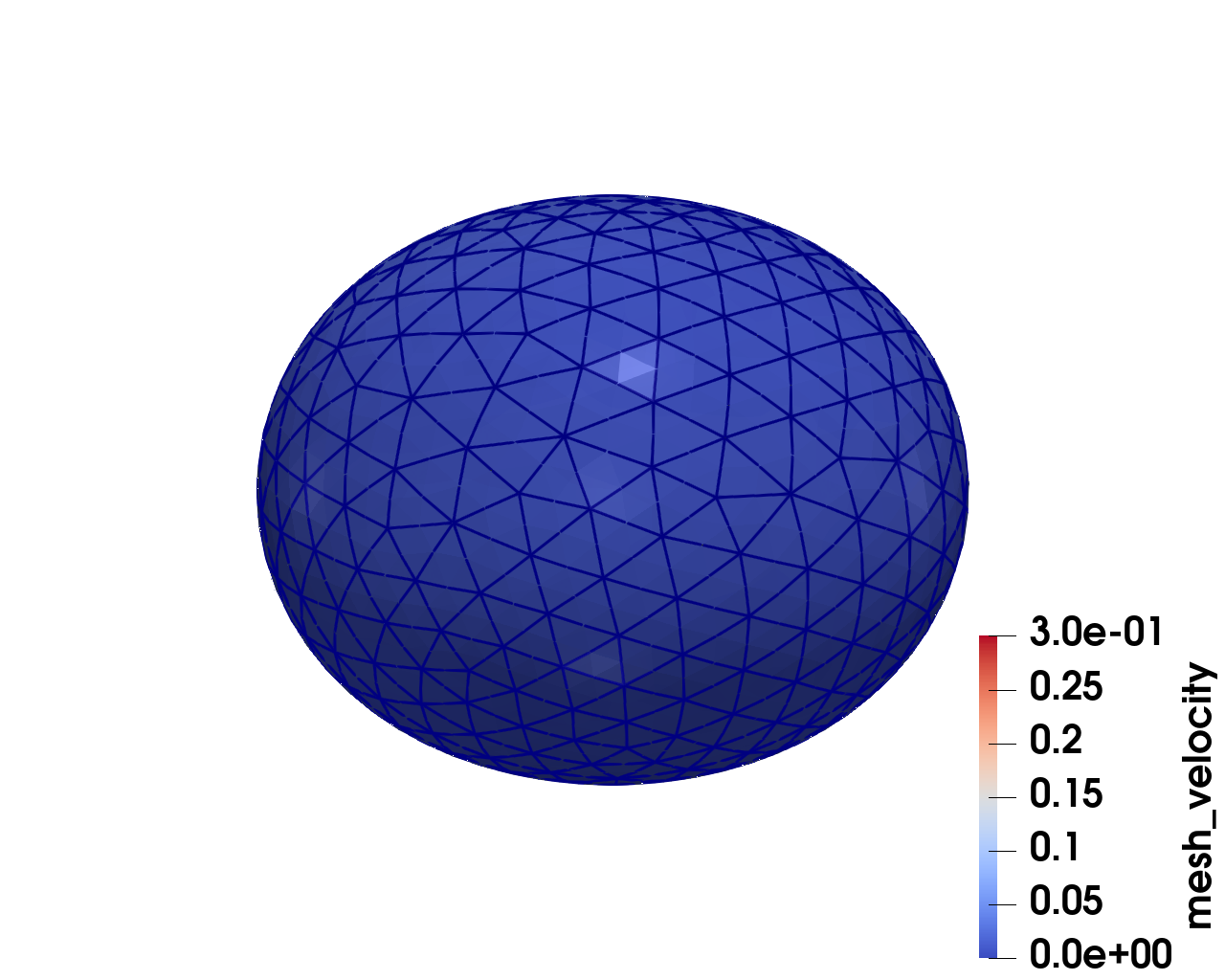} 
        \caption{$t=0.89$}
    \end{subfigure}
    % \hfill 
    \begin{subfigure}[b]{0.3\textwidth}
        \centering
        \includegraphics[width=\textwidth]{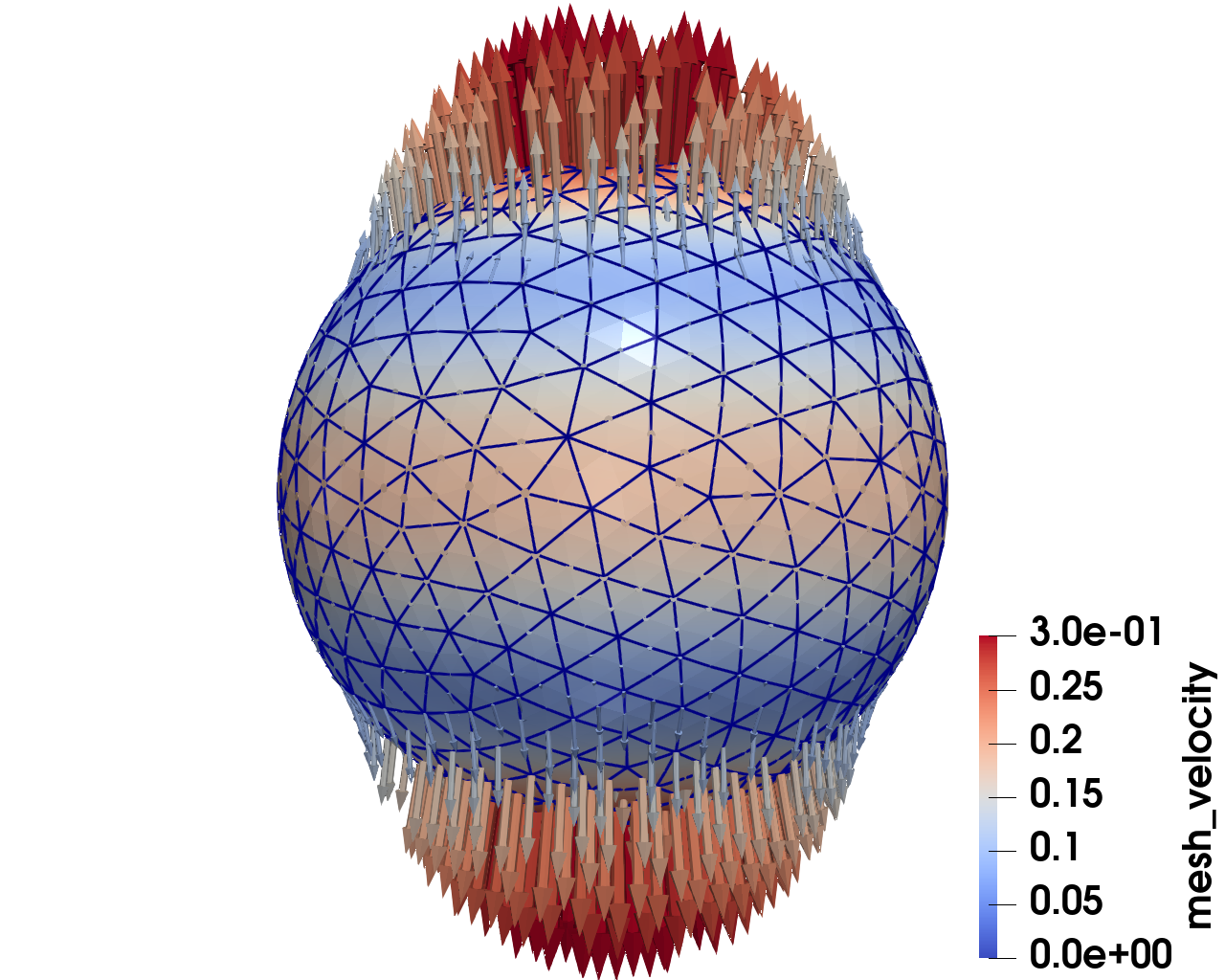} 
        \caption{$t=1.32$}
    \end{subfigure}
    \hfill 
    \begin{subfigure}[b]{0.3\textwidth}
        \centering
        \includegraphics[width=\textwidth]{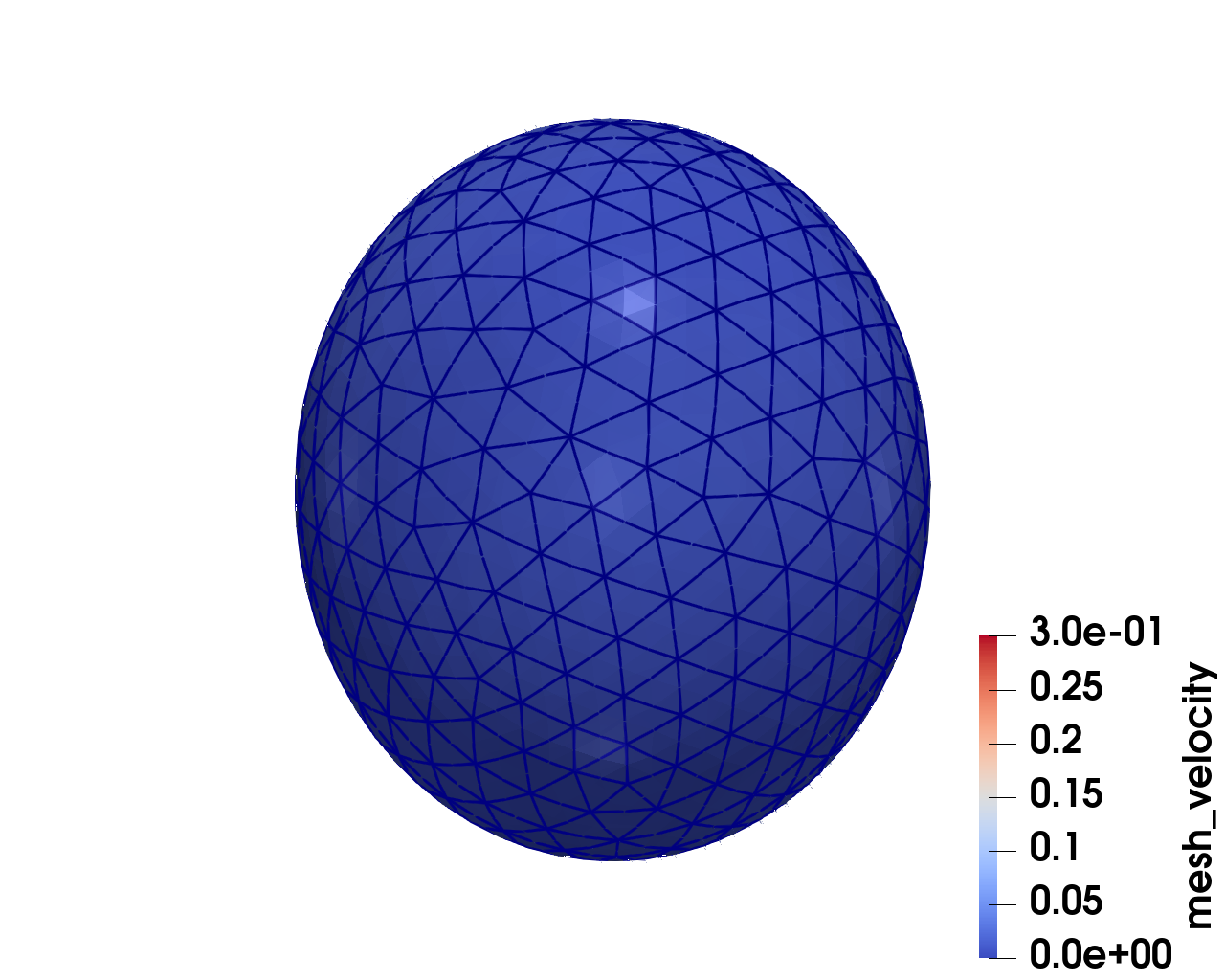} 
        \caption{$t=1.77$}
    \end{subfigure}
    \hfill
    \begin{subfigure}[b]{0.3\textwidth}
        \centering
        \includegraphics[width=\textwidth]{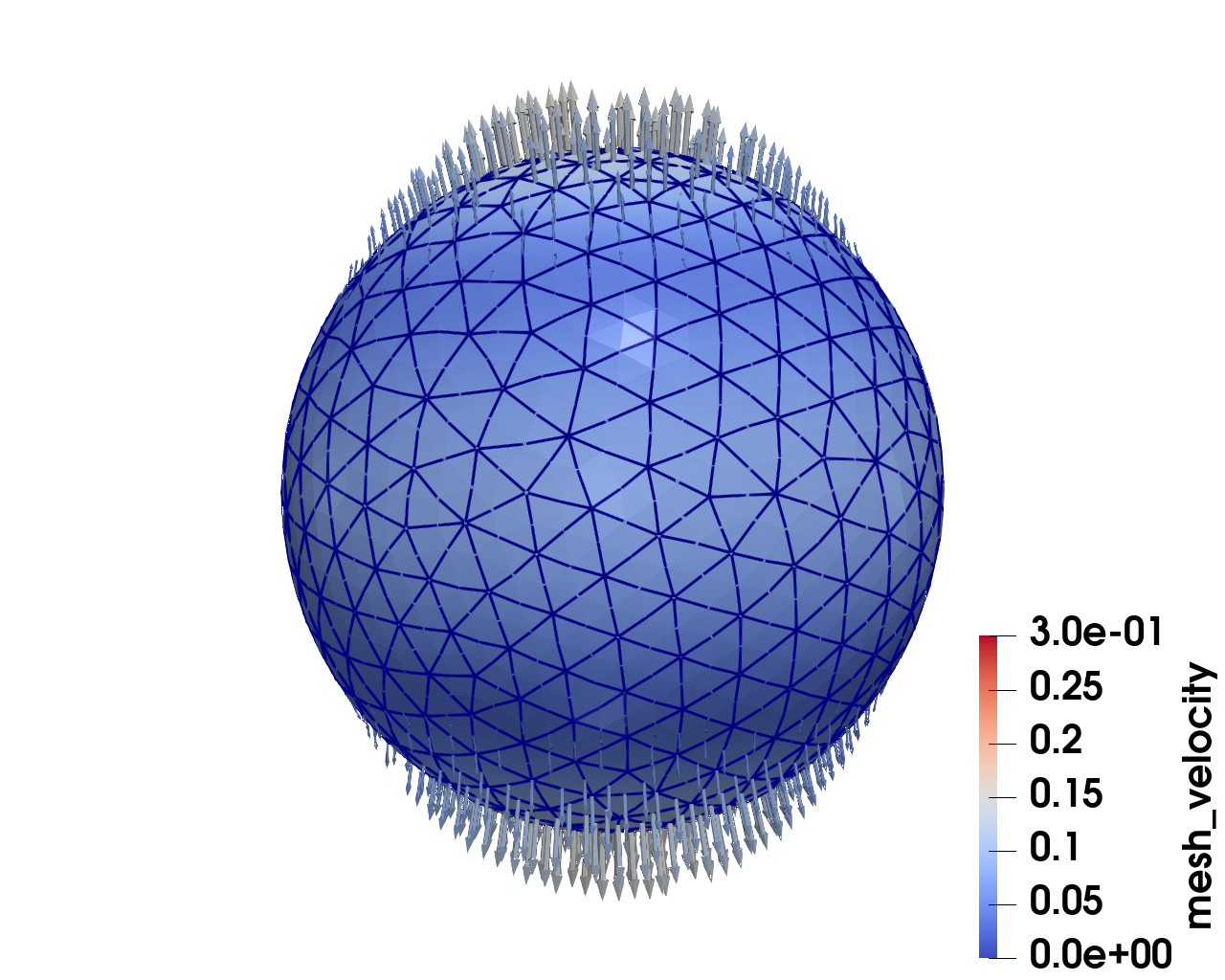} 
        \caption{$t=5.00$}
    \end{subfigure}
    \vspace{-10pt}
\caption{Mesh velocity fields at different time instants with $\Delta t=10^{-2}$ and $\gamma_0 = 1$. } 
\vspace{-10pt}
\label{fig:Ellipsoid3D}
\end{figure}

\subsection{Three-dimensional flow in a popcorn-shaped domain}

We further consider a popcorn-shaped droplet
where the initial domain is 
\[
\Omega^0=\{\bm x\in\mathbb R^3:\phi(\bm x)<0\},
\qquad
\phi(\bm x)
=|\bm x|^2-r_0^2
-\sum_{j=0}^{11}
\exp\!\left(-\frac{|\bm x-\bm p_j|^2}{\sigma_0^2}\right),
\]
with $r_0=0.5$ and $\sigma_0=0.2$.
The twelve bump centres $\{ \bm p_j: j=0,\cdots,11\}$ lie at the vertices of a regular icosahedron inscribed in the sphere of radius $r_0$ centred at the origin.
Specifically, setting $s=r_0/\sqrt{5}$, we take
\[
\begin{aligned}
\bm p_j
&=s\left(
2\cos\frac{2j\pi}{5},
2\sin\frac{2j\pi}{5},1
\right)^\top,
&&j=0,\cdots,4,\\
\bm p_{j+5}
&=s\left(
2\cos\frac{(2j-1)\pi}{5},
2\sin\frac{(2j-1)\pi}{5},-1
\right)^\top,
&&j=0,\cdots,4,\\
\bm p_{10}&=(0,0,r_0)^\top,
\qquad
\bm p_{11}=(0,0,-r_0)^\top.
\end{aligned}
\]

In this test, we set $\mu=0.01$, $\gamma_0=1$ and $\Delta t=10^{-3}$.
Figure~\ref{fig:Popcorn3D} shows the computed meshes and the corresponding mesh velocity vectors at different times.
The meshes retain good quality at these stages of the evolution, despite the nonuniform curvature of the initial interface.

\begin{figure}[htbp]
    \centering    
    \begin{subfigure}[b]{0.4\textwidth}
        \centering
        \includegraphics[width=\textwidth]{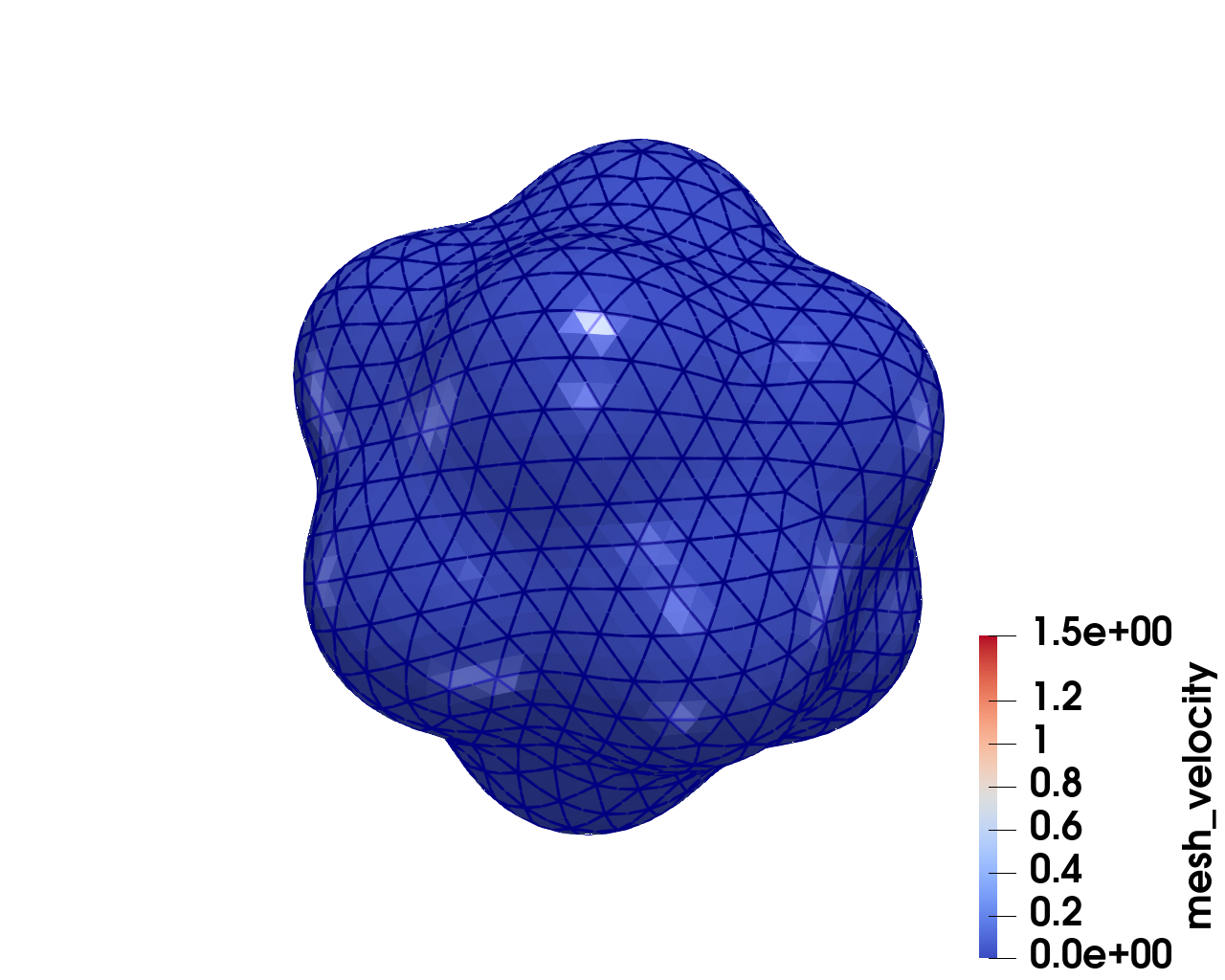} 
        \caption{$t=0.0$}
    \end{subfigure}
    \hspace{10pt} 
    \begin{subfigure}[b]{0.4\textwidth}
        \centering
        \includegraphics[width=\textwidth]{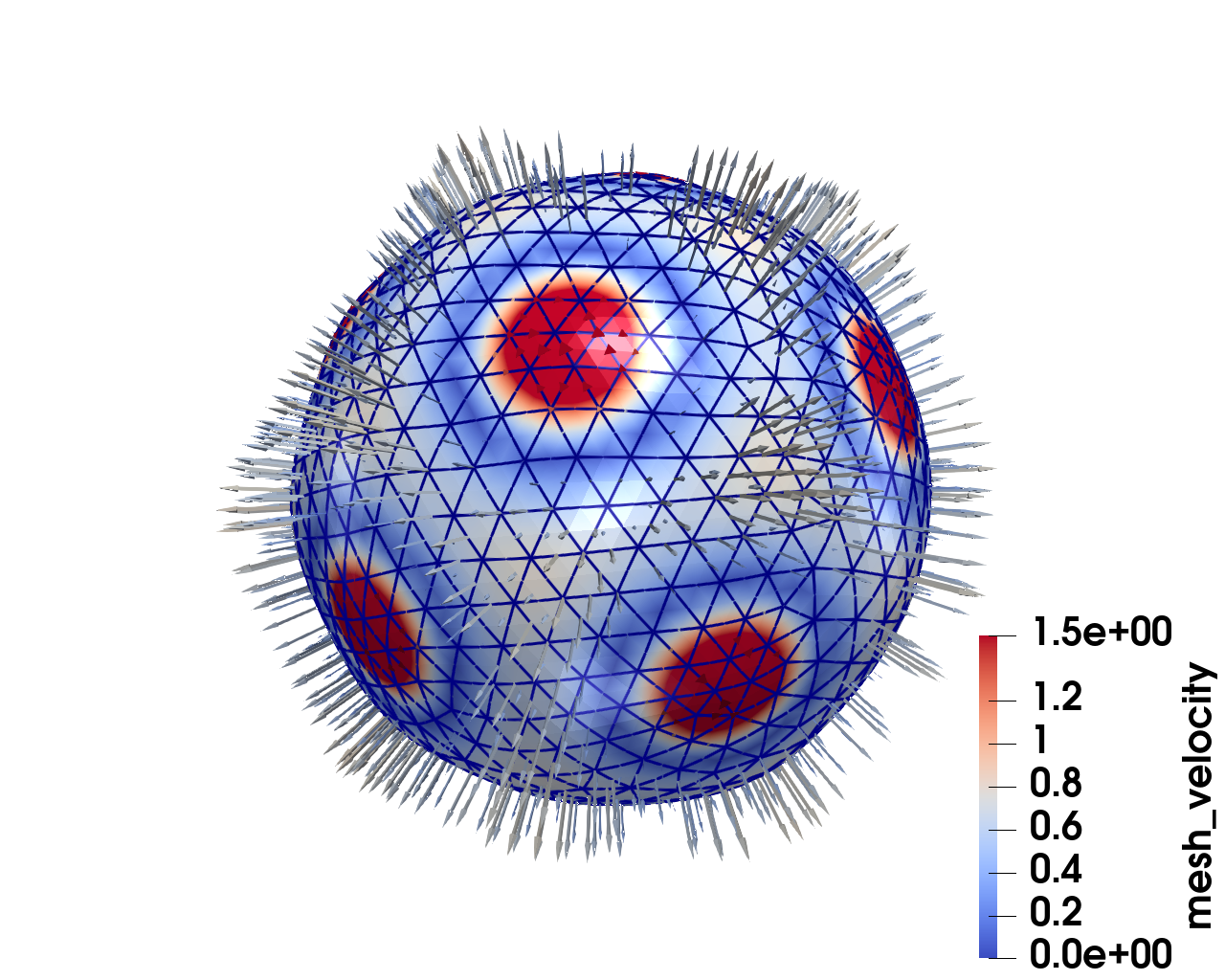} 
        \caption{$t=0.067$}
    \end{subfigure}
    \vspace{-10pt}
    \begin{subfigure}[b]{0.4\textwidth}
        \centering
        \includegraphics[width=\textwidth]{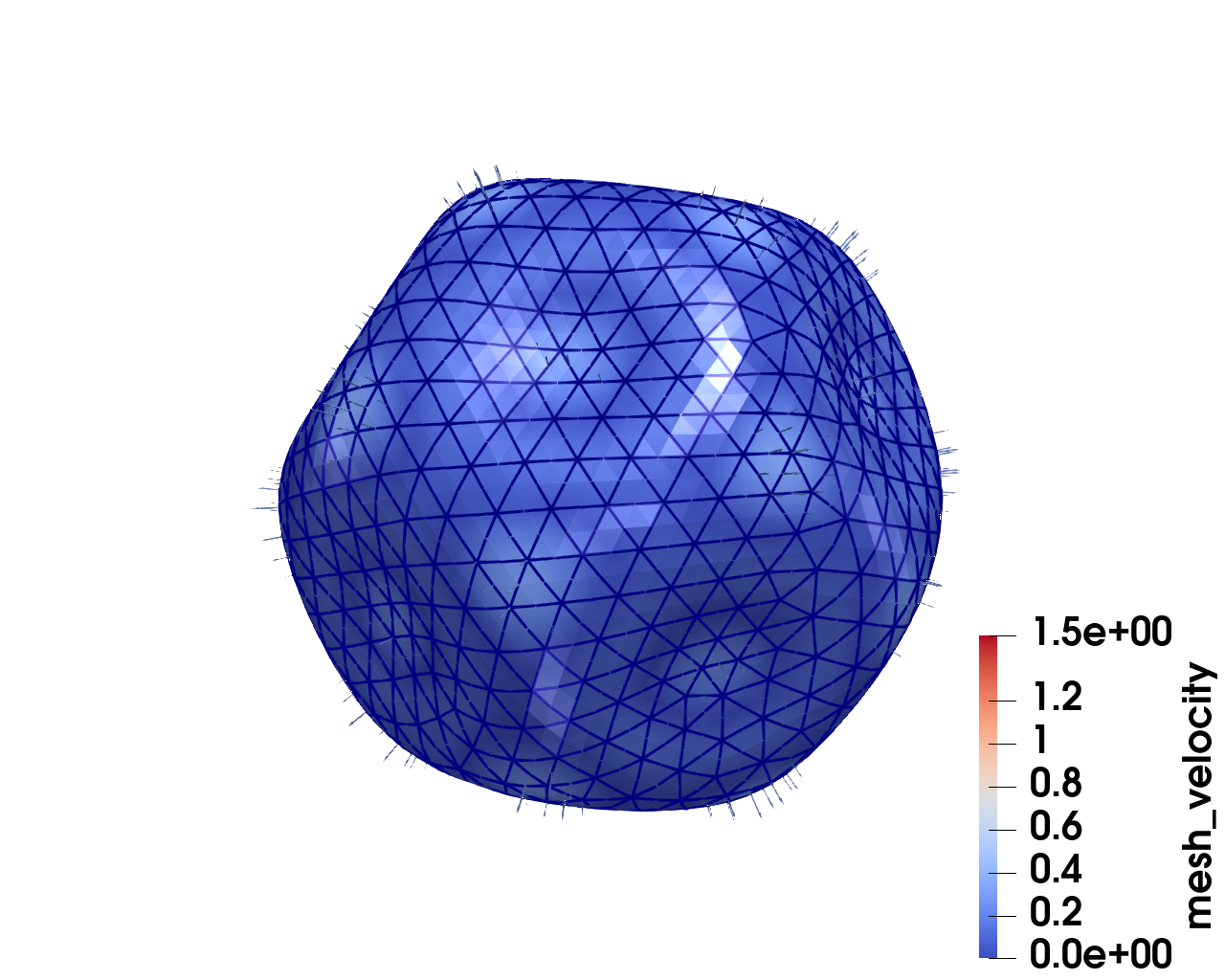} 
        \caption{$t=0.113$}
    \end{subfigure}
    \hspace{10pt} 
    \begin{subfigure}[b]{0.4\textwidth}
        \centering
        \includegraphics[width=\textwidth]{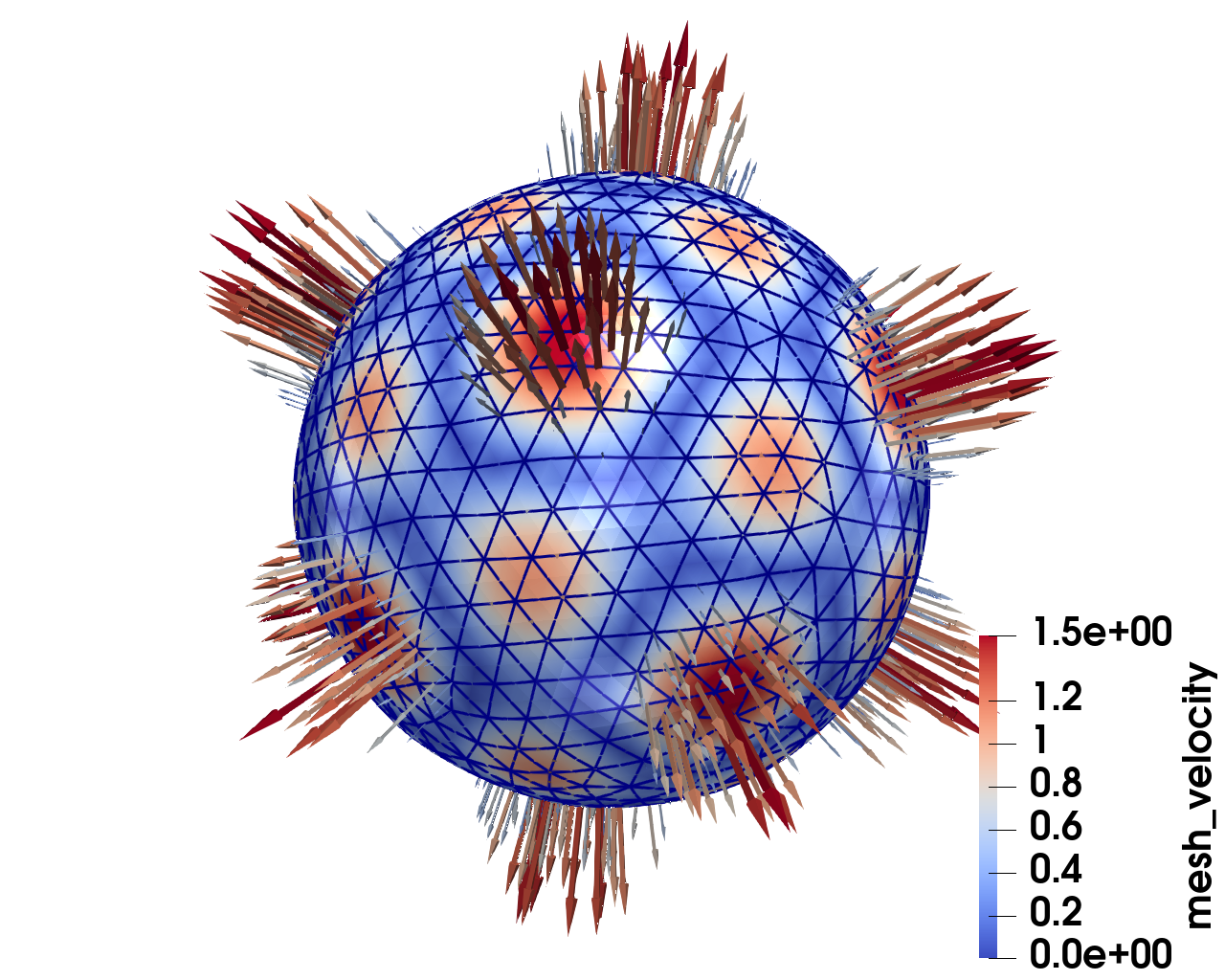} 
        \caption{$t=0.169$}
    \end{subfigure}
    % \vspace{-5pt}
\caption{Mesh velocity fields at different time instants. } 
\label{fig:Popcorn3D}
\end{figure}

\section{Conclusion}
\label{sec:conclusion}

In this work, we develop an energy-stable ALE--FEM based on compensated tangential  velocity 
for incompressible Navier--Stokes free-boundary flows driven by surface tension. 
The surface mesh velocity is decomposed into its normal and tangential components: 
the normal component follows the fluid motion, while the tangential component is determined by the CTV formulation and is used to improve the quality of the surface mesh. 
In particular, the fully discrete CTV formulation consistently recovers its semi-discrete counterpart as $\Delta t\to0$. 
The parameter $\alpha$, which controls the relaxation rate of the artificial tangential mesh redistribution, is chosen according to the viscous--capillary scaling as $\alpha=\gamma_0/\mu$.
The resulting surface mesh velocity is then extended into the bulk by the MDR extension.
We establish the well-posedness of the fully discrete scheme and energy-dissipation estimates for both the semi-discrete and fully discrete formulations. 

A natural future direction is to establish rigorous convergence results and error estimates for ALE finite element methods applied to free-boundary and two-phase problems, including Navier--Stokes flows and dewetting problems.

\section*{Acknowledgments}
The authors would like to thank Prof. Buyang Li for his constructive comments and suggestions that have significantly improved the quality of this paper.

\bibliographystyle{abbrv}
\bibliography{ref}

\end{document}